\documentclass[11pt]{amsart}

\usepackage[margin=0.9in]{geometry}

\usepackage{amsmath, amssymb, amsbsy, amscd, mathtools}
\usepackage{amsfonts}
\usepackage{mathrsfs, euscript}

\usepackage{graphicx}
\usepackage{tikz, tikz-cd}
\usepackage{xy}
\usepackage{epsf, overpic, psfrag}

\usepackage{times}
\usepackage{array}
\usepackage{caption, subcaption}

\usepackage{enumitem, bbm, stmaryrd}
\usepackage[dvipsnames]{xcolor}

\usepackage{hyperref}
\usepackage{thmtools}
\usepackage{todonotes}

\usepackage{mathrsfs}

\newtheorem{theorem}{Theorem}[section]
\newtheorem{lemma}[theorem]{Lemma}
\newtheorem{corollary}[theorem]{Corollary}
\newtheorem{proposition}[theorem]{Proposition}

\newtheorem{assumption}[theorem]{Assumption}
\newtheorem{definition}[theorem]{Definition}

\newtheorem{remark}[theorem]{Remark}
\newtheorem{notation}[theorem]{Notation}

\makeatletter
\@ifundefined{st}{\newcommand{\st}{\check}}{\renewcommand{\st}{\check}}
\@ifundefined{notation}{}{}
\@ifundefined{convention}{\newtheorem{convention}[theorem]{Convention}}{}
\makeatother

\makeatletter
\renewcommand*{\theHtheorem}{\thesection.\arabic{theorem}}
\@for\thmname:=lemma,corollary,proposition,conjecture,question,assumption,definition,example,remark,notation,convention\do{%
  \expandafter\def\csname theH\thmname\endcsname{\theHtheorem}%
}
\makeatother

\hypersetup{hidelinks}
\newcommand{\dsph}{d_{\mathrm{sph}}}

\newcommand{\D}[2]{\partial_{#1}^{#2}}

\newcommand{\Z}{\mathbb{Z}}
\newcommand{\R}{\mathbb{R}}
\newcommand{\C}{\mathbb{C}}

\newcommand{\F}{\mathbb{F}}

\newcommand{\frm}{\mathfrak{m}}

\newcommand{\frgl}{\mathfrak{gl}}

\newcommand{\cY}{\mathcal{Y}}
\newcommand{\M}{\mathcal{M}}

\newcommand{\Conf}{\operatorname{Conf}}
\newcommand{\Kh}{\operatorname{Kh}_{\mathrm{symp}}}
\newcommand{\rank}{\operatorname{rank}}
\newcommand{\Sym}{\operatorname{Sym}}

\newcommand{\ul}{\underline}

\newcommand{\crit}{\operatorname{Crit}}
\newcommand{\cP}{\mathsf}

\newcommand{\ind}{\operatorname{ind}}     % Morse index
\newcommand{\grad}{\operatorname{grad}}   % Gradient

\newcommand{\descending}{\mathscr{D}}
\newcommand{\ascending}{\mathscr{A}}

\newcommand{\Ev}{\mathrm{Ev}}
\newcommand{\Dinv}{\D{s}{s,\mathrm{inv}}}   % Invariant component
\newcommand{\Dnon}{\D{s}{s,\mathrm{non}}}   % Non-invariant component

\makeatletter
\renewcommand{\paragraph}{\@startsection{paragraph}{4}%
  \z@\z@{-\fontdimen2\font}%
  {\normalfont\itshape}}
\makeatother

\hypersetup{hidelinks}

\title[Localization of Floer homology for $\Z_p$-actions]
{Localization of Lagrangian Floer homology for $\Z_p$-actions}

\author{Erkao Bao}
\address{School of Mathematics, University of Minnesota, Minneapolis, MN 55455}
\email{bao@umn.edu}
\urladdr{https://www-users.cse.umn.edu/~bao/}

\author{Tyler Lawson}
\address{School of Mathematics, University of Minnesota, Minneapolis, MN 55455}
\email{tlawson@math.umn.edu}
\urladdr{https://www-users.cse.umn.edu/~tlawson/}

\author{Cheuk Yu Mak}
\address{School of Mathematical and Physical Sciences, University of Sheffield, Western Bank, Sheffield, S10 2TN}
\email{c.mak@sheffield.ac.uk}
\urladdr{https://sites.google.com/view/cheuk-yu-mak/home}

\date{August 2026}

\begin{document}

\begin{abstract}
For an odd prime $p$, we prove a Smith inequality over $\F_p$ for exact
Lagrangian Floer homology under a symplectic $\Z_p$-action, assuming an
equivariant stable normal trivialization compatible with the Lagrangians
and the required relative Pin data.  As an application, we obtain a rank
inequality between the symplectic Khovanov homology of a $p$-periodic link
and that of its quotient.  These results extend the involutive localization
of Seidel and Smith.

The main construction is a stabilized chain complex.  At every invariant
intersection point we adjoin the Morse complex of a finite-dimensional
normal spectral sphere, whose generators lie in free $\Z_p$-orbits.
The differential counts holomorphic strips coupled to sphere
half-trajectories by a one-sided weighted asymptotic evaluation.  We prove
transversality, a uniform spectral cutoff, and the compactification and
gluing statements needed for the chain identity.  The stabilized complex
is quasi-isomorphic to the ordinary Floer complex and contains the complex
of the fixed data as a subcomplex with free quotient, so the algebraic
Smith inequality applies.
\end{abstract}

\maketitle

\setcounter{tocdepth}{2}
\tableofcontents

\section{Introduction}
%Throughout, $p$ is an odd prime and $G=\Z_p$.   All ranks below are total ranks;
%no absolute integer grading is assumed.

\subsection{Background and main results}
Throughout, $p$ is an odd prime and $G=\Z_p$.  Classical Smith theory compares the homology of a space with that of its
fixed locus under a prime-order action
\cite{smith1938transformations,floyd1952on}.  Its chain-level mechanism is
the existence of a trivial-action subcomplex with free quotient over
$\F_p[G]$.  Seidel and Smith \cite{seidel2010localization} established a
Lagrangian Floer analogue for symplectic involutions, under a stable normal
trivialization compatible with the Lagrangian pair.
The result of Seidel and Smith and its descendants have numerous applications: to symplectic
Khovanov homology \cite{seidel2010localization,hendricks2024symplectic},
to the Heegaard Floer homology of branched covers and of symmetric knots
\cite{hendricks2012rank,hendricks2015localization,hendricks2016flexible,
large2019equivariantfloertheorydouble,boyle2022rank,hendricks2022rank,
parikh2024localization,hendricks2025note}, and to the Floer cohomology of
the square of a symplectomorphism \cite{hendricks2017spectral, seidel2015equivariant}; see also
\cite{lipshitz2016noncommutative} for an algebraic counterpart.
Many of these applications are expected to have $\Z_p$-analogues for odd
primes $p$.
In \cite{SZ21}, Shelukhin and Zhou proved a $\Z_p$-localization result for the Floer cohomology of the $p^{th}$-fold iterate of a symplectomorphism using the equivariant pants product \cite{seidel2015equivariant}, \cite{Wilk23} and local coproduct, with striking applications in \cite{SheA, ShelB, cineli2026closedorbitsdynamicallyconvex}, etc. The quantitative aspect was later generalized by Bai and Xu \cite{bai2026integralhamiltonianfloertheory} for general symplectic manifolds, with many applications in \cite{bai2026quantumsteenrodpowershamiltonian}.
However, a $\Z_p$-version of the Seidel--Smith localization theorem for Lagrangian Floer homology is still missing because the previous methods cannot be directly applied.  This is what we prove in this paper.

\begin{theorem}[Floer localization; Theorem~\ref{thm:floer-smith}]
\label{thm:intro-floer}
Let $(M,\omega=d\theta)$ and $L_0\pitchfork L_1$ satisfy the exactness,
convexity and admissibility conditions of
Subsection~\ref{subsec:lagrangian-floer-setup}.  Suppose a symplectic
$\Z_p$-action preserves the pair, and that the equivariant stable normal
trivialization and relative Pin hypotheses of
Assumptions~\ref{ass:floer-normal-triviality} and~\ref{ass:floer-pin} hold.
Then
\begin{equation}\label{eq:intro-floer-smith}
 \dim_{\F_p}HF(L_0,L_1;\F_p)
 \ge \dim_{\F_p}HF(L_0^G,L_1^G;\F_p),
\end{equation}
where the right side is taken with the orientation system of the relative
Pin structures that the fixed pair inherits from $L_0,L_1$
(Lemma~\ref{lem:floer-fixed-strips}).
\end{theorem}

The main novelty of our work is a framework for localization in Lagrangian
Floer theory, built around a new chain complex, the \emph{stabilized
Floer complex}.\footnote{It is not the Floer complex of a stabilization of
$(M,L_0,L_1)$.}  Its differential records asymptotic data of the moduli
spaces of strips of all dimensions, yet the complex remains
quasi-isomorphic to the ordinary Floer complex, so that the classical
Smith inequality applies to it.  One may think of it as a resolution of
the ordinary Floer complex; it is in the spirit of Floer homotopy theory,
but far less technical, and it retains just enough information for
equivariant localization.  It is described in
Subsection~\ref{subsec:intro-complex} below.
Our framework also works for the case $p = 2$, but the presentation is slightly different due to the representation theory of $\Z_p$.
For this reason, we restrict to $p$ odd.
In future work, we will apply our framework to more general setting where Assumption~\ref{ass:floer-normal-triviality} of the normal bundle is not satisfied.

%The normal hypothesis is componentwise: the nontrivial isotypic summands
%of the fixed-locus normal bundle are stably trivial as complex bundles,
%and the Lagrangian normal subbundles admit homotopies, in one ambient
%trivialization, to the zero section and the fiber of $T^*\R^m$.
%The hypotheses and their use in the geometric preparation are stated in
%Subsections~\ref{subsec:floer-suspension} and~\ref{subsec:lagrangian-preparation}.

We expect that the results in
\cite{hendricks2012rank,hendricks2015localization,hendricks2017spectral,
hendricks2024symplectic}  generalize directly to $\Z_p$ once
Assumptions~\ref{ass:floer-normal-triviality} and~\ref{ass:floer-pin} are
verified in each setting; we do not pursue it and leave it as a separate work.
As an example, we prove the following result about symplectic Khovanov homology
$\Kh$ \cite{seidel2006link,manolescu2006nilpotent}, which generalizes
\cite[Corollary~3]{seidel2010localization} from $p=2$ to odd primes;
for the combinatorial counterparts see
\cite{zhang2018rank,stoffregen2024localization,borodzik2021khovanov}.
\begin{theorem}[Periodic links; Theorem~\ref{thm:application-khovanov}]
\label{thm:intro-khovanov}
For a $p$-periodic link $K$ with quotient $\bar K$,
\[
 \rank_{\F_p}\Kh(K)\ge\rank_{\F_p}\Kh(\bar K).
\]
\end{theorem}
The cyclic action is defined by weighting the block coordinates of the
nilpotent slice.  Its fixed fiber and fixed matching Lagrangians identify
with the data of the quotient link.  The equivariant exact sequences and
relative normal structures are checked in Section~\ref{sec:applications}.

\subsection{The stabilized complex}
\label{subsec:intro-complex}

{\it The algebraic target.}
Smith's inequality is, at bottom, a statement about chain complexes with a
group action, and the whole construction of this paper is designed to
produce a complex to which the following lemma applies.

\begin{lemma}[Algebraic Smith inequality; Lemma~\ref{lem:algebraic-smith}]
\label{lem:intro-algebraic-smith}
Let $k=\F_p$, $G=\Z_p$ and $R=k[G]$.  Let $C$ be a finitely generated
$R$-module with an $R$-linear endomorphism $\partial$ satisfying
$\partial^2=0$, and let $C'\subset C$ be a $\partial$-invariant submodule on
which $G$ acts trivially and such that $C/C'$ is a free $R$-module.  Then
\[
    \dim_kH(C')\;\le\;\dim_kH(C),\qquad H(C):=\ker\partial/\operatorname{im}\partial .
\]
\end{lemma}

For the cellular chains of a finite $G$-CW complex $X$ with $C'=C(X^G)$ this
is Smith's theorem, and the proof, given in
Subsection~\ref{subsec:floer-smith}, uses nothing but the exact triangles of
two short exact sequences.  No grading is needed, which is why the lemma is
stated for differential modules; this is what allows us to work with
ungraded Floer complexes. 
So the Floer inequality
\eqref{eq:intro-floer-smith} follows once one has a differential
$\F_p[G]$-module $C$ with $H(C)\cong HF(L_0,L_1;\F_p)$ and a subcomplex $C'$
with trivial action, free quotient, and $H(C')\cong HF(L_0^G,L_1^G;\F_p)$.

{\it The obstruction.}
Call an intersection point \emph{invariant} if it lies in $M^G$ and
\emph{outside} otherwise; since $p$ is prime, an outside point has trivial
stabilizer and lies in a free orbit.  The ordinary Floer complex is a
differential $\F_p[G]$-module, and the invariant intersection points
$\cP{p}\in(L_0\cap L_1)^G$ span a submodule $C'$ on which $G$ acts
trivially with free quotient, so the lemma would apply to it directly if
$C'$ were $\partial$-invariant.  It is not:
a rigid strip
%\[
%    \cP{p}\in(L_0\cap L_1)^G\ \longrightarrow\ \cP{y}\notin(L_0\cap L_1)^G
%\]
may leave an invariant intersection point and end at an outside one, and
then $\partial\cP{p}$ has a component outside the fixed span.  Our goal is to
modify the complex near each invariant intersection point so that the
modified complex contains the fixed Floer complex as a subcomplex with free
quotient, without changing the homology.

{\it The idea.}
Near an invariant intersection point $\cP{p}$ a non-fixed strip is detected
by its leading normal asymptotic mode: the direction in which it leaves the
fixed locus.  There are infinitely many possible normal modes, so we first
truncate to finitely many.  The resulting finite-dimensional space of
leading coefficients is then normalized to a sphere, and this sphere
records the direction of departure.  Replacing the generator $\cP{p}$ by the
cells of that sphere, on which $G$ acts freely, turns the directions in
which a strip can escape the fixed locus into free orbits, while the top
homology class of the sphere continues to represent $\cP{p}$. Concretely,
at each invariant intersection $\cP{p}$ we attach
the Morse complex of a single finite-dimensional spectral sphere,
\[
 S_N=S(E_N^-)
\]
of the truncated negative normal eigenspace $E_N^-$ of dimension $N$, together with a copy of the invariant intersection $\cP{p}$, playing the role of filling the spectral sphere by a ball.
Every selected eigenspace is a nontrivial real rotation plane, and the
action on the sphere is free.  The sphere Morse function and metric are
lifted from its generalized lens-space quotient, so their critical points
occur in free $G$-orbits.

A non-fixed strip with invariant negative end determines a vector of
selected asymptotic coefficients.  Translation acts on the different
coefficients with different exponential weights.  Dividing by this
weighted action, rather than ordinary radial scaling, gives the evaluation
$\Ev^-_{\cP{p},N}$ into $S_N$.  An \emph{infinitesimal cascade}
is a strip coupled at this evaluation to a negative-gradient
half-trajectory on the spectral sphere.  Equivalently, its evaluation
belongs to the descending cell of a sphere critical point; this defines
the incidence space \eqref{eq:floer-incidence}.

The stabilized module has three summands,
\[
 C=C^o\oplus C^s\oplus C^i,
\]
where $C^o$ is generated by outside intersections, $C^s$ is the sum of
the local sphere complexes, and $C^i$ is generated by the
invariant intersections.  We deliberately evaluate only at the negative end.
A strip leaving $\cP{p}$ is constrained by the incidence of its evaluation
with a cell of the sphere; a strip arriving at an invariant intersection
carries no asymptotic constraint at its positive end and contributes to
the top cycle of that intersection's sphere, the sum of all top cells.
%This asymmetry is what allows the invariant generator to be represented by
%the top cycle, while the directions of departure become free orbits.  
The local augmentation maps $C^s$ to $C^i$,
and no differential leaves $C^i$ for its complement.  The restriction to
$C^i$ is the fixed-pair Floer differential, while the quotient is free
over $\Z[G]$.

The map from the ordinary Floer complex sends an outside generator to
itself and an invariant generator to the top cycle of its sphere.
Theorem~\ref{thm:floer-F-quasi-isomorphism} proves that this is a
quasi-isomorphism over $\Z$ and over $\F_p$.  Its proof uses the action
filtration and the reduced homology of the local sphere.  Lemma~\ref{lem:intro-algebraic-smith} then gives
\eqref{eq:intro-floer-smith}.

The whole paper is summarized by the following diagram, in which the top
row is the geometric preparation, the middle row the two properties of the
stabilized complex, and the bottom row the algebra.
\[
\begin{tikzcd}[column sep=2.6em, row sep=large, ampersand replacement=\&]
\text{stably trivial normal bundle}
  \arrow[r, "\text{prepare}"]
\& \text{constant normal model}
  \arrow[r, "\text{cut off}"]
\& S_N
  \arrow[d, "\text{Morse complex}"]
\\
CF(L_0,L_1)
\& CF_{\mathrm{st}}(\widehat L_0,\widehat L_1)
  \arrow[l, "\simeq"']
  \arrow[d, twoheadrightarrow]
\& CF(L_0^G,L_1^G)
  \arrow[l, hook']
\\
\& \text{free }\Z[G]\text{-module}
  \arrow[r, "\text{algebraic Smith}"]
\& \text{Smith inequality}
\end{tikzcd}
\]

Of the ingredients, the geometric preparation of the normal data follows
the strategy of Seidel and Smith, and once the finite-dimensional evaluation
is in hand the Morse theory of the sphere is standard.  The new features for
odd $p$ are the finite spectral stabilization, the weighted evaluation, and
the uniform cutoff of eigenvalues that makes the finite truncation available; the substantive analytic issue
is that the evaluation survives Floer breaking and is compatible with
gluing.

\subsection{The analytic ingredients}
\label{subsec:intro-model}\label{subsec:intro-finite}
%Two obstacles stand between the idea above and a chain complex.

The main analytic ingredient to carry out the idea is a uniform bound on
the leading normal mode of a non-fixed strip.

%{\it The normal geometry is only stably trivial.}
%The normal bundle of the fixed locus is trivial only after adding a
%representation, and the Lagrangians meet it in subbundles that vary along
%the fixed locus.  A suspension by a cotangent lift makes the normal bundle
%an honest constant representation without changing the Floer complex, and
%equivariant Hamiltonian isotopies then straighten each Lagrangian, on a
%adapted tube containing all fixed strips, to a product of its fixed part
%with a constant Lagrangian subspace.  A split almost complex structure with
%one constant normal factor completes the normal form.  Along a fixed strip
%the normal operator is then a constant-coefficient operator, invertible,
%so fixed-locus regularity reduces to the usual transversality for the fixed
%pair, and non-fixed strips are made regular by invariant perturbations
%supported in the free region.  We also prepare a constant almost complex
%structure and linear boundary conditions on a chart at every intersection
%point, so that the strip equation is exactly linear near each of them.

{\it There are infinitely many normal modes.}
After the preparation the normal asymptotic operator at every invariant
intersection point is the same, with spectrum $-\pi/2+\pi\Z$: a non-fixed
strip leaving the fixed locus does so along a superposition of modes
$e^{-\lambda s}e^{-\sqrt{-1}\lambda t}c$, $\lambda\in-\pi/2+\pi\Z$, $\lambda<0$,
and there are infinitely many of them.
Theorem~\ref{ass:floer-uniform-leading-bound} shows that \emph{all}
non-fixed strips with invariant negative end have a nonzero coefficient
among their first $r(J)$ modes, uniformly.  The bound $r(J)$ depends on the adapted
almost complex structure only, and it applies to the moduli spaces of
higher index difference from which the rigid and one-dimensional incidence
spaces are cut out.  Choosing $r\ge r(J)$ afterwards, every strip that can
contribute to the differential is seen by the first $r$ modes.

\subsection{Related work}
Localization for involutions was developed in
\cite{seidel2010localization,hendricks2015localization,hendricks2016flexible,
hendricks2020simplicial,large2019equivariantfloertheorydouble, bai2026proofarnoldgiventalconjecture}. 
Lidman and Manolescu \cite{lidman2018floer} proved a Smith inequality for
monopole Floer homology of regular $p$-fold coverings of rational homology
spheres using the Seiberg--Witten Floer stable homotopy type.
Stoffregen and Zhang \cite{stoffregen2024localization} and Borodzik,
Politarczyk and Silvero \cite{borodzik2021khovanov} obtained localization
consequences from equivariant Khovanov homotopy types, for all primes $p$;
see also \cite{zhang2018rank}.
In fixed point Floer cohomology, a Smith inequality for the square of a
symplectomorphism was also proved by Seidel \cite{seidel2015equivariant},
via an equivariant pair-of-pants product, which was later generalized by Shelukhin-Zhou \cite{SZ21}.  The stabilization
and cascade viewpoint is related to
\cite{bao2024morse,bao2024equivariantmorsehomologyreflection}.
Our analytic construction uses the separated-transition gluing method
of \cite{hutchings2009gluing}; the auxiliary sphere compactifications use
ordinary Morse trajectory spaces \cite{wehrheim2012smooth}.

\subsection{Organization}
Section~\ref{sec:lagrangian-floer} gives the Floer construction, including
normal preparation, ambient and incidence transversality, the spectral
cutoff, orientations, the differential and its square, the comparison
map, and the algebraic Smith argument.  Section~\ref{sec:applications}
applies the result to periodic links.  Appendix~\ref{app:floer-orientations}
identifies the orientations induced on the fixed pair with those of its
induced relative Pin structures.  Appendix~\ref{app:floer-compactification}
proves compactness and the exhaustive low-dimensional boundary list.
Appendix~\ref{app:floer-gluing} proves the gluing collars, coefficient
estimates, recovery, and the boundary signs of strip, sphere and
incidence spaces.   A notation guide is at the end.

\subsection{Acknowledgement}
E.B. is partially supported by the National Science Foundation under grant DMS-2404529. He thanks Kristen Hendricks, Ko Honda, and Ciprian Manolescu for helpful discussions.
T.L. is partially supported by the National Science Foundation under grant DMS-2208062. 
C.Y.M. is partially supported by the Royal Society University Research Fellowship. He thanks Kristen Hendricks, Noah Porcelli and Ivan Smith for helpful discussions.

\subsection{Declaration of AI use}
We used AI language models Claude Fable 5.1 and Chatgpt 5.6 in preparing this paper. 
Besides editorial help, it was used to draft the proofs of several routine auxiliary results.
Each of these was checked and edited by the authors, who take responsibility for their correctness.
The main ideas, constructions, and arguments of the paper are the authors' own.

\section{Lagrangian Floer theory}
\label{sec:lagrangian-floer}

\paragraph{Roadmap.}
The construction has four stages.
\begin{enumerate}[label=\textup{Stage \Roman*.},leftmargin=*]
    \item \emph{Normal preparation}
    (Subsections~\ref{subsec:floer-suspension}--\ref{subsec:lagrangian-preparation}).
    Following \cite{seidel2010localization}, the normal geometry near the fixed locus is made constant: a suspension
    trivializes the normal bundle, Hamiltonian isotopies straighten the
    Lagrangians to products near $M^G$, and a split almost complex structure
    has one constant normal factor. 
    (Subsection~\ref{subsec:floer-evaluation}).  A non-fixed strip leaving
    an invariant intersection point is recorded by its first finitely many
    normal modes, normalized to a point of a finite-dimensional sphere; a
    uniform bound shows that finitely many modes suffice for every strip
    that matters.
    \item \emph{Stabilization}
    (Subsections~\ref{subsec:floer-stabilized}--\ref{subsec:floer-dsquared}).
    The Morse complex of the sphere is attached to every invariant
    generator, incidence of the evaluation with its cells defines the new
    differential, and compactification and gluing give $\partial^2=0$.
    \item \emph{Comparison and algebra}
    (Subsections~\ref{subsec:floer-localization}--\ref{subsec:floer-smith}).
    The stabilized complex is quasi-isomorphic to the ordinary one and
    contains the fixed Floer complex as a subcomplex with free quotient,
    which is what the algebraic Smith inequality requires.
\end{enumerate}

\paragraph{Order of choices.}
The data are fixed in an order that avoids circularity, and the order is
the order of the text:
\[
    \text{normal preparation}\ \longrightarrow\ J\ \longrightarrow\
    \text{uniform leading-mode bound}\ \longrightarrow\ N\ \longrightarrow\
    \text{sphere data}.
\]
The adapted almost complex structure is fixed first, together with the
regularity of the ambient strip spaces.  The leading-mode bound is then
proved for it, uniformly in the cutoff, and only after the bound is
established is $N$ chosen.  The Morse data on the sphere are chosen last and
generically, without disturbing anything chosen before.  The proofs
concerning orientations are in Appendix~\ref{app:floer-orientations}, and
the compactification and gluing proofs are in
Appendices~\ref{app:floer-compactification} and~\ref{app:floer-gluing}.

\subsection{Floer setup}
\label{subsec:floer-setup}
\label{subsec:lagrangian-floer-setup}

Let $(M,\omega=d\theta)$ be exact and convex at infinity, and let the odd cyclic
group $G=\Z_p$ act by $\theta$-preserving symplectomorphisms (an action by exact
symplectomorphisms reduces to this case by averaging $\theta$ over $G$).  Let $L_0,L_1\subset M$ be exact
$G$-invariant Lagrangians, compact or conical at infinity, intersecting
transversally and only in a compact region, so that $L_0\cap L_1$ is a
finite $G$-set; in particular there is no intersection along the
cylindrical ends.  Throughout, an
intersection point is called \emph{invariant} if it lies in $M^G$, that
is, if its stabilizer is all of $G$, and \emph{outside} otherwise; since $p$
is prime, an outside point has trivial stabilizer and lies in a free
$G$-orbit.  A strip or a configuration is \emph{fixed} if it lies in $M^G$
and \emph{non-fixed} otherwise, and the \emph{fixed pair} is
$(L_0^G,L_1^G)$ in $M^G$.

Neither the fixed locus nor the fixed pair needs a hypothesis of its own:
$M^G$ is a smooth submanifold of $M$, with components of possibly
different dimensions, and it is symplectic, exact for $\theta|_{M^G}$,
because at a fixed point the invariant part of the tangent space is
$\omega$-orthogonal to the sum of the nontrivial isotypic components; and
$L_i^G=L_i\cap M^G$ is an exact Lagrangian submanifold of $M^G$, isotropic
since $\omega$ vanishes on $L_i$ and of half dimension since
$T_{\cP{p}}M\cong T_{\cP{p}}L_i\oplus(T_{\cP{p}}L_i)^*$ as representations,
with $L_0^G$ and $L_1^G$ transverse in $M^G$ at the invariant intersection
points.

To define the complex over $\Z$ we need orientation data, fixed once and
for all:

\begin{assumption}[Relative Pin structures]\label{ass:floer-pin}
There is a $G$-equivariant real vector bundle $E\to M$, possibly zero,
such that $TL_i\oplus E|_{L_i}$ admits a Pin structure for $i=0,1$.
\end{assumption}

Nothing equivariant is asked of the Pin structures: since $p$ is odd they
can always be chosen $G$-equivariant (Lemma~\ref{lem:floer-invariant-pin}).  Every orientation
convention is fixed in Subsection~\ref{subsec:floer-orientations}.  We take the
Hamiltonian to be zero throughout: the generators of the Floer complex are the
intersection points, a point being nondegenerate exactly because the
intersection is transverse.  The perturbation datum is a $t$-dependent family
$J=(J_t)_{t\in[0,1]}$ of $G$-invariant $\omega$-compatible almost complex
structures, equal to an admissible contact-type structure at infinity and
subject to the following local requirements imposed in
Subsection~\ref{subsec:lagrangian-preparation}: $J_t$ is independent of $t$
and constant on a neighborhood of every intersection point, equal there to the
standard structure in the charts of Subsection~\ref{subsec:lagrangian-preparation}; and $J_t$ is
generic in the region where the $G$-action is free, which is where
transversality of the moduli spaces is achieved.  A Floer strip is a
finite-energy solution
\begin{equation}\label{eq:lagrangian-floer-strip}
\begin{gathered}
    u\colon\R\times[0,1]\longrightarrow M,
    \qquad u(s,0)\in L_0,
    \quad u(s,1)\in L_1,\\
    \partial_su+J_t(u)\,\partial_tu=0,
    \qquad
    \lim_{s\to\pm\infty}u(s,t)=y_\pm\in L_0\cap L_1.
\end{gathered}
\end{equation}
Translation acts on nonconstant strips, and $\M(\cP{x},\cP{y})$ denotes
the quotient of the strips from $\cP{x}$ to $\cP{y}$ by this action; a strip
from $\cP{x}$ to $\cP{y}$ has $y_-=\cP{x}$ at its negative end $s\to-\infty$
and $y_+=\cP{y}$ at its positive end.  In the
exact convex setting there is no bubbling, and the relevant compactifications
consist of broken strips; compare \cite{seidel2010localization}.  Since we
work over $\Z$, the moduli spaces must be oriented; this requires
relative Pin structures on $L_0$ and $L_1$, Assumption~\ref{ass:floer-pin};
all orientation conventions of this section are fixed in
Subsection~\ref{subsec:floer-orientations}.

\subsubsection*{Indices}

The Floer complex $CF(L_0,L_1)$ is the free $\Z$-module on $L_0\cap L_1$
with the differential defined below; $HF(L_0,L_1)$ is its homology.  The
Fredholm index of a strip $u$ from $\cP{x}$ to $\cP{y}$ is its
Maslov index, which depends on the homotopy class of $u$ among
strips with the given boundary conditions and ends.  Accordingly the space
of strips from $\cP{x}$ to $\cP{y}$ is a disjoint union of components of
possibly different dimensions, and we write $\M^0(\cP{x},\cP{y})$ for the
union of its index-one components; ``rigid'' and ``one-dimensional'' refer
to components throughout.  All strips from $\cP{x}$ to $\cP{y}$ have the
same energy, the difference of the actions of their ends, so the whole space
is compact up to breaking and only finitely many indices occur.

The fixed data of the construction are the fixed parts
$L_0^G,L_1^G\subset M^G$, a transverse pair of exact Lagrangians in the
symplectic manifold $M^G$ with $L_0^G\cap L_1^G$ the set of invariant
intersection points; the stabilized complex will contain their Floer complex
as a subcomplex.  

\subsubsection{The Floer complex}
\label{subsec:floer-complex}

For regular data with the orientation conventions of
Subsection~\ref{subsec:floer-orientations}, the ordinary Floer differential is
\begin{equation}\label{eq:ordinary-floer-differential}
 d_{\mathrm F}\cP{x}
   =\sum_{\cP{y}\in L_0\cap L_1}\#\M^0(\cP{x},\cP{y})\,\cP{y}.
\end{equation}
The compactified index-two strip spaces give $d_{\mathrm F}^2=0$.
We suppress the auxiliary almost complex structure from $CF(L_0,L_1)$;
continuation identifies its homology for different admissible regular
choices \cite{seidel2008fukaya}.  The same convention applies
to the suspended and adapted pairs below.

We use \emph{differential module} to mean a module with a square-zero
endomorphism.  Its homology is $\ker d/\operatorname{im}d$; no absolute
integer grading is assumed.  The local sphere degree and the Fredholm
indices still specify which moduli spaces are counted.  Write
$\theta|_{L_i}=dK_i$ and set
\begin{equation}\label{eq:floer-action}
 \mathcal A(\cP{x})=K_1(\cP{x})-K_0(\cP{x}).
\end{equation}
With the strip orientation $ds\wedge dt$, Stokes' theorem gives
$E(u)=\mathcal A(\cP{x})-\mathcal A(\cP{y})$ for a strip from
$\cP{x}$ to $\cP{y}$.  Every nonconstant strip strictly lowers this action.

\subsection{\texorpdfstring{$G$}{G}-equivariant symplectic transformation and Lagrangian subspaces}

We first record the elementary representation theory
needed; see \cite{serre1977linear}.

\begin{lemma}[Real representations of the cyclic group]\label{lem:real-irreps-of-Zp}
Fix a generator $\sigma$ of $G=\Z_p$ and put $\zeta=e^{2\pi i/p}$.
The irreducible real representations are the trivial line and the rotation
planes $\rho_k=\C$, on which $\sigma z=\zeta^kz$, for
$1\le k\le(p-1)/2$.  If a real representation $U$ has $U^G=0$, then its
unit sphere has a free orientation-preserving $G$-action.
\end{lemma}

\begin{definition}[Group complex structure]\label{def:group-complex-structure}
Let $U$ be a real $G$-representation with $U^G=0$, and let
$U=\bigoplus_{k=1}^{(p-1)/2}U_k$ be its isotypic decomposition, $U_k$ the
sum of the copies of $\rho_k$ in $U$.  Each $U_k$ carries a unique complex
structure $J$ for which $\sigma$ acts as multiplication by $\zeta^k$, that
is,
\[
  \sigma|_{U_k}=\cos(2\pi k/p)+\sin(2\pi k/p)\,J .
\]
On a copy of $\rho_k=\C$ it is the given complex structure, and on each
plane of $U_k$ that $\sigma$ rotates through $2\pi k/p$ it is the rotation
through $\pi/2$ in the same sense.  The direct sum of these is the
\emph{group complex structure} $J_G$ on $U$, and the orientation it
induces is the \emph{complex orientation}.  Both depend on the generator
$\sigma$ through the convention $1\le k\le(p-1)/2$; replacing $\sigma$ by
$\sigma^{-1}$ reverses $J_G$.
\end{definition}

By Schur's lemma, $\operatorname{End}_G(\rho_k)=\C$, acting on $\rho_k$ by
the complex scalars; the complex scalars of $U_k$ are the action of this
algebra, and a $G$-equivariant linear map between representations without
invariant vectors sends each isotypic component to the corresponding one
and commutes with it.  Hence:

\begin{lemma}[Equivariant maps are complex linear]\label{lem:group-complex-structure}
Let $U,U'$ be real $G$-representations without invariant vectors.  Every
$G$-invariant subspace of $U$ is $J_G$-complex, every $G$-equivariant
linear map $U\to U'$ is $J_G$-complex linear, and a $G$-equivariant
isomorphism preserves complex orientations.
\end{lemma}

Let $U$ be a nontrivial irreducible representation of $G$, that is,
$U=\rho_k$ for some $k$, let $m$ be a positive integer, and equip
$W:=T^*(U^{\oplus m})=U^{\oplus m}\oplus(U^{\oplus m})^*$ with the induced
$G$-action and the canonical symplectic form $\omega_W$.  Using $J_G$ we
identify $U$ with $\C$, and $(U^{\oplus m})^*$ with $U^{\oplus m}$ by the
invariant inner product, so that
\begin{align}\label{eq:Wid}
W \simeq \C^m \oplus \C^m,
\end{align}
where the induced action is $\sigma\cdot(z,w)=(\zeta^kz,\zeta^kw)$ and the
symplectic form is
\[
\omega_W=\operatorname{Im}(h_W),\qquad
h_W\bigl((z,w),(z',w')\bigr):=\sqrt{-1}\,\bigl(\langle w',z\rangle-\langle z',w\rangle\bigr),
\qquad \langle a,b\rangle=\sum_ja_j\bar b_j .
\]
Here $h_W$ is a nondegenerate Hermitian form  with respect to $J_G$, antilinear in its first and
linear in its second argument, and $h_W(a,b)=\omega_W(a,J_Gb)+\sqrt{-1}\,\omega_W(a,b)$.

\begin{lemma}[Invariant Lagrangian subspaces]\label{l:eqLag}
Under the identification \eqref{eq:Wid}, every complex linear automorphism
of $\C^m\oplus\C^m$ preserving $h_W$ is a $G$-equivariant symplectic
automorphism of $W$, that is, $U(h_W) \cong U(m,m)\subseteq Sp(W)^G$, and every
$G$-invariant Lagrangian subspace of $W$ is an $m$-dimensional complex
subspace on which $h_W$ vanishes.  

The group $U(h_W)$ is connected and acts transitively on these subspaces, so
the space of $G$-invariant Lagrangian subspaces of $W$ is connected.  
%{\cym For a general orthogonal representation $\R^m$ without invariant vectors, the space of $G$-invariant Lagrangian subspaces of $W$ is a product of $U(m_i,m_i)/GL(m_i,\mathbb{C})$ with respect to the isotypic decomposition of $W$.}
The same holds for $W=T^*\R^m$ with $\R^m$ any orthogonal representation without invariant vectors.
\end{lemma}

\begin{proof}
Both summands of $W=U^{\oplus m}\oplus(U^{\oplus m})^*$ are of type
$\rho_k$, so $\sigma$ acts on $W$ as the $J_G$-scalar $\zeta^k$; hence a
complex linear map of $W$ commutes with $\sigma$, and a $G$-invariant
subspace is complex (Lemma~\ref{lem:group-complex-structure}).
Multiplying $z$ and $w$ by a common unit complex number preserves
$\operatorname{Re}\langle z,w'\rangle$, so $\omega_W$ is $J_G$-invariant
and $h_W=\omega_W(\cdot,J_G\cdot)+\sqrt{-1}\,\omega_W$ is Hermitian.
Therefore a complex linear map preserving $h_W$ preserves
$\omega_W=\operatorname{Im}h_W$, and a complex subspace on which
$\omega_W$ vanishes is $h_W$-isotropic, of complex dimension $m$ if it is
Lagrangian.  (The reverse inclusion $Sp(W)^G\subseteq U(h_W)$ also holds,
equivariant maps being complex linear, but is not needed.)  The matrix of $h_W$ in the coordinates $(z,w)$ is
$\bigl(\begin{smallmatrix}0&\sqrt{-1}\\-\sqrt{-1}&0\end{smallmatrix}\bigr)$,
which squares to the identity and has trace zero, so $h_W$ has signature
$(m,m)$ and $U(h_W)\cong U(m,m)$ is connected.

By Witt's theorem \cite[Chapitre~IX, \S4]{bourbaki1959algebre9}, every
isometry between subspaces of the nondegenerate Hermitian space $(W,h_W)$
extends to an element of $U(h_W)$; a linear isomorphism between two
$m$-dimensional $h_W$-isotropic subspaces is an isometry, so $U(h_W)$ acts
transitively on them.

Finally, if $\R^m=\bigoplus_kU_k^{\oplus m_k}$ is the isotypic
decomposition of an orthogonal representation without invariant vectors,
then $\prod_kSp(T^*(U_k^{\oplus m_k}))^G$ acts on
$T^*\R^m=\bigoplus_kT^*(U_k^{\oplus m_k})$ by equivariant symplectic
automorphisms, and an invariant Lagrangian subspace is the sum of its
intersections with the summands, each Lagrangian in its summand by a
dimension count; so transitivity and connectedness hold componentwise.
\end{proof}

\subsection{Suspension}
\label{subsec:floer-suspension}

The stable normal trivialization is not yet an honest trivialization of
the normal bundle.  We therefore add a cotangent factor whose fixed locus is
unchanged but whose normal representation cancels the stable defect; after
this suspension the normal bundle is a constant $G$-representation, and the
Floer complex is unchanged.  This is the first of three steps converting the
stable normal hypothesis into a constant normal model, the others being
equivariant Hamiltonian straightening and the choice of a split almost
complex structure (see \cite[Section 3e]{seidel2010localization} for the case that $p=2$).

Throughout, $\nu$ denotes the symplectic normal bundle of $M^G$ in $M$
a
symplectic $G$-vector bundle with $\nu^G=0$, and
$\nu_{L_i}\subset\nu|_{L_i^G}$ the normal bundle of $L_i^G$ in $L_i$, a
fiberwise $G$-invariant Lagrangian subbundle.

For a real orthogonal $G$-representation $\R^n$, the cotangent lift
$T^*\R^n=\C^n$ carries the standard symplectic form and complex
structure; the complexified action is unitary, and the zero section $\R^n$
and the fiber $\sqrt{-1}\,\R^n$ are transverse $G$-invariant Lagrangian
subspaces.  This is the stabilizing datum of
\cite{seidel2010localization}, where the involution acts on
$\C^k=T^*\R^k$ by the cotangent lift of $-\operatorname{id}$; for odd $p$
the only difference is which real representations exist, and since every
nontrivial real irreducible representation of $\Z_p$ is a rotation plane
(Lemma~\ref{lem:real-irreps-of-Zp}), $(\R^n)^G=0$ forces $n$ even.

The following assumption is the odd-prime analogue of the stable normal
trivialization of \cite[Definition 18]{seidel2010localization}.

\begin{assumption}[Equivariant stable normal triviality]
\label{ass:floer-normal-triviality}
There are real orthogonal $G$-representations $\R^d$ and $\R^m$, neither
with nonzero invariant vectors, together with a $G$-equivariant
isomorphism of symplectic vector bundles
\begin{align}\label{eq:stably-trivial}
    \nu\oplus\bigl(M^G\times T^*\R^d\bigr)
    \;\cong\;M^G\times T^*\R^m,
\end{align}
equivalently a trivialization of each isotypic component of the left side
as a complex vector bundle.  Moreover, for $i=0,1$ there is a path
$\Lambda_i(t)$, $t\in[0,1]$, of fiberwise $G$-invariant Lagrangian
subbundles of $M^G\times T^*\R^m$ over $L_i^G$ from
\[
    \Lambda_0(0)=\nu_{L_0}\oplus\R^d,\qquad
    \Lambda_1(0)=\nu_{L_1}\oplus\sqrt{-1}\,\R^d
\]
to the constant subspaces
\[
    \Lambda_0(1)=\R^m,\qquad
    \Lambda_1(1)=\sqrt{-1}\,\R^m.
\]
\end{assumption}

\begin{remark}
    The condition \eqref{eq:stably-trivial} means that $[\nu] \in \tilde{K}^0_G(M^G)$ is the zero class in the reduced $G$-equivariant $K$ group, and the condition on $\nu_{L_0}\oplus\R^d$ and $\R^m$ mean that they are in the same connected component in the space of stable $G$-equivariant Lagrangian subbundles.
\end{remark}

\begin{remark}\label{r:rotate}
    Note that by Lemma \ref{l:eqLag}, $\R^m$ is $G$-equivariantly isotopic to $\sqrt{-1}\,\R^m$ via Lagrangian subspaces.
    %To see this, we can apply irreducible representation decomposition and consider only the case $m=2$.
    %In this case, the $G$ action on $(p,q) \in T^*\R^2=\R^2 \times \R^2$ is given by $\xi \cdot (p,q)=(\xi \cdot p, \xi \cdot q)$.
    %It is easy to see that for any $a,b \in \mathbb{R}$,  $L_{a,b}:=\{(p,q)=(ap,bp): p \in \R^2\}$ is a $G$-equivariant Lagrangian subspace of $T^*\R^2$. Therefore, any path from $(0,1)$ to $(1,0)$ will induce a $G$-equivariant Lagrangian isotopy between $\R^2$ and $\sqrt{-1}\,\R^2$.
\end{remark}

The two representations play different roles: $\R^d$ is an auxiliary
stabilizer, absorbed by the suspension below and not seen again, while
$\R^m$ is the model normal representation in which all later normal data
live.  We write
\[
    W:=T^*\R^m=\C^m
\]
for its cotangent lift, with its standard complex structure $J_W$; then
$W^G=0$, and $\R^m$, $\sqrt{-1}\,\R^m$ are its standard Lagrangian
subspaces.  An additional cotangent suspension preserves the assumption,
so we may suppose $m>0$.

\begin{remark}
\label{rem:floer-paths-independent}A natural question is whether the two paths must be compatible.  The two paths are independent pieces of data, over the different bases
$L_0^G$ and $L_1^G$.  No compatibility between them is imposed: over a
common point $x\in L_0^G\cap L_1^G$ the subspaces $\Lambda_0(t)_x$ and
$\Lambda_1(t)_x$ are not required to be transverse, or related in any
other way, at intermediate times $t$.  This is possible because each
Lagrangian is straightened by its own Hamiltonian isotopy
(Lemma~\ref{lem:floer-straightening}); only the endpoints interact. 
%In \cite[Definition 18]{seidel2010localization} the
%two endpoints are the standard real and imaginary Lagrangians, so their
%transversality holds automatically; for odd $p$ the endpoints are arbitrary
%invariant Lagrangians and the condition must be imposed.
\end{remark}

\begin{lemma}[Suspension]\label{lem:floer-suspension}
Set
\[
    \widehat M:=M\times T^*\R^d,
    \qquad
    \widehat L_0:=L_0\times\R^d,
    \qquad
    \widehat L_1:=L_1\times\sqrt{-1}\,\R^d,
\]
with the product symplectic form, the diagonal $G$-action, and the product
almost complex structure $\widehat J_t:=J_t\oplus J_{\mathrm{std}}$, with
$J_{\mathrm{std}}$ the standard complex structure of $T^*\R^d=\C^d$.  Then:
\begin{enumerate}
    \item the data are exact, convex at infinity, and $G$-invariant, and
    $\widehat M^G=M^G$;
    \item $\widehat L_0\cap\widehat L_1=(L_0\cap L_1)\times\{0\}$,
    transversally;
    \item the $T^*\R^d$-component of any finite-energy strip for
    $\widehat J_t$ is holomorphic with boundary on $\R^d,\sqrt{-1}\,\R^d$ and
    asymptotic to $0$ at both ends, hence has zero energy and is constant;
    consequently
    \[
        CF(\widehat L_0,\widehat L_1)=CF(L_0,L_1)
    \]
    as differential modules, the generators and the rigid strips being the
    same, and the constant factor contributing a canonical orientation to
    every determinant line.
\end{enumerate}
\end{lemma}

\begin{corollary}\label{cor:floer-suspension-normal}
The symplectic normal bundle of $\widehat M^G=M^G$ in $\widehat M$ is
$\nu\oplus(M^G\times T^*\R^d)$, trivialized by
Assumption~\ref{ass:floer-normal-triviality}, and the Lagrangian normal
bundle $\nu_{\widehat L_i}$ of $\widehat L_i^G=L_i^G$ in $\widehat L_i$ is
$\nu_{L_0}\oplus\R^d$ for $i=0$ and $\nu_{L_1}\oplus\sqrt{-1}\,\R^d$ for
$i=1$.
\end{corollary}

\begin{proof}[Proof of Lemma~\ref{lem:floer-suspension} and Corollary~\ref{cor:floer-suspension-normal}]
Items (1) and (2), and the corollary, are immediate from the definitions:
$(\R^d)^G=0$ gives $\widehat M^G=M^G\times\{0\}$, the zero section and
the fiber meet transversally in the origin, and transversality of a product
pair is transversality of the factors.  For (3), the equation splits, so the
$T^*\R^d$-component $v$ is $J_{\mathrm{std}}$-holomorphic with the stated
boundary and asymptotics; its energy is the difference of actions of its two
asymptotic intersection points, both the origin, so $E(v)=0$ and $v$ is
constant.  The
strips for $\widehat J_t$ therefore biject with the strips for $J_t$, with
identical linearizations up to the constant nondegenerate summand, which
contributes an index shift independent of the chord and no kernel or
cokernel.
\end{proof}

\subsection{Straightening the Lagrangians}
\label{subsec:floer-straightening-lagrangians}

Choose a $G$-invariant inner product on $W$ and write
$D(W)=\{w\in W:|w|<1\}$.  Fix a $G$-equivariant symplectic tubular neighborhood
\[
    \Xi\colon M^G\times D(W)\longrightarrow \widehat U\subset\widehat M
\]
realizing the trivialization of Corollary~\ref{cor:floer-suspension-normal}, obtained
from the equivariant symplectic neighborhood theorem by averaging.

\begin{lemma}[Equivariant straightening]\label{lem:floer-straightening}
For each $i=0,1$ there is a compactly supported $G$-equivariant Hamiltonian
isotopy $(\phi_i^s)_{s\in[0,1]}$ of $\widehat M$, fixing $\widehat M^G$
pointwise, such that, after shrinking the tube,
\[
    \phi_i^1(\widehat L_i)\cap \widehat U
    =\Xi\bigl(L_i^G\times(\Lambda_i(1)\cap D(W))\bigr).
\]
Moreover, $\phi_0^1(\widehat L_0)$ and $\phi_1^1(\widehat L_1)$ intersect
transversally at every invariant intersection point.
\end{lemma}

The proof assembles two constructions: a \emph{gauge},
which rotates the tangent data of a Lagrangian along its fixed part into a
prescribed constant position, and a \emph{shear}, which removes the
higher-order terms once the tangent data is constant.  Both use the
elementary observation that if a time-dependent Hamiltonian $H_t$ satisfies
$dH_t=0$ at every point of a subset $Z\subset\widehat M$ for every $t$,
then its Hamiltonian vector field vanishes along $Z$ and its isotopy fixes
$Z$ pointwise; and that an invariant Hamiltonian generates an equivariant
isotopy, since $\omega$ is invariant.  Throughout,
$\ell_i:=L_i^G=\widehat L_i\cap\widehat M^G$ denotes the fixed part, a
Lagrangian submanifold of $M^G$, compact or conical at infinity, and $\nu_{\widehat L_i}\subset\ell_i\times W$ the normal bundle of
Corollary~\ref{cor:floer-suspension-normal}, a fiberwise invariant
Lagrangian subbundle, so that
$T\widehat L_i|_{\ell_i}=T\ell_i\oplus\nu_{\widehat L_i}$ inside
$T\widehat M|_{\ell_i}=TM^G|_{\ell_i}\oplus W$.

\begin{lemma}[Equivariant gauge]\label{lem:floer-gauge}
Let $\ell$ be one of the $\ell_i$ and let $\Lambda(t)$, $t\in[0,1]$, be a
smooth path of fiberwise $G$-invariant Lagrangian subbundles of
$\ell\times W$.  Then there is a compactly supported $G$-invariant
Hamiltonian isotopy of $\widehat M$ that fixes $\widehat M^G$ pointwise and
whose time-$t$ linearization at a point $(x,0)$ of $\ell$ is the identity
on $T_xM^G$ and a symplectic map of $W$ carrying $\Lambda(0)_x$ to
$\Lambda(t)_x$.  Applied to $\widehat L_i$, it leaves $\ell_i$ pointwise
fixed and replaces the normal tangent bundle $\nu_{\widehat L_i}$ by $\Lambda(1)$.
\end{lemma}

\begin{proof}
By Lemma \ref{l:eqLag}, $Sp(W)^G$ acts transitively on invariant Lagrangian subspaces so the
path lifts to a family $u(x,t)\in Sp(W)^G$ with $u(x,0)=\operatorname{id}$
and $u(x,t)\Lambda(0)_x=\Lambda(t)_x$.  Its logarithmic derivative
$\xi(t)_x=(\partial_tu)u^{-1}$ lies in the invariant symplectic Lie algebra;
extend it to a compactly supported family $\widehat\xi$ over the base of
the tube.  Consider a Hamiltonian that is of quadratic form
$K_t(x,w)=\tfrac12\omega_W(\widehat\xi(t)_xw,w)$, cut off in $w$ inside
$D(W)$.  Being quadratic in $w$ means that $dK=0$ along $w=0$, so its flow fixes
$\widehat M^G$ pointwise; being $G$-invariant means that it generates an equivariant
isotopy; and its Hamiltonian vector field is to first order the linear
field $w\mapsto\widehat\xi(t)_xw$, so the linearization of the flow along
$\ell$ solves $\partial_t\widehat u=\widehat\xi\,\widehat u$ with
$\widehat u(0)=\operatorname{id}$, that is, $\widehat u=u$.
\end{proof}

\begin{lemma}[Equivariant shear]\label{lem:floer-shear}
Let $\ell$ be one of the $\ell_i$, let $\Lambda\subset W$ be a constant
$G$-invariant Lagrangian subspace, and let $\widehat L\subset\widehat M$ be
a $G$-invariant Lagrangian with $\widehat L\cap\widehat M^G=\ell$ whose
normal tangent bundle along $\ell$ is the constant bundle
$\ell\times\Lambda$.  Then for every relatively compact open $V\subset M^G$
there exist a $G$-invariant Hamiltonian isotopy $\psi^s$, $s\in[0,1]$, of
$\widehat M$ and a neighborhood $\mathcal U$ of $\ell\cap V$ in $\widehat M$
such that:
\begin{enumerate}
    \item $\psi^s$ has compact support, contained in a prescribed
    neighborhood of $\ell\cap\overline V$;
    \item $\psi^s$ fixes $\widehat M^G$ pointwise and $d\psi^s=\operatorname{id}$
    at every point of $\widehat M^G$, so it changes neither the fixed locus
    nor the tangent spaces along it of any Lagrangian;
    \item $\psi^1(\widehat L)\cap\mathcal U=\Xi(\ell\times\Lambda)\cap\mathcal U$.
\end{enumerate}
\end{lemma}

\begin{proof}
Choose a Lagrangian complement $\Lambda'$ of $\Lambda$ in $W$ and a Weinstein
neighborhood of $\ell$ in $M^G$; together with the tube these give cotangent
coordinates near $\ell$ in which the product $\ell\times\Lambda$ is the zero
section and $\widehat L$, having the same tangent spaces along $\ell$, is
the graph of the differential of a $G$-invariant function $S$ vanishing to
second order along $\ell$.  Let $\psi^s$ be the Hamiltonian flow of
$-\upsilon S$, with $\upsilon$ an invariant cutoff equal to $1$ on a neighborhood
$\mathcal U$ of $\ell\cap V$ and supported in the prescribed neighborhood
of $\ell\cap\overline V$; this gives (1).  On $\mathcal U$ the flow is the
fiberwise translation by $-s\,dS$, which at time one carries the graph of
$dS$ onto the zero section, which is (3).  Since $S$ vanishes to second
order along $\ell$, and $d(\upsilon S)=0$ along $\widehat M^G$, the flow fixes
$\widehat M^G$ pointwise with identity linearization there, which is (2).
\end{proof}

\begin{proof}[Proof of Lemma~\ref{lem:floer-straightening}]
It follows directly from applying Lemma \ref{lem:floer-gauge} and then Lemma \ref{lem:floer-shear}. 
\end{proof}

The isotopies
$\phi_i^1$ of Lemma~\ref{lem:floer-straightening} are compactly supported and
Hamiltonian, so
\[
    CF(\phi_0^1\widehat L_0,\phi_1^1\widehat L_1)\;\cong\;CF(\widehat L_0,\widehat L_1)
    \;\cong\;CF(L_0,L_1),
\]
the first quasi-isomorphism by Hamiltonian invariance and the second by the
suspension isomorphism above.  From now on $\widehat L_0,\widehat L_1$ denote
the \emph{straightened} suspended Lagrangians $\phi_0^1\widehat L_0$,
$\phi_1^1\widehat L_1$.

\subsection{Model almost complex structures}
\label{subsec:lagrangian-preparation}

For the sake of later analysis, we want the almost complex structure to be a standard form near invariant intersection points.

The next lemma is only needed so that the gluing analysis in the Appendix~\ref{app:floer-gluing}
can be carried out in a standard local chart. One can skip it if one is not interested in the detailed gluing analysis.
It provides a local chart of $M^G$ at an invariant
intersection point $\cP{p}$; the group acts trivially on $M^G$, so no
equivariance is involved.  Write $2n=\dim M^G$, and let
$\omega_{\mathrm{std}}$ and $J_{\mathrm{std}}$ be the standard symplectic
form and complex structure of $\C^n=T^*\R^n$, for which $\R^n$ and
$\sqrt{-1}\,\R^n=J_{\mathrm{std}}\R^n$ are transverse Lagrangian subspaces.

\begin{lemma}[Darboux chart with linear Lagrangians]
\label{lem:lagrangian-preparation}
Let $\cP{p}$ be an invariant intersection point.  There are a ball
$B\subset\C^n$ about the origin and a symplectic embedding
$\Phi\colon(B,\omega_{\mathrm{std}})\to(M^G,\omega)$ with $\Phi(0)=\cP{p}$
in which both fixed Lagrangians are linear:
\[
 \Phi^{-1}(L^G_0)=\R^n\cap B,\qquad
 \Phi^{-1}(L^G_1)=\sqrt{-1}\,\R^n\cap B.
\]
\end{lemma}

\begin{proof}
Since $T_{\cP{p}}L^G_0$ and $T_{\cP{p}}L^G_1$ are transverse Lagrangian
subspaces of $T_{\cP{p}}M^G$, a linear symplectic isomorphism
$T_{\cP{p}}M^G\to\C^n$ carries them to $\R^n$ and $\sqrt{-1}\,\R^n$.
Composing a Darboux chart at $\cP{p}$ with it gives a symplectic embedding
$\Phi_0$ of a ball in $\C^n$ with $\Phi_0(0)=\cP{p}$ whose preimages of
$L^G_0$ and $L^G_1$ are tangent at the origin to $\R^n$ and to
$\sqrt{-1}\,\R^n$.  These are straightened by two symplectomorphisms of
$\C^n$ near the origin.  Write $(x,y)\in\R^n\oplus\sqrt{-1}\,\R^n$.  Near
the origin $\Phi_0^{-1}(L^G_0)$ is a graph $y=dS_0(x)$ with $dS_0(0)=0$
and $d^2S_0(0)=0$, because it is Lagrangian and tangent to $\R^n$ at $0$;
the shear $\psi_0(x,y)=(x,y-dS_0(x))$ is a symplectomorphism, since it
changes $\sum_idx_i\wedge dy_i$ by
$-\sum_{i,j}\partial_i\partial_jS_0\,dx_i\wedge dx_j=0$, it fixes the
origin with differential the identity, and it carries the graph to
$\R^n$.  After it the second Lagrangian is a graph $x=dS_1(y)$ over
$\sqrt{-1}\,\R^n$ with $dS_1(0)=0$, and the shear
$\psi_1(x,y)=(x-dS_1(y),y)$ is a symplectomorphism carrying it to
$\sqrt{-1}\,\R^n$ and fixing $\R^n$ pointwise, because $dS_1(0)=0$.  The
chart $\Phi:=\Phi_0\circ\psi_0^{-1}\circ\psi_1^{-1}$, restricted to a small
ball, has the required properties.
\end{proof}

\paragraph{The almost complex structure on \texorpdfstring{$M^G$}{M^G}.}
Fix a chart $\Phi_{\cP{p}}$ as in Lemma~\ref{lem:lagrangian-preparation}
at each invariant intersection point, with disjoint images, and choose the
$\omega$-compatible almost complex structure $J_0$ of $M^G$ to be
$(\Phi_{\cP{p}})_*J_{\mathrm{std}}$ on a smaller ball around each of them;
this is possible because the space of compatible almost complex structures
is the space of sections of a bundle with contractible fibers, so a
compatible structure given near a closed subset extends.  Then $J_0$ is
constant in these charts with $J_0T_{\cP{p}}L^G_0=T_{\cP{p}}L^G_1$, and it
is free, and later taken generic, elsewhere.

Let $V\subset M^G$ be the relatively compact open set of the proof of
Lemma~\ref{lem:floer-straightening}, containing $L_0^G\cap L_1^G$, which we
take to be a Liouville subdomain, and define $\widehat U_V:=\Xi(V \times D(W))$, which we call the adapted tube.
We choose $J$ split on the adapted
tube $\widehat U_V$, $J=J_0\oplus J_W$ with $J_0$ cylindrical on $M^G \setminus V$, $J_W$ the standard complex
structure of $W=T^*\R^m$, for which
$\sqrt{-1}\,\R^m=J_W\R^m$, cylindrical outside a compact
set containing the closure of $\widehat U_V$, and arbitrary in between; a
split structure is not cylindrical, which is the reason for defining
$\widehat U_V$.\label{lem:floer-confinement}\label{conv:adapted-region}
Floer solutions lie in a compact set for cylindrical almost complex
structures: since the Lagrangians are conical and $J$ is cylindrical
outside a compact set, the maximum principle for the radial coordinate
confines every finite-energy strip to a fixed compact region, and the
same holds for the fixed pair in $M^G$, whose finite-energy strips lie in
a compact $Q^G\subset M^G$; we take $V\supset Q^G$.  All strips of the straightened pair
then lie in a fixed compact set, and Gromov--Floer compactness applies to
them.

\begin{corollary}[Product chart at an invariant intersection point]\label{cor:lagrangian-preparation-hat}
Composing the tube $\Xi$ with the product of $\Phi_{\cP{p}}$ and the
identity of $W$ gives a $G$-equivariant symplectic embedding
$\Phi\colon B\subset\C^n\oplus W\to\widehat M$ of a ball about the origin,
$G$ acting on the factor $W$ alone, such that
\begin{enumerate}[label=\textup{(L\arabic*)}]
    \item\label{it:lagW-chart}
    $\Phi(0)=\cP{p}$ and $\Phi^*\omega=\omega_{\mathrm{std}}\oplus\omega_W$;
    \item\label{it:lagW-linear}
    both Lagrangians are linear in this chart:
    $\Phi^{-1}(\widehat L_0)=(\R^n\oplus\R^m)\cap B$ and
    $\Phi^{-1}(\widehat L_1)=(\sqrt{-1}\,\R^n\oplus\sqrt{-1}\,\R^m)\cap B$;
    \item\label{it:lagW-J}
    $\Phi^*J=J_{\mathrm{std}}\oplus J_W$ is the standard complex structure
    $J_{\cP{p}}$ of $\C^n\oplus W=\C^{n+m}$, so that
    $J_{\cP{p}}T_{\cP{p}}\widehat L_0=T_{\cP{p}}\widehat L_1$.
\end{enumerate}
\end{corollary}
\begin{proof}
The tube identifies $\widehat L_i\cap\widehat U$ with
$L^G_i\times\Lambda_i(1)$, $\Lambda_0(1)=\R^m$ and
$\Lambda_1(1)=\sqrt{-1}\,\R^m$ (Lemma~\ref{lem:floer-straightening}), and
$J=J_0\oplus J_W$ on the adapted tube; the rest is
Lemma~\ref{lem:lagrangian-preparation} and the choice of $J_0$.
\end{proof}

\begin{convention}\label{conv:lagrangian-preparation}
We fix the charts $\Phi_{\cP{p}}$ of
Corollary~\ref{cor:lagrangian-preparation-hat} at every invariant intersection point, and
a Darboux chart at every outside intersection point in which both
Lagrangians are linear and $J$ is constant, obtained by the same
construction with trivial stabilizer and transported over the free orbit,
with the charts of an orbit disjoint from each other and from the adapted
tube.
\end{convention}

Let $\mathcal J$ denote the set of $t$-dependent, $G$-invariant,
$\omega$-compatible families $J=(J_t)$ that are split, $J_0(t)\oplus J_W$,
on the adapted tube $\widehat U_V$, with the fixed constant normal factor
$J_W$, such that $J_0(t)$ is cylindrical outside $V$ and $J_t$ is
cylindrical for all $t$.

\begin{proposition}[Regularity of fixed strips]
\label{thm:floer-model-regularity}
Let $J\in\mathcal J$ and let $u\subset\widehat M^G$ be a strip between invariant
intersection points.  Then the linearized operator splits,
$D_u=D^T\oplus D^\perp$, the normal operator
$D^\perp=\partial_s+J_W\partial_t$ with the constant transverse boundary pair
has vanishing kernel and cokernel on the exponential-decay spaces, and $u$ is
regular if and only if it is regular as a strip in
$(\widehat M^G,\widehat L_0^G,\widehat L_1^G,J_0)$.  Moreover, for a generic tangential factor $J_0$,
every nonconstant fixed strip is regular, in every Fredholm index.
\end{proposition}

\begin{proposition}[Regularity of non-fixed strips]
\label{thm:floer-free-regularity}
There is a comeagre subset of $\mathcal J$ such that for every $J$ in it,
every moduli space of non-fixed strips between intersection points, and
every codimension-one broken stratum thereof, is cut out transversally.
\end{proposition}

Both theorems are standard.  Along a fixed strip the linearized operator
splits because $J$ is split and invariant; the normal operator is
$\partial_s+A^\perp$ with the constant boundary pair, which has no kernel or
cokernel on the exponentially weighted spaces since $A^\perp$ is invertible,
so regularity is that of the fixed-pair problem, for which the tangential
factor $J_0$ is an unconstrained family on $M^G$.  For non-fixed strips the
perturbations of $J$ are supported where $G$ acts freely, and an
invariant perturbation is obtained from a local one by averaging over the
orbit; since a non-fixed strip is not a translate of any of its other
orbit images, the averaged perturbation acts on it as the local one does,
and the transversality theory of \cite{floer1995transversality} applies.

Fix $J$ in the common residual set of
Propositions~\ref{thm:floer-model-regularity} and~\ref{thm:floer-free-regularity}.
All subsequent spectral and incidence data are chosen for this fixed $J$.

\subsection{Negative evaluation map}
\label{subsec:floer-evaluation}

In this subsection, we want to define an evaluation map which keeps track of the asymptotic behaviour of a strip $u$ in the normal direction.

\subsubsection{Expansion in the normal direction}
Fix the adapted regular almost complex structure of
Subsection~\ref{subsec:lagrangian-preparation}.
In the chart of Corollary~\ref{cor:lagrangian-preparation-hat} at an invariant
intersection point the normal strip equation is
$\partial_sw+J_W\partial_tw=0$ for
$w\colon\R\times[0,1]\to W=\R^m\otimes_\R\C$ with
$w(s,0)\in \R^m$ and $w(s,1)\in\sqrt{-1}\,\R^m$. A
decaying solution is a sum of solutions of the form 
\[w=c\,e^{-\lambda s}e^{-\sqrt{-1}\lambda t}
\]
with
$c\in \R^m$ and $\lambda\in-\pi/2+\pi\Z$, $\lambda<0$.

We define
\begin{equation}\label{eq:floer-ordered-negative-spectrum}
    \lambda_j=-\tfrac{\pi}{2}-(j-1)\pi,\qquad j=1,2,\ldots,
    \qquad
    E_j:= \R^m,
\end{equation} 
For $r\ge1$ put $N:=rm$ and
\begin{equation}\label{eq:floer-finite-spectral-space}
    E_N^-:=\bigoplus_{j=1}^{r}E_j\cong(\R^m)^{\oplus r},
    \qquad
    \dim_{\R}E_N^-=N,
    \qquad
    S_N:=S(E_N^-).
\end{equation}

For a strip $u$ with invariant negative end $\cP{p}$, there is an $s_1 \in \mathbb{R}$
such that $u(s, t)$ lies in the tube for $(s,t) \in (-\infty,s_1] \times [0,1]$.
Let $w_u: (-\infty,s_1] \times [0,1] \to W$ be its normal component in the adapted
tube. As a result, there are $c_j:=c_j(u) \in E_j$ for $j=1,2,\ldots$ such that
\[
w_u(s,t)= \sum_{j} c_j e^{-\lambda_j s}e^{-\sqrt{-1}\lambda_j t}.
\]
We summarise a few properties of $w_u$ in the following lemma:

\begin{lemma}[Exact normal expansion]
\label{rem:lagrangian-exact-expansion}\label{lem:floer-coefficient-readoff}\label{lem:floer-normal-transport}
Let $w_u$ be as above. Then
\begin{enumerate}
    \item For and $\sigma$ and $s':=s-\sigma$, 
    \[
w_u(s',t)=\sum_{j} \left(c_je^{-\lambda_j\sigma}\right) e^{-\lambda_j s'}e^{-\sqrt{-1}\lambda t}
    \]
    \item All $c_j$ vanish if and only if $u\subset M^G$.
    \item The $c_j(u)$ varies continuously with respect to a continuous family of solutions $u$ in a fixed parametrization  
\end{enumerate}
\end{lemma}

\begin{proof}
Properties (1) and (3) are obvious. Property (2) follows from the unique continuation lemma.
\end{proof}

For a non-fixed strip some normal coefficient is nonzero by
Lemma~\ref{rem:lagrangian-exact-expansion}(3).  Define its leading normal
eigenvalue by
\begin{equation}\label{eq:floer-leading-eigenvalue}
    \lambda_{\mathrm{lead}}(u)
    :=\max\{\lambda<0\mid c_\lambda(u)\ne0\}.
\end{equation}
For a chosen cutoff $N=rm$, let
\[
    c_N(u):=
    \bigl(c_{\lambda_1}(u),\ldots,c_{\lambda_r}(u)\bigr)
    \in E_N^-.
\]
Call $u$ \emph{$N$-visible} when $c_N(u)\ne0$.

\subsubsection{The uniform leading-mode bound and the choice of \(N\)}
\label{subsec:lagrangian-finiteness}

The normal Floer spectrum is unbounded below, so a finite cutoff must be
chosen uniformly over every strip that the construction will evaluate.  The
bound is formulated for all strips with invariant negative end,
with no restriction on the expected dimension, and it depends only on the
adapted almost complex structure, not on $N$.

The theorem says that a finite spectral cutoff sees every strip that can
contribute to the differential.

\begin{theorem}[Uniform leading-mode bound]
\label{ass:floer-uniform-leading-bound} 
Fix $J\in\mathcal J$ adapted as in Convention~\ref{conv:lagrangian-preparation}.
There is an integer $r(J)$ such that every non-fixed strip
$u\in\M(\cP{p},\cP{b})$ with invariant negative end $\cP{p}$ and arbitrary
positive end $\cP{b}\in\widehat L_0\cap\widehat L_1$, of any expected
dimension, has a nonzero coefficient among its first $r(J)$ modes:
\[
    \lambda_{\mathrm{lead}}(u)\ge\lambda_{r(J)} .
\]
\end{theorem}

The proof has two ingredients: Gromov--Floer compactness, and the fact that
in the adapted tube the normal equation is exactly linear, so that normal
coefficients can be read on any slice and transported along fixed components
without error.  Exact transport through the preceding fixed components allows the leading
mode to be read on the first non-fixed component of the limiting building.

\begin{proof}
Since $\widehat L_0\cap \widehat L_1$ is finite, there are finitely many pairs
$(\cP{p},\cP{b})$, and for each pair the space $\M(\cP{p},\cP{b})$ of all
strips from $\cP{p}$ to $\cP{b}$, whatever their expected dimension, is
compact up to breaking by Gromov--Floer compactness in the exact convex
setting, the energy of every such strip being the action difference of its
ends.  It suffices to bound each pair and take the minimum;  and nothing below refers to their dimension or to
$N$.  Suppose, for contradiction, that some $\mathcal M(\cP{p},\cP{b})$
with invariant $\cP{p}$ contains non-fixed strips $u_n$ with
$\lambda_{\mathrm{lead}}(u_n)\to-\infty$.  By Gromov--Floer compactness a
subsequence converges to a broken building $(v_1,\ldots,v_k)$ from $\cP{p}$
to $\cP{b}$, with $v_1$ leaving $\cP{p}$: after translations adapted to each
component, $u_n$ converges to $v_i$ uniformly on compact subsets, the energy
on the necks tends to zero, and the images of the necks converge to the
break points.  Let $U:=\widehat U_V$ be the adapted tube
(Subsection~\ref{subsec:lagrangian-preparation}), in which $J$ is split with
constant normal factor $J_W$, the boundary conditions are the constant
Lagrangians $\mathbb{R}^m$ and $\sqrt{-1} \mathbb{R}^m$, and the normal component of any
$J$-holomorphic strip is an exact solution of the constant linear equation
$\partial_sw+J_W\partial_tw=0$.  Recall two consequences of
exact linearity.  First, if a strip lies in $U$ on a region
$(-\infty,s_1]\times[0,1]$ and converges to an intersection point in $M^G$
as $s\to-\infty$, then on that region its normal part is
$\sum_{j}c_j e^{-\lambda_j s}e^{-\sqrt{-1}\lambda_j t}$.
Reparametrizing the strip from $(s,t)$ to $(s',t)=(s-\sigma,t)$ multiplies each
$c_j$ by the nonzero factor $e^{-\lambda_j\sigma}$, so the set of
$\lambda_j$ with $c_j\ne0$, and in particular $\lambda_{\mathrm{lead}}$,
is independent of the parametrization of the strip.  Second, if a strip lies in $U$
on all of $\R\times[0,1]$ and converges to intersection points at both ends,
its normal part is a global solution of the constant linear equation decaying
at both ends; since the asymptotic operator is invertible, the mode expansion
forces it to vanish, and the strip is fixed.

\emph{Step 1: not every component is fixed.}  If all $v_i$ lie in $M^G$,
then  $u_n\subset U$ for large $n$. By the second consequence above, $u_n$ is fixed, a contradiction.

\emph{Step 2: the first non-fixed component.}  Let $v_i$ be the first
component not contained in $M^G$, with negative end $\cP{q}_i\in M^G$. 
By Lemma \ref{rem:lagrangian-exact-expansion} (2), some normal coefficient of $v_i$ is nonzero.
Thus $\lambda_*:=\lambda_{\mathrm{lead}}(v_i)$
is defined and finite.

\emph{Step 3: transport through the fixed components.}  Fix a slice
$\{s_1\}\times[0,1]$ in the negative end region of $v_i$, inside $U$, and
consider $u_n$ in the translation adapted to $v_i$.  On that slice $u_n\to v_i$ in $C^0$. Therefore by Lemma \ref{rem:lagrangian-exact-expansion} (3)
$c_{\lambda_*}(u_n)\to c_{\lambda_*}(v_i)\ne0$, so
$c_{\lambda_*}(u_n)\ne0$ and $\lambda_{\mathrm{lead}}(u_n)\ge\lambda_*$
for all large $n$, contradicting $\lambda_{\mathrm{lead}}(u_n)\to-\infty$.
\end{proof}

\begin{corollary}[Visibility of non-fixed strips]
\label{cor:lagrangian-dis-zero}
For $r\ge r(J)$ and $N=rm$, every non-fixed strip
with invariant negative end is $N$-visible. 
\end{corollary}

\begin{remark}[Order of the choices]\label{rem:cutoff-order}
This is where the order announced at the start of the section is
enforced: $J$ was fixed first, $r(J)$ depends only on $J$, $N$ is chosen
from $r(J)$, and the transport of the sphere pair in the next subsection
changes neither $J$ nor the strips, hence neither $r(J)$ nor $N$.
\end{remark}

\subsubsection{The negative evaluation}
\label{subsec:floer-negative-evaluation}

On $E_N^-\setminus\{0\}$ use the weighted action
\begin{equation}\label{eq:floer-weighted-scaling}
    a\star(c_1,\ldots,c_r)
    :=(e^{-\lambda_1a}c_1,\ldots,e^{-\lambda_ra}c_r),\qquad a\in\R,
\end{equation}
Every orbit meets $S_N$ exactly once.  Let
\[
    q_N\colon E_N^-\setminus\{0\}\longrightarrow S_N
\]
be this weighted normalization.

\begin{definition}[The negative evaluation]
\label{def:floer-finite-spectral-evaluation}
For an $N$-visible strip with invariant negative end $\cP{p}$, that is, one
with $c_N(u)\ne0$, define
\begin{equation}\label{eq:lagrangian-negative-normal-evaluation}
    \Ev^-_{\cP{p},N}(u)
    :=q_N\bigl(c_N(u)\bigr)
    \in S_N.
\end{equation}
\end{definition}

The restriction to $N$-visible strips is not a formality: for a fixed $N$
some non-fixed strips may have all their first $r$ coefficients zero.
Corollary~\ref{cor:lagrangian-dis-zero} says that once $r$ is chosen
large enough, depending on $J$, every non-fixed strip with invariant
negative end is $N$-visible, and that is how $N$ is fixed.  The
transformation law for the coefficients makes this evaluation
translation-invariant and $G$-equivariant.

\begin{lemma}[Weighted normalization]\label{lem:floer-weighted-quotient}
The normalization $q_N:E_N^-\setminus\{0\}\to S_N$
is a smooth $G$-equivariant submersion.  It identifies the weighted orbit
space with $S_N$.
\end{lemma}
\begin{proof}
For $c=(c_1,\ldots,c_r)\ne0$, the function
$F(\sigma,c)=\sum_a e^{-2\lambda_a\sigma}|c_a|^2-1$ increases from $-1$ to
infinity and has strictly positive $\sigma$-derivative.  The implicit function
theorem gives a unique smooth zero $\sigma(c)$, and
$q_N(c)=\sigma(c)\star c$.  The unit sphere is a global slice
transverse to every orbit.  The weighted action commutes with $G$, so the
normalization is equivariant.
\end{proof}

\subsubsection{Sphere data and incidence transversality}
\label{subsec:floer-transversality}
Fix an admissible cutoff $N=rm$ as in
Corollary~\ref{cor:lagrangian-dis-zero}.  All auxiliary Morse theory below
is on the spectral sphere; no Morse function on the ambient symplectic
manifold is used.

Orient $\R^m$, hence each $E_j\cong\R^m$ and $E_N^-=\bigoplus_jE_j$, by the
complex orientation of Definition~\ref{def:group-complex-structure}, the
order of the summands being immaterial since each is even-dimensional,
and orient $S_N=S(E_N^-)$ as the boundary of the oriented unit ball. 
Each $E_j$ is a copy of $\R^m$, a sum of rotation planes
by Lemma~\ref{lem:real-irreps-of-Zp}, so the action of $G$ on $S_N$ is free and orientation preserving.  Its
quotient $L_N=S_N/G$ is an oriented generalized
lens space.  Choose a Morse--Smale pair
$(\bar h_N,\bar k_N)$ on this quotient with a unique
maximum and minimum, in standard quadratic coordinates with constant
metric near its critical points, and let $(h_N,k_N)$
be its lift.
Write
\[
 C_N=CM_*(S_N,h_N,k_N),
 \qquad j(\cP{v})=\ind_{h_N}(\cP{v}).
\]
Each chain group is a free $\Z[G]$-module.  With the descending-cell
orientations lifted from the quotient, positive orientations on the
minima, and the top cells oriented as the sphere
Let $M_g$ are the lifts of the quotient maximum.

The Morse differential $d_N$, augmentation $\epsilon_N$, and top cycle $\Theta_N$ of $C_N$ are defined by

\begin{gather}
 d_N\cP{v}
 =\sum_{j(\cP{w})=j(\cP{v})-1}
   \#\M_{h_N}^0(\cP{v},\cP{w})\,\cP{w},
 \qquad
 \epsilon_N(\cP{v})=
 \begin{cases}1,&j(\cP{v})=0,\\0,&j(\cP{v})>0,\end{cases}
 \label{eq:floer-sphere-differential}\\
 \Theta_N=\sum_{g\in G}M_g, \label{eq:floer-top-cycle}
 \end{gather}

\begin{lemma}[The sphere and its Morse data]
\label{lem:floer-sphere-data}
The differential, augmentation, and top cycle of $C_N$ satisfy
\begin{align}
 (d_N)^2=0,\quad d_N\Theta_N=0,
 \quad \epsilon_Nd_N=0,
 \quad \epsilon_N(\Theta_N)=0.
 \label{eq:floer-local-identities}
\end{align}
  The cycle
$\Theta_N$ represents the fundamental class of $S_N$.
\end{lemma}
\begin{proof}
  The ordinary compactification
of sphere trajectories proves $(d_N)^2=0$.
The oriented boundary of a one-dimensional descending cell has total
coefficient zero, proving $\epsilon_Nd_N=0$.
The lifted top cells, all with the ambient orientation, constitute the
fundamental chain of the sphere, so their sum has zero boundary and
represents its fundamental class.  Finally $N\ge2$, and the augmentation
vanishes in the positive local degree $N-1$ of the top cycle.  The precise
orientation convention for sphere trajectories is recorded in
Lemma~\ref{lem:floer-sphere-boundary}.
\end{proof}

For a critical point $\cP{v}$ of $h_N$, let
$\descending_{h_N}(\cP{v})$ be its descending manifold: the
points reached at time zero by solutions of
$\dot\vartheta=-\grad_{k_N}h_N(\vartheta)$ on
$(-\infty,0]$ with limit $\cP{v}$.  Its dimension is $j(\cP{v})$.
Let $e_{\cP{v}}$ denote its inclusion in $S_N$, and let
$\M_{h_N}(\cP{v},\cP{w})$ denote the space of nonconstant
whole trajectories modulo translation.

For an invariant $\cP{p}$ and any target $\cP{b}$, set
$\M^\circ(\cP{p},\cP{b})=\M(\cP{p},\cP{b})\setminus
\M_{M^G}(\cP{p},\cP{b})$; when the target is outside this is the whole
space.  Define
\begin{equation}\label{eq:floer-incidence}
 \mathcal N_{\cP{p},N}(\cP{v};\cP{b})
 =\descending_{h_N}(\cP{v})
 \,{}_{e_{\cP{v}}}\!\times_{\Ev^-_{\cP{p},N}}
 \M^\circ(\cP{p},\cP{b}).
\end{equation}
Equivalently, an element is a strip $u$ whose evaluation belongs to the
specified descending cell, together with the uniquely determined
half-trajectory $\vartheta$ from $\cP{v}$ to that evaluation.  We call
$(\vartheta,u)$ an \emph{incidence configuration} and identify it with $u$
when no confusion arises.  
On a component with $\iota=\ind D_u$, transversality gives
\begin{equation}\label{eq:floer-incidence-dimension}
 \dim\mathcal N_{\cP{p},N}(\cP{v};\cP{b})
 =\iota-N+j(\cP{v}).
\end{equation}
We use superscripts on $\mathcal N$ to refer to its dimension.
For example, $\mathcal N^0$ means the $0$-dimensional moduli space of $\mathcal N$.
A rigid incidence configuration generally has an ambient strip in a
positive-dimensional moduli space, of dimension $N-1-j(\cP{v})$.

For a diffeomorphism $\varphi$ of $L_N$ isotopic to the identity, with
$\widetilde\varphi$ the $G$-equivariant lift of the isotopy to $S_N$, the
\emph{transport} of the sphere pair is
$(h_N\circ\widetilde\varphi^{-1},\widetilde\varphi_*k_N)$; it is again a
lifted Morse--Smale pair, and $\widetilde\varphi$ carries critical points,
descending cells, orientations, augmentation and top cycle to those of the
new pair.\label{def:floer-sphere-transport}

\begin{proposition}[Incidence transversality]
\label{prop:floer-transversality-outside}\label{prop:floer-generic-transversality}
For the fixed regular $J$, a transport of the reference sphere pair can be
chosen so that all spaces \eqref{eq:floer-incidence} are transverse, for
every invariant source, every target, and every sphere critical point.
The same transport makes all broken-stratum and iterated incidence
conditions used below transverse.
\end{proposition}
\begin{proof}[Sketch]
By ambient regularity the strip spaces, and the products of their
components defining each broken stratum, are smooth manifolds, and the
evaluation is smooth on them: on a translation slice its coefficients are
linear projections of the strip at a fixed time in the adapted tube, the
normalization is smooth by Lemma~\ref{lem:floer-weighted-quotient}, and
on a broken configuration the evaluation is read on the first non-fixed
factor, preceding fixed factors carrying it by the identity.  None of this
depends on the sphere pair.  Transporting the pair by a diffeomorphism of
$L_N$ chosen from a finite-dimensional family whose parameter derivative
spans $TL_N$ everywhere, parametric transversality makes all of the
countably many evaluation maps of all components and breaking types
transverse to the descending cells for a residual set of parameters;
iterated incidences lie on distinct strip factors, or on a trajectory
factor of the transported pair for sphere breaking, so they reduce to the
individual maps.  The dimension formula
\eqref{eq:floer-incidence-dimension} follows.
\end{proof}

From now on $h_N,k_N,C_N$ and
$\Theta_N$ denote the transported data.  Compactness of the
zero-dimensional incidence spaces, and hence finiteness of the counts,
is proved in Proposition~\ref{prop:IIF-strata}.  Their orientations are
specified in Subsection~\ref{subsec:floer-orientations}.

\begin{lemma}[Top cells]\label{lem:floer-top-cells}
For an invariant source $\cP{p}$ and any target $\cP{b}$, the evaluation
of every ordinary rigid non-fixed strip lies in exactly one descending
cell of a lifted maximum $M_g$, so that
\begin{equation}\label{eq:floer-top-cell-sum}
 \M^{\circ,0}(\cP{p},\cP{b})
 =\coprod_{g\in G}\mathcal N^0_{\cP{p},N}(M_g;\cP{b}).
\end{equation}
\end{lemma}
\begin{proof}
The descending cells of the maxima are disjoint open subsets of the
sphere, and their complement is the union of the lower cells.  By
Proposition~\ref{prop:floer-transversality-outside} the evaluation on a
zero-dimensional strip space is transverse to every lower cell, where the
incidence space has negative dimension by
\eqref{eq:floer-incidence-dimension}, so it misses each of them; its value
therefore belongs to exactly one top cell.
\end{proof}

\subsection{The stabilized Floer complex}
\label{subsec:floer-stabilized}

We define the stabilized Floer complex in this subsection.
The basic idea is to replace $\Z\langle\st{\cP{p}}\rangle$ for every invariant intersection point $\cP{p}$ by its quasi-isomorphic model
\[
 \mathcal T_{\cP{p}}=C^s_{\cP{p}}\oplus\Z\langle\st{\cP{p}}\rangle,
 \qquad d_{\mathcal T}=\dsph+\epsilon_N\,\st{\cP{p}},
\]
where $C^s_{\cP{p}}$ is a copy of $C_N$. The $N$ was chosen large enough so that the module $\oplus_{\cP{p}} \Z\langle\st{\cP{p}}\rangle$ will form a subcomplex.
We write $\Theta_{\cP{p}}$ for the copy of the top cycle \eqref{eq:floer-top-cycle} in $C^s_{\cP{p}}$.

\begin{lemma}[Augmented local sphere]\label{lem:floer-augmented-sphere}
The homology of $(\mathcal T_{\cP{p}},d_{\mathcal T})$ is
$\Z\langle[\Theta_{\cP{p}}]\rangle$, and the same holds after replacing
$\Z$ by any coefficient ring.
\end{lemma}
\begin{proof}
Before forgetting the local degree this is the augmented Morse complex of
$S_N\cong S^{N-1}$, with the formal generator in degree $-1$.  The
augmentation removes the degree-zero homology, leaving the reduced
homology of the sphere, freely generated by its fundamental class
$\Theta_{\cP{p}}$ in degree $N-1$.
\end{proof}

To explain the construction, we define
\[
\begin{aligned}
    C^o&=\Z\langle\text{outside intersection points}\rangle,\\
    C^s&=\bigoplus_{\cP{p}\text{ invariant}}C^s_{\cP{p}},\\
    C^i&=\Z\langle\st{\cP{p}}\mid\cP{p}\text{ invariant}\rangle,
\end{aligned}
\]
and put
\[
    CF_{\mathrm{st}}(\widehat L_0,\widehat L_1):=C^o\oplus C^s\oplus C^i.
\]

For $\alpha,\beta\in\{o,s,i\}$, the notation
$\D{\alpha}{\beta}:C^\alpha\to C^\beta$ places the source in the subscript
and the target in the superscript.  Source summands index columns in the
matrix below.

For every invariant point, a local
generator $\cP{v}\in\crit(h_N)$ of local Morse index
$j(\cP{v})$ placed at $\cP{p}$ is written $\cP{v}|_{\cP{p}}$. 
The copy of \eqref{eq:floer-top-cycle} at $\cP{p}$ is denoted
$\Theta_{\cP{p}}$.
Let $\dsph$ be the direct sum of the local sphere Morse
differentials.  With signed counts of rigid regular spaces, oriented as in
Subsection~\ref{subsec:floer-orientations}, define
\begin{align}
    \D{o}{o}\cP{x}
    &=\sum_{\cP{y}\text{ outside}}
      \#\M^0(\cP{x},\cP{y})\,\cP{y},
    \label{eq:floer-doo}\\
    \D{o}{s}\cP{x}
    &=\sum_{\cP{q}\text{ invariant}}
      \#\M^0(\cP{x},\cP{q})\,\Theta_{\cP{q}},
    \label{eq:floer-dos}\\
    \D{s}{o}(\cP{v}|_{\cP{p}})
    &=\sum_{\cP{z}\text{ outside}}
      \#\mathcal N^0_{\cP{p},N}(\cP{v};\cP{z})\,\cP{z},
    \label{eq:floer-dso}\\
    \D{s}{s}
    &=\dsph+\Dinv+\Dnon,
    \label{eq:floer-dss}\\
    \Dinv(\cP{v}|_{\cP{p}})
    &=(-1)^{N-1-j(\cP{v})}
      \sum_{\cP{q}\text{ invariant}}
      \sum_{u\in\M^0_{M^G}(\cP{p},\cP{q})}
      \epsilon(u)\,\cP{v}|_{\cP{q}},
    \label{eq:floer-dinv}\\
    \Dnon(\cP{v}|_{\cP{p}})
    &=\sum_{\cP{q}\text{ invariant}}
      \#\mathcal N^0_{\cP{p},N}(\cP{v};\cP{q})\,
      \Theta_{\cP{q}},
    \label{eq:floer-dnon}\\
    \D{s}{i}|_{C^s_{\cP{p}}}
    &=\epsilon_N(\,\cdot\,)\,\st{\cP{p}},
    \label{eq:floer-dsi}\\
    \D{i}{i}\st{\cP{p}}
    &=\sum_{\cP{q}\text{ invariant}}
      \#\M^0_{M^G}(\cP{p},\cP{q})\,\st{\cP{q}},
    \label{eq:floer-dii}
\end{align}
The other blocks are set to zero:
\[
 \D{o}{i}=\D{i}{o}=\D{i}{s}=0.
\]
Thus the operator has matrix
\begin{equation}\label{eq:floer-block-matrix}
 \partial=\begin{pmatrix}
 \D{o}{o}&\D{s}{o}&0\\
 \D{o}{s}&\dsph+\Dinv+\Dnon&0\\
 0&\D{s}{i}&\D{i}{i}
 \end{pmatrix}.
\end{equation}
The proof of $\partial^2=0$ is given in
Subsection~\ref{subsec:floer-dsquared}, using the compactification and gluing
of Appendices~\ref{app:floer-compactification} and~\ref{app:floer-gluing}.
Once this identity is established, $C^i$ is a subcomplex, and its
restriction is the Floer differential of the fixed pair.  Its quotient is
free over $\Z[G]$ as a module.

Before we move on to the next subsection, let us explain some features of $\partial$.
\begin{enumerate}
    \item $\M^0_{M^G}(\cP{p},\cP{q})$ is the union of the index-one components
of the space of fixed strips; a fixed strip has the same index in $\widehat M$ and in
$M^G$, the normal operator along it having index zero
(Proposition~\ref{thm:floer-model-regularity}).

\item An outgoing non-invariant strip from an invariant intersection point is coupled with the incidence condition on the spectral sphere to form the moduli spaces $\mathcal N^0_{\cP{p},N}(\cP{v};\cP{z})$ and $\mathcal N^0_{\cP{p},N}(\cP{v};\cP{q})$ which contribute to $\D{s}{o}$ and $\Dnon$, respectively.

\item A rigid incoming strip at an invariant target is represented by
$\D{o}{s}$ or $\Dnon$ and lands in its top cycle.

\item The sign $(-1)^{N-1-j(\cP{v})}$ is the transport sign defined in
Definition~\ref{def:floer-transport-sign}.  It is $+1$ on the top cycle
and $-1$ on the local minima.

\end{enumerate}

\subsection{Orientations}\label{subsec:floer-orientations}

This subsection fixes the orientation data of the stabilized complex, the
conventions by which every other orientation is induced from them, and the
four facts the rest of the paper uses: the boundary signs of strip and
sphere spaces (Lemmas~\ref{lem:floer-boundary-formulas}
and~\ref{lem:floer-sphere-boundary}), the identification of the
orientations induced on fixed strips with those of the fixed data
(Lemma~\ref{lem:floer-fixed-strips}), the invariance of all orientations
under the group (Lemma~\ref{lem:floer-equivariant-orientations}), and the
transport sign (Definition~\ref{def:floer-transport-sign}).  The boundary
signs are computed in Subsection~\ref{subsec:IIF-orientations}, and the
fixed-strip lemma is proved in Appendix~\ref{app:floer-orientations}.  Two
elementary facts do most of the equivariance work: the complex linearity
of equivariant maps for the group complex structure
(Lemma~\ref{lem:group-complex-structure}), and the following.

\begin{lemma}[Odd order]\label{lem:odd-order}
Every linear action of $G$ on a real line is trivial, and every action of
$G$ on a set with two elements is trivial.
\end{lemma}

\subsubsection{Pin data}

Fix once and for all one of the two groups $\mathrm{Pin}^\pm(n)$; either
choice works below, and the choice must be the same for both Lagrangians.
A \emph{relative Pin structure} on a Lagrangian $L\subset M$ with respect
to a real vector bundle $E\to M$, the \emph{background bundle}, is a Pin
structure on $TL\oplus E|_L$; for $E=0$ this is a Pin structure on $L$, and
$w_2(E)|_L$ is the background class.  Relative Pin structures on $L_0$ and
$L_1$ with respect to the same background bundle are what is needed to
orient the moduli spaces of strips with boundary on $L_0$ and $L_1$ and to
define $HF(L_0,L_1)$ over $\Z$ \cite{wehrheim2015orientations}; see
\cite[Section~11]{seidel2008fukaya} for Pin structures and
\cite[Chapter~8]{fooo2009lagrangian} for relative spin structures.

A Pin structure on a Euclidean vector bundle $V\to L$ of rank $n$ is a
principal $\mathrm{Pin}(n)$-bundle $P\to L$ together with a
$\mathrm{Pin}(n)$-equivariant map $P\to\mathrm{Fr}(V)$ to the orthonormal
frame bundle covering the identity of $L$, where $\mathrm{Pin}(n)$ acts on
$\mathrm{Fr}(V)$ through $\mathrm{Pin}(n)\to O(n)$.  Put
$V_i:=TL_i\oplus E|_{L_i}$ with a $G$-invariant metric.  Since $L_i$ and
$E$ are $G$-equivariant, each $g\in G$ acts on $V_i$ by orthogonal bundle
maps covering its action on $L_i$, hence on $\mathrm{Fr}(V_i)$.  A
\emph{lift} of $g$ to a Pin structure $P_i$ on $V_i$ is a
$\mathrm{Pin}(n)$-equivariant bundle map $\widetilde g$ making
\begin{equation}\label{eq:floer-pin-lift}
\begin{tikzcd}[row sep=large, column sep=large]
 P_i \arrow[r,"\widetilde g"] \arrow[d] & P_i \arrow[d]\\
 \mathrm{Fr}(V_i) \arrow[r,"g"] \arrow[d] & \mathrm{Fr}(V_i) \arrow[d]\\
 L_i \arrow[r,"g"] & L_i
\end{tikzcd}
\end{equation}
commute.  A lift exists exactly when $g^*P_i\cong P_i$ as Pin structures,
and two lifts of the same $g$ differ by the deck transformation
$-1\in\mathrm{Pin}(n)$.  The $G$-action on $L_i$ \emph{lifts to an action
on $P_i$} if lifts $\widetilde g$ can be chosen for all $g\in G$ so that
$\widetilde{hh'}=\widetilde h\,\widetilde h'$ for all $h,h'\in G$.

\begin{lemma}[Invariant Pin structures]\label{lem:floer-invariant-pin}
Under Assumption~\ref{ass:floer-pin} there are Pin structures $P_i$ on
$V_i=TL_i\oplus E|_{L_i}$, $i=0,1$, to which the $G$-action on $L_i$ lifts
as an action.  These are the Pin structures used below.
\end{lemma}

\begin{proof}
A Pin structure on $V_i:=TL_i\oplus E|_{L_i}$ is a principal
$\mathrm{Pin}(n)$-bundle $P_i\to L_i$ with a $\mathrm{Pin}(n)$-equivariant
map $P_i\to\mathrm{Fr}(V_i)$ to the orthonormal frame bundle covering the
identity of $L_i$.  Since $E$ is equivariant, each $g\in G$ acts on $V_i$
by orthogonal bundle maps covering its action on $L_i$, hence on
$\mathrm{Fr}(V_i)$, and a \emph{lift} of $g$ is a
$\mathrm{Pin}(n)$-equivariant bundle map $P_i\to P_i$ covering this action;
it exists exactly when $g^*P_i\cong P_i$.

\emph{An invariant Pin structure.} 
Now we construct a Pin structure $P_i$ on $V_i:=TL_i\oplus E|_{L_i}$ such that $g^*P_i\cong P_i$.
The Pin structures on $V_i$ form a
torsor over the $\F_2$-vector space $H^1(L_i;\Z_2)$
\cite[Section~2]{solomon-it}, on which $G$ acts by pullback.  Given one
Pin structure $P$, the affine combination
$P':=\sum_{j=0}^{p-1}(\sigma^j)^*P$ of its $p$ translates is a well defined point of the torsor because
the coefficients sum to $p=1$ in $\F_2$.  In other words, for any Pin structure $Q$, the sum $P':=Q+\sum_{j=0}^{p-1}((\sigma^j)^*P-Q)$ defines a Pin structure which is independent of $Q$.
Note that $\sigma^*$ permutes the
summands of $P'$, so $\sigma^*P'=P'$.  Hence every $g=\sigma^k$ can be lifted to a $\mathrm{Pin}(n)$-equivariant bundle map
$P' \to P'$.

\emph{Normalizing the lifts.}  Let $\widetilde\sigma$ be a lift of
$\sigma$.  Two lifts of the same element differ by the deck
transformation $-1\in\mathrm{Pin}(n)$, so $c:=\widetilde\sigma^{\,p}$, a
lift of the identity, is a locally constant function $L_i\to\{\pm1\}$, and
$c$ is constant along $\sigma$-orbits of components because it commutes
with $\widetilde\sigma$.  Let $c\widetilde\sigma$ be the lift
$q\mapsto c(\pi(q))\,\widetilde\sigma(q)$ of $\sigma$, with
$\pi\colon P_i\to L_i$ the projection.  Since $\widetilde\sigma$ commutes
with the deck transformations, over a component $C$ its $p$-th power is
\[
 (c\widetilde\sigma)^p=c(C)\,c(\sigma C)\cdots c(\sigma^{p-1}C)\,\widetilde\sigma^{\,p}
 =c(C)^{p+1}=1,
\]
because $c$ is constant along the orbit of $C$ and $p$ is odd.  Hence
$\sigma^k\mapsto(c\widetilde\sigma)^k$ is an action of $G$ on $P_i$ by
$\mathrm{Pin}(n)$-equivariant bundle maps covering the action on
$\mathrm{Fr}(V_i)$.
\end{proof}

For $p=2$ neither step is available: the average of two Pin structures is
not defined, and a lift of an involution need not square to one.  This is
one of the reasons the theory of \cite{seidel2010localization} is set up
over $\F_2$.

Assumption~\ref{ass:floer-pin} and the equivariant structures of
Lemma~\ref{lem:floer-invariant-pin} are preserved by the suspension and the
straightening: the suspension summand $T^*\R^d$ with
$\R^d,\sqrt{-1}\,\R^d$ is a product with $G$-equivariant Pin data, these
Lagrangians being contractible, and a $G$-equivariant
Hamiltonian isotopy carries equivariant relative Pin structures along.  We
therefore use the induced data on the adapted pair of
Convention~\ref{conv:lagrangian-preparation}; in the rest of this subsection
$E$ denotes the induced background bundle over $\widehat M$, and
$\widehat\nu=M^G\times W$ the normal bundle of $M^G$ in $\widehat M$,
trivialized by the tube of
Subsection~\ref{subsec:floer-straightening-lagrangians}.

\subsubsection{Orientation lines and the sign of a strip}

We use the formalism of \cite{wehrheim2015orientations}, see also
\cite[Section~11]{seidel2008fukaya} and
\cite[Section~8.1]{fooo2009lagrangian}; the properties of it that enter
the proofs are listed in Appendix~\ref{app:floer-orientations}.  To each
intersection point $\cP{x}\in\widehat L_0\cap\widehat L_1$ the relative Pin
structures assign a real line $o_{\cP{x}}$, the \emph{orientation line} of
$\cP{x}$: it is the determinant line of a Cauchy--Riemann operator
$D_{\cP{x}}$ on the disc with one boundary puncture, the \emph{cap
operator}, whose boundary condition is a \emph{capping path}, a path of
Lagrangian subspaces of $T_{\cP{x}}\widehat M$ from $T_{\cP{x}}\widehat L_0$
to $T_{\cP{x}}\widehat L_1$ in the homotopy class fixed in
Definition~\ref{conv:floer-orientation-data}\ref{it:F-paths} below, together
with a Pin structure along the path extending the given ones at its ends.
The extension is unique up to one sign, and changing it reverses the
isomorphism \eqref{eq:floer-gluing-isomorphism} at $\cP{x}$ for every
strip with an end there, which is the same as reversing the orientation of
$o_{\cP{x}}$.  For a regular strip $u$ from $\cP{x}$ to $\cP{y}$, linear
gluing of the cap at $\cP{y}$ to $u$ gives a canonical isomorphism
\begin{equation}\label{eq:floer-gluing-isomorphism}
    o_{\cP{x}}\;\cong\;\det(D_u)\otimes o_{\cP{y}},
\end{equation}
where $D_u$ is the linearized operator of \eqref{eq:lagrangian-floer-strip}
at $u$.  These isomorphisms are compatible with breaking: for a broken
strip $(u_1,u_2)$ through $\cP{y}$ and the glued strip $u$, the composite
of \eqref{eq:floer-gluing-isomorphism} for $u_1$ and $u_2$ agrees with
\eqref{eq:floer-gluing-isomorphism} for $u$ under the linear gluing
isomorphism $\det(D_{u_1})\otimes\det(D_{u_2})\cong\det(D_u)$, in which the
incoming factor comes first; this order is forced by the composite
$o_{\cP{x}}\cong\det(D_{u_1})\otimes o_{\cP{y}}
\cong\det(D_{u_1})\otimes\det(D_{u_2})\otimes o_{\cP{z}}$.
When $u$ is regular, $\det(D_u)=\Lambda^{\mathrm{top}}\ker D_u$ and
$\ker D_u$ is the tangent space of the parametrized moduli space,
containing $\partial_su$.

\begin{convention}[Orientation of strip spaces]
\label{conv:floer-strip-orientation}
Orient the line $o_{\cP{x}}$ for every $\cP{x}\in \widehat L_0\cap \widehat L_1$, as
specified in Definition~\ref{conv:floer-orientation-data} below.  For a
regular strip $u$ from $\cP{x}$ to $\cP{y}$, orient $\ker D_u$ so that
\eqref{eq:floer-gluing-isomorphism} reads
$o_{\cP{x}}=\det(D_u)\otimes o_{\cP{y}}$ as oriented lines, and orient
$\M(\cP{x},\cP{y})$ by
\[
    \ker D_u=T_{[u]}\M(\cP{x},\cP{y})\oplus\R\,\partial_su
\]
as oriented spaces.  On an index-one component the orientation of the point
$[u]$ is a sign $\epsilon(u)$, and $\#\M^0(\cP{x},\cP{y})$ is the sum of
these signs over all index-one components.
\end{convention}

The order in this convention is ``moduli first, translation second.''
It fixes the sign of each rigid strip once its endpoint orientation lines
have been chosen.

\begin{lemma}[Boundary orientation of strip spaces]\label{lem:floer-boundary-formulas}
Suppose that near a codimension-one stratum
$\M(\cP{x},\cP{y})\times\M(\cP{y},\cP{z})$ the compactification
$\overline{\M(\cP{x},\cP{z})}$ is a manifold with boundary carrying the
neck-length collar of the standard gluing theory
\cite{floer1995transversality,seidel2008fukaya}.  Then, with
Convention~\ref{conv:floer-strip-orientation}, the boundary orientation
along the stratum is $(-1)^{\dim\M(\cP{x},\cP{y})}$ times the product
orientation.  In particular the boundary of a one-dimensional strip space
carries the product orientation; the sum over all intermediate
intersection points is the identity $d_{\mathrm F}^2=0$.
\end{lemma}

The proof is a short computation with ordered bases, given in
Subsection~\ref{subsec:IIF-orientations} once the collar is available; its
inputs are Convention~\ref{conv:floer-strip-orientation}, the order of the
factors in the linear gluing, and the fact that the outward direction of
the collar is the relative translation of the two components.

\subsubsection{Orientation data}

At an invariant intersection point $\cP{p}$ the adapted data are a
product: $T_{\cP{p}}\widehat M=T_{\cP{p}}M^G\oplus W$,
$T_{\cP{p}}\widehat L_i=T_{\cP{p}}\widehat L_i^G\oplus\Lambda_i(1)$ with
$\Lambda_0(1)=\R^m$ and $\Lambda_1(1)=\sqrt{-1}\,\R^m$, and
$J_{\cP{p}}=J_0\oplus J_W$, by Lemma~\ref{lem:floer-straightening}
and Corollary~\ref{cor:lagrangian-preparation-hat}.

\begin{definition}[Orientation data]
\label{conv:floer-orientation-data}
Fix once and for all:
\begin{enumerate}[label=\textup{(F\arabic*)}]
    \item\label{it:F-pin}
    the background bundle $E$ and the $G$-equivariant relative Pin
    structures $P_i$ on $T\widehat L_i\oplus E|_{\widehat L_i}$ of
    Lemma~\ref{lem:floer-invariant-pin}, with the lifts of the group
    elements normalized to form an action of $G$, as in its proof;
    \item\label{it:F-paths}
    at every intersection point a capping path, together with a Pin
    structure along it: at an invariant $\cP{p}$ the path
    $t\mapsto e^{\sqrt{-1}\pi t/2}\,T_{\cP{p}}\widehat L_0$ in the adapted
    chart, which ends at $T_{\cP{p}}\widehat L_1$ by
    Corollary~\ref{cor:lagrangian-preparation-hat}\ref{it:lagW-J}, is $G$-invariant
    because $G$ commutes with $J_{\cP{p}}$, with either
    of its two Pin extensions, both of which are $G$-invariant by
    Lemma~\ref{lem:odd-order}; at an outside point, a path of the same form with respect to a chart in Convention \ref{conv:lagrangian-preparation} and an extension
    chosen at one point of each free orbit and transported along the orbit;
    \item\label{it:F-fixed}
    for every invariant intersection point $\cP{p}$, an orientation of
    $o_{\cP{p}}$, chosen arbitrarily;
    \item\label{it:F-outside}
    for every outside intersection point, an orientation of $o_{\cP{x}}$
    chosen arbitrarily at one point of each free orbit and transported
    along the orbit by the isomorphisms $o_{\cP{x}}\to o_{g\cP{x}}$
    induced by the lifted action;
    \item\label{it:F-sphere}
    the orientations of $E_N^-$, of $S_N$, of $L_N$ and of the descending
    cells fixed in Lemma~\ref{lem:floer-sphere-data}.
\end{enumerate}
All other orientations are induced from these data without further choice:
those of strip spaces by Convention~\ref{conv:floer-strip-orientation}; of
boundaries and broken configurations with the outward normal first and,
for $V=A\oplus B$, with a basis of $A$ before a lifted basis of $B$
(Lemmas~\ref{lem:floer-boundary-formulas}
and~\ref{lem:floer-sphere-boundary}); of incidence spaces by the
normal-first rule of Convention~\ref{conv:floer-incidence-orientation};
and the transport of a sphere generator along a fixed strip carries the
sign of Definition~\ref{def:floer-transport-sign}.
\end{definition}

\subsubsection{Fixed strips}
Three things happen in this subsubsection.  First, the fixed strips are
oriented already: a fixed strip space is open and closed in the ambient
one (Lemma~\ref{lem:floer-fixed-open-closed}), so it carries the
orientation of Convention~\ref{conv:floer-strip-orientation} computed in
$\widehat M$, and a rigid fixed strip $u$ enters \eqref{eq:floer-dinv}
and \eqref{eq:floer-dii} with its sign $\epsilon(u)$ as a strip in
$\widehat M$.  Thus $(C^i,\D{i}{i})$ is a differential module:
$(\D{i}{i})^2=0$ follows from the boundary signs of
Lemma~\ref{lem:floer-boundary-formulas} applied in $\widehat M$ once the
compactification of Subsection~\ref{subsec:floer-dsquared} is in place,
and nothing about $M^G$ enters.  Second, the ambient data
\ref{it:F-pin}--\ref{it:F-fixed} induce orientation data on the fixed
pair $(M^G,\widehat L^G_0,\widehat L^G_1)$: relative Pin structures $Q_i$
with respect to a background bundle $E'$ on $M^G$, tangential capping
paths, and orientations of the orientation lines of the fixed pair.
Third, Lemma~\ref{lem:floer-fixed-strips} says that the sign of a fixed
strip computed in $\widehat M$ agrees with the sign computed in $M^G$
from the induced data, so that $H(C^i)$ is a Floer homology of the fixed
pair in the usual sense.  The third step is needed in one place only,
Subsection~\ref{subsec:application-normal-data}, where the induced data
are compared with the standard ones; for Theorem~\ref{thm:floer-smith}
alone, $H(C^i)$ could serve as the definition of its right-hand side.

\begin{lemma}[Fixed strips are open and closed]\label{lem:floer-fixed-open-closed}
Let $\cP{p},\cP{q}$ be invariant intersection points.  Then
$\M_{M^G}(\cP{p},\cP{q})$ is open and closed in $\M(\cP{p},\cP{q})$, so
that $\M(\cP{p},\cP{q})=\M_{M^G}(\cP{p},\cP{q})\sqcup\M^\circ(\cP{p},\cP{q})$
as manifolds, and a fixed strip is oriented by
Convention~\ref{conv:floer-strip-orientation} in $\widehat M$.
\end{lemma}
\begin{proof}
By Proposition~\ref{thm:floer-model-regularity} the normal operator along
a fixed strip has index zero and no kernel or cokernel, so near a fixed
strip the two spaces are smooth of the same dimension and
$\M_{M^G}(\cP{p},\cP{q})$ is open; it is closed because $M^G$ is closed.
\end{proof}

\paragraph{Induced data on the fixed pair.}
Along $L^G_i$ one has $T\widehat L_i|_{L^G_i}=TL^G_i\oplus\nu_{\widehat L_i}$
with $\nu_{\widehat L_0}=L^G_0\times\R^m$ and
$\nu_{\widehat L_1}=L^G_1\times\sqrt{-1}\,\R^m$ after the straightening,
constant subbundles of the trivialized normal bundle
$\widehat\nu=M^G\times W$.  Put
\begin{equation}\label{eq:floer-fixed-background}
 E':=E|_{M^G}\oplus(M^G\times\R^m),
\end{equation}
let $\rho_0$ be the identity of $L^G_0\times\R^m$, and let
$\rho_1\colon L^G_1\times\sqrt{-1}\,\R^m\to L^G_1\times\R^m$ be $-J_W$;
both are $G$-equivariant isomorphisms, since $J_W$ commutes with $G$.
A Pin structure on a vector bundle restricts to submanifolds and is
carried along by any isomorphism of vector bundles over the same base,
both operations acting on frame bundles; metrics play no role, the space
of metrics being contractible.  Restricting $P_i$ to $L^G_i$ therefore
gives a Pin structure on
$V_i|_{L^G_i}=TL^G_i\oplus\nu_{\widehat L_i}\oplus E|_{L^G_i}$, and
$1\oplus\rho_i\oplus1$ carries it to a Pin structure $Q_i$ on
$TL^G_i\oplus(L^G_i\times\R^m)\oplus E|_{L^G_i}=TL^G_i\oplus E'|_{L^G_i}$,
that is, a relative Pin structure on $L^G_i$ with respect to the
background bundle $E'\to M^G$.  The summand $M^G\times\R^m$ of $E'$ is
there only because a background bundle must live on all of $M^G$, whereas
$\nu_{\widehat L_0}$ and $\nu_{\widehat L_1}$ live only over $L^G_0$ and
$L^G_1$ and are different subspaces of $W$.  A Pin structure on $L^G_i$
itself need not exist; the background class $w_2(E')|_{L^G_i}$ is forced
by $P_i$.  In the same way the capping data of \ref{it:F-paths} induce
capping data for the fixed data: the tangential paths
$e^{J_0\pi t/2}\,T_{\cP{p}}\widehat L^G_0$ and Pin extensions along them
(Appendix~\ref{app:floer-orientations}).  Finally the orientations
\ref{it:F-fixed} of the lines $o_{\cP{p}}$ induce orientations of the
orientation lines $o^G_{\cP{p}}$ of the fixed pair through the splitting
$o_{\cP{p}}=o^G_{\cP{p}}\otimes o^\perp$ of the cap operator at $\cP{p}$
into its tangential and normal parts, the normal line $o^\perp$ being the
same at every invariant point and oriented once
(Appendix~\ref{app:floer-orientations}).  We call $(Q_0,Q_1,E')$, these
capping data and these orientations the \emph{induced orientation data}
of the fixed pair, and by the \emph{Floer theory of the fixed data} we
mean the theory of $(M^G,\widehat L^G_0,\widehat L^G_1,J_0)$ with the
induced orientation data.

\begin{lemma}[The two orientations of fixed strips agree]\label{lem:floer-fixed-strips}
For invariant intersection points $\cP{p},\cP{q}$, the orientation of
$\M_{M^G}(\cP{p},\cP{q})$ induced from $\widehat M$ by
Definition~\ref{conv:floer-orientation-data} and
Convention~\ref{conv:floer-strip-orientation} is the orientation given by
the same convention for the Floer theory of the fixed data.  Consequently $(C^i,\D{i}{i})$ is the Floer
complex of the fixed pair with the induced orientation data, and $H(C^i)$
is $HF(\widehat L_0^G,\widehat L_1^G;\Z)$ for these data.
\end{lemma}

The proof is in Appendix~\ref{app:floer-orientations}.  The idea is that
along a fixed strip the normal operators are $G$-equivariant, hence
complex linear for the group complex structure
(Lemma~\ref{lem:group-complex-structure}), and the normal boundary data
are the same for every strip, so the normal factor of
\eqref{eq:floer-gluing-isomorphism} is a sign independent of the strip;
the constant strip shows that this sign is $+1$.

\subsubsection{Sphere and incidence orientations}

 Let $\descending(\cP{v})$ be the descending manifold of $\cP{v}$ and $\ascending(\cP{w})$ be the ascending manifold of $\cP{w}$.

\begin{lemma}[Boundary orientations on the sphere]
\label{lem:floer-sphere-boundary}
Orient the descending cells as in Lemma~\ref{lem:floer-sphere-data}, with
the positive orientation on every zero-dimensional cell.  Orient a sphere
trajectory space by the exact-sequence conventions
\[
 T\descending(\cP{v})=T(\descending(\cP{v})\cap\ascending(\cP{w}))
       \oplus TS_N/T\ascending(\cP{w}),
\]
\[
 T(\descending(\cP{v})\cap\ascending(\cP{w}))
   =T\M_{h_N}(\cP{v},\cP{w})\oplus\R\dot\vartheta,
\]
where the quotient orientation is transported from
$T_{\cP{w}}\descending(\cP{w})$.
On a codimension-one stratum of the compactified descending cell,
\[
 \M_{h_N}(\cP{v},\cP{w})\times\descending(\cP{w}),
\]
the boundary orientation is $(-1)^{\dim\M_{h_N}(\cP{v},\cP{w})}$
times the product orientation.  The same sign holds for breaking in a
sphere trajectory space.  In particular rigid sphere breaking carries
the product sign.
\end{lemma}

The proof, a computation in a Morse chart at $\cP{w}$, is given in
Subsection~\ref{subsec:IIF-orientations}.  Only the rigid case
$\dim\M_{h_N}(\cP{v},\cP{w})=0$ is used: a one-dimensional trajectory
space compactifies to a compact one-manifold with boundary, and near a
rigid break the closure of a descending cell is a manifold with boundary
(Lemma~\ref{lem:IIF-sheets}).  For the boundary formulas of
finite-dimensional Morse theory compare \cite{hutchings2002morse}.

\begin{convention}[Orientation of incidence spaces]
\label{conv:floer-incidence-orientation}
For a sphere critical point $\cP{v}$, let
\[
 \nu(\cP{v})=TS_N|_{\descending(\cP{v})}/T\descending(\cP{v}),
 \qquad c(\cP{v})=N-1-j(\cP{v}).
\]
Orient it by $TS_N=T\descending(\cP{v})\oplus\nu(\cP{v})$.
Transversality identifies the normal bundle of the incidence space in
$\M^\circ(\cP{p},\cP{b})$ with the pullback of $\nu(\cP{v})$.
We orient the incidence space by
\begin{equation}\label{eq:floer-incidence-orientation}
 T\M^\circ(\cP{p},\cP{b})
  =\Ev^*\nu(\cP{v})\oplus T\mathcal N_{\cP{p},N}(\cP{v};\cP{b}),
\end{equation}
with the normal directions first.  A rigid incidence configuration has
the sign of the induced isomorphism from the ambient tangent space to
$\nu(\cP{v})$.
\end{convention}

For a lifted maximum $M_g$ the bundle $\nu(M_g)$ has rank zero and the
cell carries the orientation of the sphere, so
\eqref{eq:floer-incidence-orientation} returns the orientation of the
strip: the partition \eqref{eq:floer-top-cell-sum} of
Lemma~\ref{lem:floer-top-cells} holds as oriented zero-dimensional spaces,
and summing the incidence counts over the lifted top cells, that is, over
$\Theta_{\cP{p}}$, counts each ordinary rigid non-fixed strip exactly once
with its ordinary sign.

\subsubsection{Equivariance and the transport sign}

\begin{lemma}[The $G$-action preserves all orientations]
\label{lem:floer-equivariant-orientations}
For every $g\in G$:
\begin{enumerate}
    \item the lifted action induces isomorphisms $g_*\colon o_{\cP{x}}\to o_{g\cP{x}}$
    compatible with \eqref{eq:floer-gluing-isomorphism} and preserving the
    orientations of Definition~\ref{conv:floer-orientation-data}; for an
    invariant $\cP{p}$ the action on $o_{\cP{p}}$ is trivial;
    \item $g\colon\M(\cP{x},\cP{y})\to\M(g\cP{x},g\cP{y})$ is an
    orientation-preserving diffeomorphism, and for invariant
    $\cP{p},\cP{q}$ it preserves the decomposition of
    Lemma~\ref{lem:floer-fixed-open-closed}; deck transformations act
    orientation-preservingly on the trajectory spaces of $h_N$;
    \item $g\colon\mathcal N_{\cP{p},N}(\cP{v};\cP{b})\to\mathcal N_{\cP{p},N}(g\cP{v};g\cP{b})$
    preserves orientations, and every block of \eqref{eq:floer-block-matrix}
    is $G$-equivariant, for the permutation action on outside and sphere
    generators and the trivial action on the formal invariant generators.
\end{enumerate}
\end{lemma}

\begin{proof}
The data of Definition~\ref{conv:floer-orientation-data} are
$G$-invariant by construction: the Pin structures and the lifted action by
Lemma~\ref{lem:floer-invariant-pin}, the capping data by \ref{it:F-paths},
the orientations at outside points by transport, and the sphere
orientations by Lemma~\ref{lem:floer-sphere-data}.  At an invariant
$\cP{p}$ the lifted action, which is an action of $G$ because the lifts
are normalized, acts on the cap operator $D_{\cP{p}}$ and hence linearly on
the line $o_{\cP{p}}$; such an action is trivial by
Lemma~\ref{lem:odd-order}.  This gives (1).  Every orientation in (2)
and (3) is obtained from these data by gluing determinant lines along the
isomorphisms \eqref{eq:floer-gluing-isomorphism}, by the fiber-product
convention for incidence spaces, and by the product convention for broken
configurations, each of which is natural with respect to the isomorphisms
induced by the lifted action; so $g$ preserves them.  For the trajectory
spaces of $h_N$ this is the statement that the deck transformations
preserve the lifted orientations of the cells.  Equivariance of the
blocks follows since each block is a signed count of an oriented moduli
space on which $G$ acts preserving the orientation.
\end{proof}

\begin{definition}[Transport sign]\label{def:floer-transport-sign}
Transport of a sphere generator $\cP{v}$ along a rigid fixed strip $u$
is counted with the sign
\[
 (-1)^{c(\cP{v})}\epsilon(u)=(-1)^{N-1-j(\cP{v})}\epsilon(u)
 =(-1)^{j(\cP{v})+1}\epsilon(u),
\]
the last form because $N=rm$ is even.  This is the sign in
\eqref{eq:floer-dinv}: it is $+\epsilon(u)$ on the lifted maxima, whose
index is $N-1$, and $-\epsilon(u)$ on the minima.
\end{definition}

\subsection{Compactification, gluing, and the chain identity}
\label{subsec:floer-dsquared}
This subsection states the compactness and gluing
results, proved in Appendices~\ref{app:floer-compactification}
and~\ref{app:floer-gluing}, and deduces $\partial^2=0$ from them block by
block.  The gluing theorem does more than give compactness: it identifies
each boundary point of a one-dimensional moduli space with exactly one
algebraic term of $\partial^2$.

The four kinds of spaces to be compactified are ordinary strip spaces
$\M(\cP{x},\cP{b})$ with outside source, fixed strip spaces
$\M_{M^G}(\cP{p},\cP{q})$, incidence spaces
$\mathcal N_{\cP{p},N}(\cP{v};\cP{b})$, and auxiliary sphere trajectory
spaces $\M_{h_N}(\cP{v},\cP{w})$.  Superscripts denote their
actual dimensions; in particular $\M^d$ consists of index-$(d+1)$
components, whereas \eqref{eq:floer-incidence-dimension} specifies the
index for $\mathcal N^d$.

\begin{proposition}[Sequential compactness and evaluation]
\label{thm:floer-compactification-cited}
For the adapted data and the chosen spectral cutoff, a sequence of strips
with prescribed ends has a subsequence converging to a broken strip with
finitely many components, with no bubbling.  A sequence of sphere trajectories or half-trajectories
has a subsequence converging in the corresponding broken topology.
For a sequence of incidence configurations, these convergences are
compatible: its evaluation converges to the evaluation of the first
non-fixed strip of the limit, carried across the preceding fixed strips
by the identity of the spectral spheres.  A limit of non-fixed strips has
at least one non-fixed component; a limit of fixed strips consists only of
fixed strips.
\end{proposition}

\begin{theorem}[Compactification and gluing]
\label{thm:floer-gluing-package}
Every zero-dimensional space of these four kinds is finite, and every
one-dimensional space has a compactification as a compact oriented
one-manifold with boundary.  Its boundary consists precisely of the
following products of zero-dimensional spaces.
\begin{enumerate}[label=\textup{(S\arabic*)}]
\item\label{it:S1} \emph{Outside breaking.}  At an outside intersection
$\cP{y}$, the boundary points of an ordinary space and of an incidence space are,
respectively,
\[
 \M^0(\cP{x},\cP{y})\times\M^0(\cP{y},\cP{b}),\qquad
 \mathcal N^0_{\cP{p},N}(\cP{v};\cP{y})\times\M^0(\cP{y},\cP{b}).
\]
For the latter the limiting evaluation is that of the incoming factor.
\item\label{it:S2} \emph{Sphere breaking.}  For an incidence space the boundary points
are
\[
 \M_{h_N}^0(\cP{v},\cP{w})
 \times\mathcal N^0_{\cP{p},N}(\cP{w};\cP{b}),
 \qquad j(\cP{w})=j(\cP{v})-1.
\]
An ordinary one-dimensional sphere trajectory space has its usual broken
sphere trajectories as boundary.
\item\label{it:S3} \emph{Fixed breaking.}  For invariant $\cP{p},\cP{b}$, a fixed strip space has boundary
\[
 \coprod_{\cP{q}\text{ invariant}}
 \M_{M^G}^0(\cP{p},\cP{q})\times\M_{M^G}^0(\cP{q},\cP{b}).
\]
\item\label{it:S4} \emph{Invariant intermediate intersection.}  At an invariant
$\cP{q}$ there are three mixed possibilities.  
\begin{enumerate}[label=\textup{(\alph*)}]
\item\label{it:S4a} \emph{A non-fixed strip followed by a fixed strip.}  The boundary points are
\[
 \M^0(\cP{x},\cP{q})\times\M_{M^G}^0(\cP{q},\cP{b}),\qquad
 \mathcal N^0_{\cP{p},N}(\cP{v};\cP{q})
       \times\M_{M^G}^0(\cP{q},\cP{b}),
\]
for an outside source and an incidence source, respectively.  Then
$\cP{b}$ is invariant, and for an incidence source the evaluation is that
of the incoming strip.
\item\label{it:S4b} \emph{A fixed strip followed by a non-fixed strip.}  The boundary
points of an incidence space are
\[
 \M_{M^G}^0(\cP{p},\cP{q})
       \times\mathcal N^0_{\cP{q},N}(\cP{v};\cP{b}).
\]
The source incidence is transported to $\cP{q}$ by the identity.  The
fixed factor carries the sign of
Definition~\ref{def:floer-transport-sign}.
\item\label{it:S4c} \emph{Two non-fixed strips.}  Partition the ordinary
rigid outgoing space by its top-cell incidences
(Lemma~\ref{lem:floer-top-cells}):
\[
 \M^{\circ,0}(\cP{q},\cP{b})
 =\coprod_{g\in G}\mathcal N^0_{\cP{q},N}(M_g;\cP{b}).
\]
The boundary points are its products with $\M^0(\cP{x},\cP{q})$ or
$\mathcal N^0_{\cP{p},N}(\cP{v};\cP{q})$, respectively.  
\end{enumerate}
\end{enumerate}
For each of the finite boundary strata $Z$ above there is a gluing map
\[
 \mathcal G_Z:Z\times[T_1,\infty)\longrightarrow\mathcal X_Z
\]
into the one-dimensional space $\mathcal X_Z$ whose boundary contains $Z$.  It is a homeomorphism onto its
image and extends at $T=\infty$ by the inclusion of $Z$ in the boundary;
the extension is a collar of $Z$.  The maps can be chosen equivariantly
on the union of group translates of the strata.  The boundary signs are
those of Lemma~\ref{lem:floer-boundary-formulas} for strip spaces, of
Lemma~\ref{lem:floer-sphere-boundary} for sphere trajectory spaces, and
of Lemma~\ref{lem:floer-boundary-signs} for incidence spaces.
\end{theorem}

\begin{lemma}[Incidence boundary signs]\label{lem:floer-boundary-signs}
Let $(z_1,z_2)$ be a boundary point of a one-dimensional incidence space,
and let $\epsilon(z_i)$ be the sign of the rigid factor $z_i$: of a strip
(Convention~\ref{conv:floer-strip-orientation}), of an incidence
configuration (Convention~\ref{conv:floer-incidence-orientation}), or of a
sphere trajectory (Lemma~\ref{lem:floer-sphere-boundary}).  With the
outward direction that of increasing neck length, the boundary sign is
$\epsilon(z_1)\epsilon(z_2)$, except in type \textup{(S4b)}, where it is
$(-1)^{N-1-j(\cP{v})}\epsilon(z_1)\epsilon(z_2)$.
\end{lemma}
The determinant-line and normal-first sign calculation is given in
Subsection~\ref{subsec:IIF-orientations}, after the collars have been
constructed.

\subsubsection{The block identities}
For the incidence space $\mathcal N^1_{\cP{p},N}(\cP{v};\cP{q})$ with
invariant target, the boundary types of
Theorem~\ref{thm:floer-gluing-package} contribute to $\partial^2$ as
follows; compositions are read from right to left, and the five terms are
those of \eqref{eq:floer-square-incidence} below.
\begin{center}
\renewcommand{\arraystretch}{1.15}
\begin{tabular}{lll}
\hline
boundary configuration & type & contribution to $\partial^2$\\
\hline
break at an outside intersection point & \textup{(S1)} & $\D{o}{s}\D{s}{o}$\\
break of the source sphere half-trajectory & \textup{(S2)} & $\Dnon\dsph$\\
a fixed strip followed by a non-fixed strip & \textup{(S4b)} & $\Dnon\Dinv$\\
a non-fixed strip followed by a fixed strip & \textup{(S4a)} & $\Dinv\Dnon$\\
two non-fixed strips & \textup{(S4c)} & $\Dnon\Dnon$\\
\hline
\end{tabular}
\end{center}
The remaining blocks of $\partial^2$ are handled by the ordinary-strip
identities and by the local and fixed identities of
Lemma~\ref{lem:floer-fixed-identities}.

The local relations \eqref{eq:floer-local-identities} imply
\begin{equation}\label{eq:floer-support-identities}
\begin{gathered}
 \dsph\Theta=0,\qquad\D{s}{i}\Theta=0,\qquad
 \D{s}{i}\dsph=0,\\
 \dsph\D{o}{s}=\dsph\Dnon=0,\qquad
 \D{s}{i}\D{o}{s}=\D{s}{i}\Dnon=0.
\end{gathered}
\end{equation}
The last line uses that both $\D{o}{s}$ and $\Dnon$ land in top cycles.

\begin{lemma}[Identities of the local and fixed blocks]\label{lem:floer-fixed-identities}
The operator on $C^s\oplus C^i$ consisting only of the local sphere,
augmentation, and fixed-strip terms squares to zero, that is,
\[
 (\dsph+\Dinv)^2=0,\qquad(\D{i}{i})^2=0,
 \qquad \D{s}{i}\Dinv+\D{i}{i}\D{s}{i}=0;
\]
in fact $(\Dinv)^2=0$ and $\dsph\Dinv+\Dinv\dsph=0$ separately.
\end{lemma}
\begin{proof}
 The fixed Floer differential
squares to zero by ordinary fixed-strip breaking, and applying the
transport sign twice multiplies its square by
$(-1)^{2c(\cP{v})}=1$.  Thus $(\Dinv)^2=(\D{i}{i})^2=0$.

The two orders of a sphere step and a fixed strip have opposite signs,
because a sphere step decreases $j$ by one.  Hence
$\dsph\Dinv+\Dinv\dsph=0$.

The augmentation kills the sphere differential and is supported at
$j=0$, where the transport sign is $(-1)^{N-1}=-1$.  It therefore
anticommutes with the fixed differential giving $\D{s}{i}\Dinv+\D{i}{i}\D{s}{i}=0$.
\end{proof}

\begin{proposition}\label{prop:floer-differential-squares-zero}
The stabilized operator satisfies $\partial^2=0$ over $\Z$.
\end{proposition}
\begin{proof}
Multiplying the matrix \eqref{eq:floer-block-matrix} leaves seven
potentially nonzero blocks:
\begin{align}
 (\partial^2)_o^o&=\D{o}{o}\D{o}{o}+\D{s}{o}\D{o}{s},&
 (\partial^2)_o^s&=\D{o}{s}\D{o}{o}+\D{s}{s}\D{o}{s},\label{eq:floer-square-outside}\\
 (\partial^2)_o^i&=\D{s}{i}\D{o}{s},&
 (\partial^2)_s^o&=\D{o}{o}\D{s}{o}+\D{s}{o}\D{s}{s},\label{eq:floer-square-mixed}\\
 (\partial^2)_s^s&=\D{o}{s}\D{s}{o}+(\D{s}{s})^2,&
 (\partial^2)_s^i&=\D{s}{i}\D{s}{s}+\D{i}{i}\D{s}{i},\label{eq:floer-square-sphere}\\
 (\partial^2)_i^i&=(\D{i}{i})^2.&&\label{eq:floer-square-fixed}
\end{align}

\emph{Outside source.}
The boundary of $\M^1(\cP{x},\cP{z})$ for outside endpoints consists of
outside breaks and invariant breaks with two non-fixed factors.  The first
type contributes $\D{o}{o}\D{o}{o}$; the top-cell partition at an invariant
break gives $\D{s}{o}\D{o}{s}$.  Its signed boundary sum is zero.  For
an invariant final target, the outside break and the two possibilities for
the outgoing strip at an invariant break give
\[
 \D{o}{s}\D{o}{o}+(\Dinv+\Dnon)\D{o}{s}=0.
\]
Together with \eqref{eq:floer-support-identities}, these are the $o\to o$
and $o\to s$ identities.  The $o\to i$ identity is already
$\D{s}{i}\D{o}{s}=0$.

\emph{Incidence source and outside target.}
The four boundary types of
$\mathcal N^1_{\cP{p},N}(\cP{v};\cP{z})$ are an outside break, a sphere
break, a fixed incoming strip, and two non-fixed strips breaking at an
invariant point.  Their oriented sum, with the transport sign on the fixed
incoming factor, is
\[
 \D{o}{o}\D{s}{o}
 +\D{s}{o}\dsph+\D{s}{o}\Dinv+\D{s}{o}\Dnon=0.
\]
This is the $s\to o$ block.

\emph{Incidence source and invariant target.}
For $\mathcal N^1_{\cP{p},N}(\cP{v};\cP{q})$, the outside break, sphere
break, fixed incoming, fixed outgoing, and two-non-fixed break give
\begin{equation}\label{eq:floer-square-incidence}
 \D{o}{s}\D{s}{o}+\Dnon\dsph
 +\Dnon\Dinv+\Dinv\Dnon+\Dnon\Dnon=0.
\end{equation}
The term $\dsph\Dnon$ vanishes because $\Theta_{\cP{q}}$ is a cycle,
\eqref{eq:floer-support-identities}.
Adding $(\dsph+\Dinv)^2=0$ from
Lemma~\ref{lem:floer-fixed-identities}, and using
$\dsph\Dnon=0$, gives the $s\to s$ block.

\emph{Augmentation and fixed source.}
The $s\to i$ block reduces to
$\D{s}{i}\Dinv+\D{i}{i}\D{s}{i}=0$ by
\eqref{eq:floer-support-identities}; this is
Lemma~\ref{lem:floer-fixed-identities}.  Finally $(\D{i}{i})^2=0$ is the
ordinary fixed-pair Floer identity, with boundary described by
\textup{(S3)}.  The two remaining source-$i$ blocks vanish identically
from the zero entries of the matrix.  Theorem~\ref{thm:floer-gluing-package}
exhausts the boundary types, so no block receives a further contribution.
\end{proof}

\subsection{Comparison with the ordinary Floer complex}
\label{subsec:floer-localization}
Let $C^f$ be the free module on the invariant intersections, retaining
their ordinary Floer generators, so that
$CF(\widehat L_0,\widehat L_1)=C^o\oplus C^f$.  Write its differential as
\[
 d_{\mathrm F}=\begin{pmatrix}d_o^o&d_f^o\\d_o^f&d_f^f\end{pmatrix}.
\]
Define $\jmath:C^f\to C^s$ by $\jmath(\cP{p})=\Theta_{\cP{p}}$ and
\begin{equation}\label{eq:floer-comparison-map}
 F:CF(\widehat L_0,\widehat L_1)\longrightarrow
       CF_{\mathrm{st}}(\widehat L_0,\widehat L_1),\qquad
 F(\cP{x})=\cP{x},\quad F(\cP{p})=\Theta_{\cP{p}}.
\end{equation}

\begin{lemma}[Comparison on top cells]\label{lem:floer-top-cell-comparison}
The map $F$ commutes with the differentials.  More precisely,
\[
 d_o^o=\D{o}{o},\quad \jmath d_o^f=\D{o}{s},\quad
 \D{s}{o}\jmath=d_f^o,\quad
 \D{s}{s}\jmath=\jmath d_f^f,\quad \D{s}{i}\jmath=0.
\]
\end{lemma}
\begin{proof}
The first two equalities are the definitions of the outside-source
blocks.  For a non-fixed rigid strip leaving an invariant point,
Lemma~\ref{lem:floer-top-cells} partitions its evaluation by the lifted
top cells; summing their incidence counts counts the strip once with its
ordinary sign.  This gives $\D{s}{o}\jmath=d_f^o$ and the non-fixed part
of $\D{s}{s}\jmath=\jmath d_f^f$.  A fixed strip transports the top cycle
by the identity with transport sign $+1$ on the top cells, and it enters
$\D{s}{s}$ and $d_f^f$ with the same sign $\epsilon(u)$, which is
unambiguous by Lemma~\ref{lem:floer-fixed-open-closed}.  Finally the sphere differential
and augmentation both kill $\Theta$.  These prove all five identities,
which are precisely $\partial F=Fd_{\mathrm F}$.
\end{proof}

The theorem says that the stabilization does not change Floer homology.

\begin{theorem}[Floer comparison]\label{thm:floer-F-quasi-isomorphism}
The map $F$ is a quasi-isomorphism over $\Z$, and after extension to any
coefficient ring.
\end{theorem}
\begin{proof}
Filter both differential modules by the finitely many values of the action
\eqref{eq:floer-action} for the adapted pair.  Assign every sphere or
formal generator at $\cP{p}$ the action of $\cP{p}$.  Each strip term
strictly lowers action.  The only terms preserving action are the local
sphere differential and augmentation.  Consequently the associated graded
stabilized module is a direct sum of the outside generators and the local
modules $\mathcal T_{\cP{p}}$.  The associated graded ordinary module has
zero differential.

By Lemma~\ref{lem:floer-augmented-sphere}, $F$ induces an isomorphism on
\emph{the homology} of every associated graded piece.  Induction over the
finite filtration, using the exact homology triangles of successive
short exact sequences of differential modules, proves that $F$ is a
quasi-isomorphism.  Both the local homology calculation and the filtered
argument hold over any coefficient ring, proving the base-change claim.
\end{proof}

\subsection{The Smith inequality}
\label{subsec:floer-smith}

This subsection records the equivariant structure of the stabilized
complex and deduces the Smith inequality for $(L_0,L_1)$ and its fixed pair
from the algebraic lemma proved below.

\begin{lemma}[Subcomplex and free quotient]\label{lem:floer-module-structure}
With the orientations above:
\begin{enumerate}
    \item the differential of $CF_{\mathrm{st}}(\widehat L_0,\widehat L_1)$ is $G$-equivariant;
    \item $C^i$ is a subcomplex with trivial $G$-action, equal to the
    Floer differential module of the fixed data with the orientation system of
    Lemma~\ref{lem:floer-fixed-strips}, that of the induced relative Pin
    structures $Q_0,Q_1$, and
    \[
        H(C^i\otimes\F_p)
        \cong HF(L_0^G,L_1^G;\F_p);
    \]
    \item the quotient $CF_{\mathrm{st}}(\widehat L_0,\widehat L_1)/C^i$, which is
    $C^o\oplus C^s$ as a $\Z[G]$-module, is a complex of free
    $\F_p[G]$-modules after tensoring with $\F_p$.
\end{enumerate}
\end{lemma}

\begin{proof}
The group permutes strips and incidence configurations and preserves all
orientations (Lemma~\ref{lem:floer-equivariant-orientations}).  The blocks
from $C^i$ to $C^o$ and $C^s$ vanish, so $C^i$ is a subcomplex and the
restriction of $\partial$ to it is the Floer differential of the fixed pair;
the augmentation $\D{s}{i}$ maps $C^s$ into $C^i$, so $C^o\oplus C^s$ is
not a subcomplex, and it is the quotient that carries a differential.
Outside points and all critical points of the lifted local Morse function
lie in free $G$-orbits, making $C^o$ and $C^s$ free permutation modules.
\end{proof}

\begin{lemma}[Algebraic Smith inequality; Lemma~\ref{lem:intro-algebraic-smith}]\label{lem:algebraic-smith}
Let $k=\F_p$, $G=\Z_p$, $R=k[G]$.  Let $C$ be a finitely generated
$R$-module with an $R$-linear endomorphism $\partial$ satisfying
$\partial^2=0$, and let $C'\subset C$ be a $\partial$-invariant submodule on
which $G$ acts trivially and such that $C/C'$ is a free $R$-module.  Then
\[
    \dim_kH(C')\;\le\;\dim_kH(C),
    \qquad
    H(C):=\ker\partial/\operatorname{im}\partial .
\]
\end{lemma}

For the cellular chains of a finite $G$-CW complex $X$ with $C'=C(X^G)$,
summed over all degrees, this is the Smith inequality
\cite{smith1938transformations,floyd1952on}.  The proof uses only that a
short exact sequence $0\to A\to B\to C\to0$ of differential modules
induces an exact triangle $H(A)\to H(B)\to H(C)\to H(A)$.

\begin{proof}
Let $\sigma$ generate $G$ and put $e=\sigma-1$, $\mathsf n=1+\sigma+\cdots+\sigma^{p-1}$,
so that $e\mathsf n=\mathsf n e=0$; on a free $R$-module $\ker e=\operatorname{im}\mathsf n$ and
$\ker \mathsf n=\operatorname{im}e$, and both $e$ and $\mathsf n$ kill $C'$.  Since $C/C'$
is free, hence projective, $C=C'\oplus F$ for some free submodule $F$, and
for $\rho\in\{e,\mathsf n\}$ with $\bar\rho$ the other element,
\[
    \rho C=\rho F,\qquad
    \ker(\rho\colon C\to C)=C'\oplus\bar\rho C .
\]
As $\partial$ is $R$-linear, $\rho C$ and $\bar\rho C$ are
$\partial$-invariant, and
\begin{equation}\label{eq:smith-sequence}
    0\longrightarrow C'\oplus\bar\rho C\longrightarrow C
    \xrightarrow{\;\rho\;}\rho C\longrightarrow0
\end{equation}
is a short exact sequence of differential modules.  Write $b=\dim H(C)$,
$c=\dim H(C')$, and $a^\rho=\dim H(\rho C)$.  Exactness of the triangle of
\eqref{eq:smith-sequence} at $H(C'\oplus\bar\rho C)=H(C')\oplus H(\bar\rho C)$
gives
\[
    c+a^{\bar\rho}
    =\operatorname{rank}\bigl(H(\rho C)\to H(C')\oplus H(\bar\rho C)\bigr)
    +\operatorname{rank}\bigl(H(C')\oplus H(\bar\rho C)\to H(C)\bigr)
    \le a^{\rho}+b .
\]
This holds for both $\rho=e$ and $\rho=\mathsf n$; adding the two inequalities,
$2c+a^e+a^{\mathsf n}\le a^e+a^{\mathsf n}+2b$, so $c\le b$.  (For $p=2$ one has $e=\mathsf n$ and a
single inequality suffices.)
\end{proof}

\begin{theorem}[Smith inequality for Floer homology]
\label{thm:floer-smith}
Let $(M,\omega=d\theta)$, $G=\Z_p$, and $L_0\pitchfork L_1$ satisfy the
hypotheses of Subsection~\ref{subsec:lagrangian-floer-setup} and
Assumptions~\ref{ass:floer-normal-triviality} and~\ref{ass:floer-pin}.  Then
\[
    \dim_{\F_p}HF(L_0,L_1;\F_p)
    \;\ge\;
    \dim_{\F_p}HF(L_0^G,L_1^G;\F_p),
\]
where the right side is $H(C^i\otimes\F_p)$, the homology of the fixed
complex with the orientations induced from $\widehat M$, which by
Lemma~\ref{lem:floer-fixed-strips} is the Floer homology of the fixed
pair for the induced relative Pin structures $Q_0,Q_1$ with respect to the
background bundle $E'$ of \eqref{eq:floer-fixed-background}.
\end{theorem}

\begin{proof}
Apply Lemma~\ref{lem:algebraic-smith} to the differential modules
\[
    C=CF_{\mathrm{st}}(\widehat L_0,\widehat L_1)\otimes\F_p,
    \qquad
    C'=C^i\otimes\F_p;
\]
Lemma~\ref{lem:floer-module-structure} verifies its hypotheses and
identifies $H(C')$ with the right side, and
Theorem~\ref{thm:floer-F-quasi-isomorphism}, tensored with $\F_p$,
identifies $H(C)$ with $HF(L_0,L_1;\F_p)$.
\end{proof}

\section{Periodic links and symplectic Khovanov homology}\label{sec:applications}

In this section we apply the localization theorem to $p$-periodic links,
generalizing Seidel--Smith \cite[Corollary~3]{seidel2010localization} from
$p=2$ to odd primes.  The section is organized as follows: we recall symplectic Khovanov homology, which is the Floer
homology of two matching Lagrangians in a nilpotent slice
(Subsection~\ref{subsec:matching-lagrangians}); we show that for a $p$-periodic link, there is
a $\Z_p$-action on the slice preserving both Lagrangians, whose fixed locus
and fixed Lagrangians are the corresponding data of the quotient link
(Subsection~\ref{subsec:cyclic-action}); we show that the normal bundle of the fixed locus
satisfies the stable triviality hypothesis and the Lagrangians carry the
required relative Pin data (Subsection~\ref{subsec:application-normal-data}); and
therefore Theorem~\ref{thm:floer-smith} gives the rank inequality
(Subsection~\ref{subsec:application-inequality}).

\subsection{The slice and matching Lagrangians}\label{subsec:matching-lagrangians}

We first give an overview of symplectic Khovanov homology. 
A crossingless matching $\frm$ of $2m$ points in
the plane is a collection of $m$ pairwise disjoint paths in the plane whose endpoints are the $2m$ points. 
Any link $K$ can be presented as the plat closure of a braid on $2m$ strands, or
equivalently by two crossingless matchings $\frm_0,\frm_1$ of $2m$ points in
the plane related by the braid. 
The crossingless matchings $\frm_0,\frm_1$ are allowed to intersect, and wherever there is an intersection, one should think of the points in $\frm_0$ as going under the points in $\frm_1$ so that the union of $\frm_0$ and $\frm_1$ forms the link $K$.
We will define a symplectic manifold $\cY_\tau$, which is a fiber of the adjoint quotient map restricted
to a nilpotent slice, and to each crossingless matching $\frm$, we associate a
Lagrangian $L_\frm$ in $\cY_\tau$.
The symplectic Khovanov homology of $K$ is defined to be
\[
    \Kh(K):=HF(L_{\frm_0},L_{\frm_1}).
\]
This subsection recalls the slice, its fibers, and the matching
Lagrangians, so that we can explain how to exhibit the periodic
symmetry as a $\Z_p$-action on these data in the remaining subsections.

Let $\frgl_{2m}=\frgl_{2m}(\mathbb{C})$ be the space of $2m \times 2m$ matrices.
The general linear group $GL(2m)$ acts on $\frgl_{2m}$ by conjugation (i.e., the adjoint action).
The adjoint quotient map\footnote{In Lie-theoretic language, the target is the Cartan subalgebra of $\frgl_{2m}$ quotiented by the Weyl group.}
\[
\chi_{\frgl}:\frgl_{2m} \to \mathbb{C}^{2m}
\]
sends a matrix $A \in \frgl_{2m}$ to the coefficients of its characteristic polynomial.
Let $\Sym^{2m}(\mathbb{C})=\mathbb{C}^{2m}/\Sym_{2m}$ be the $2m$-fold symmetric product.
We have an isomorphism
\begin{align}\label{eq:taut_sym}
\Sym^{2m}(\mathbb{C}) \simeq \mathbb{C}^{2m}
\end{align}
given by sending a tuple to its elementary symmetric polynomials.
We use $\Sym^{2m}(\mathbb{C})$ and $ \mathbb{C}^{2m}$ interchangeably by \eqref{eq:taut_sym}.
In particular, we also interpret $\chi_{\frgl}$ as sending a matrix to its eigenvalues.

Let $\Conf^{2m}(\mathbb{C})$ be the configuration space of $2m$ points in $\mathbb{C}$ (i.e., the subset of $ \Sym^{2m}(\mathbb{C})$ consisting of tuples with pairwise distinct entries).
It is the set of regular values of $\chi_{\frgl}$.

Consider the following affine slice of $\frgl_{2m}$.
\begin{align}\label{eq:S_m}
S_m=\left\{ A=\left(
\begin{array}{cccccc}
   A_1  &  I & 0\\
   A_2  &  0 & I\\
   \cdots & & &\cdots\\
   A_{m-1} & &  &  &I\\
   A_{m}& 0& \cdots &\cdots&0
\end{array}
\right) :  A_j \in \frgl_2 \text{ for all }j\right\}.
\end{align}
We sometimes denote $A$ by $(A_1,\dots,A_m)$.
Define $\chi:=\chi_{\frgl}|_{S_m}$.
For any $\tau \in \Sym^{2m}(\mathbb{C})$, we define
\[
\cY_{\tau}:= \chi^{-1}(\tau) = \chi^{-1}_{\frgl}(\tau) \cap S_m.
\]

Let $\alpha>m$ be a real number. Let $\varrho_i(z):=|z|^{\frac{2\alpha}{i}}$.
Let $\eta_i:\mathbb{C} \to \mathbb{R}$ be a compactly supported function such that $\psi_i:=\varrho_i+\eta_i$ is $C^{\infty}$ and $-dd^c \psi_i>0$ everywhere.
Let $\psi:S_m \to \mathbb{R}$ be the function given by summing $\psi_i(z)$ over all entries $z$ in $A_i$ and all $i=1,\dots,m$.
Then $\psi$ defines a K\"ahler potential with K\"ahler form $-dd^c \psi$ on $S_m$.
When $\tau \in \Conf^{2m}(\mathbb{C})$, $\cY_{\tau}$ is a smooth affine variety with the restricted K\"ahler form.

Given a path $\gamma:[0,1] \to \Conf^{2m}(\mathbb{C})$ and a relatively compact open set $U \subset \cY_{\gamma(0)}$, Seidel--Smith defined a rescaled parallel transport map \cite[Section~4]{seidel2006link}, which is a symplectic embedding
\[
h^{resc}_{\gamma}: U \to \cY_{\gamma(1)} .
\]
There are choices involved in the construction of $h^{resc}_{\gamma}$, and different choices are related by an exact isotopy.

Let $0<\mu_1< \dots <\mu_m$.
Let $\epsilon>0$ be a number smaller than $|\mu_i-\mu_j|/2$ for all $i \neq j$.
Let $D_{\epsilon}$ be the polydisk
\begin{align}\label{eq:disc}
D_{\epsilon}:=
\{(\mu_1 \pm \sqrt{\epsilon_1}, \dots, \mu_m \pm \sqrt{\epsilon_m}): \epsilon_i \in \mathbb{C}, |\epsilon_i| \le \epsilon^2 \text{ for all }i\} \subset \Sym^{2m}(\mathbb{C}).
\end{align}
There is a local isomorphism near the minimal adjoint orbit $\mathcal O^{min}$ (the unique fixed point) in $\chi^{-1}(D_{\epsilon})$
\begin{equation}\label{eq:local_iso}
	\begin{tikzcd}[row sep=large, column sep=large]
     \chi^{-1}(D_{\epsilon}) \arrow [d, "\chi"] \arrow[rrr, "\text{local isom near }\mathcal O^{min} \simeq 0"] & & & (\mathbb{C}^3)^m  \arrow [d, "f"]\\
     D_{\epsilon}  \arrow[rrr ] & & &\mathbb{C}^m_{\epsilon_1,\dots,\epsilon_m}
	\end{tikzcd}
\end{equation}
where $f(a_1,b_1,c_1,\dots,a_m,b_m,c_m)=(a_1^2+b_1^2+c_1^2, \dots, a_m^2+b_m^2+c_m^2)$.

Let $0<\delta<\epsilon^2$ and define $\tau_\delta=(\mu_1 \pm \sqrt{\delta},\dots, \mu_m\pm \sqrt{\delta}) \in D_{\epsilon}$.
For $i=1,\dots,m$, let $\gamma_{std,i}:[0,1] \to \mathbb{C}$ be the path $(\mu_i-\sqrt{\delta})+t(2\sqrt{\delta})$ so its two endpoints are $\mu_i \pm \sqrt{\delta}$.
Let $\frm_{std}$ be the crossingless matching $\{\gamma_{std,1},\dots,\gamma_{std,m}\}$ with endpoints on $\tau_{\delta}$, and call it the standard crossingless matching.
There is a Lagrangian submanifold $L_{\frm_{std}} \subset \chi^{-1}(\tau_\delta) =\cY_{\tau_\delta}$ which, under the isomorphism \eqref{eq:local_iso}, becomes
\[
\{a_i^2+b_i^2+c_i^2=\delta, a_i,b_i,c_i \in \mathbb{R}\} \subset f^{-1}(\delta,\dots,\delta).
\]
We call $L_{\frm_{std}}$ the standard Lagrangian.

Let $\tau \in \Conf^{2m}(\mathbb{C})$ and let $\frm$ be a crossingless matching with endpoints on $\tau$.
Then Seidel--Smith constructed a Lagrangian $L_{\frm}$ in $\cY_\tau$ as follows.
Write $\frm$ as $\{\gamma_1,\dots,\gamma_m\}$.
Consider the path $\Gamma: [0,1) \to \Conf^{2m}(\mathbb{C})$ given by $\Gamma(t)=\{\gamma_1(t/2), \gamma_1(1-t/2), \dots, \gamma_m(t/2), \gamma_m(1-t/2)\}$.
Let $\mu_i:=\gamma_i(1/2)$.
When $t$ is close to $1$, $\Gamma(t)$ lies in $D_{\epsilon}$ for our chosen $\mu_i$.
We can perturb $\gamma_i$ such that $\Gamma(t_0)=\tau_{\delta}$ for some small $\delta$ and $t_0$ close to $1$.
Let $h_\Gamma^{resc}$ be the rescaled parallel transport from $\cY_{\Gamma(t_0)}$ to $\cY_{\Gamma(0)}$ along $\Gamma(t_0-t)$.
We define
\[
L_\frm:=h_\Gamma^{resc}(L_{\frm_{std}}),
\]
which is well defined up to Hamiltonian isotopy.

For a link $K$, there is $m$ such that we can present $K$ as the union of crossingless matchings $\frm_0:=\frm_{std}$ and $\frm_1$ (in particular $\frm_0$ and $\frm_1$ have the same endpoints).
The symplectic Khovanov homology is defined to be
\[
\Kh(K):=HF(L_{\frm_0},L_{\frm_1}).
\]
It is well-defined and independent of the choices of presentation of $K$.

\subsection{The cyclic action and fixed locus}\label{subsec:cyclic-action}

Consider the $\mathbb{Z}_p$-action on $S_{pm}$ given by $\xi \cdot (A_1,\dots,A_{pm}) =(\xi A_1,\dots,\xi^{pm}A_{pm})$, or in matrix form,
\begin{align}\label{eq:S_action}
\xi \cdot A=\xi \cdot \left(
\begin{array}{cccccc}
   A_1  &  I & 0\\
   A_2  &  0 & I\\
   \cdots & & &\cdots\\
   A_{pm-1} & &  &  &I\\
   A_{pm}& 0& \cdots &\cdots&0
\end{array}
\right) := \left(
\begin{array}{cccccc}
 \xi  A_1  &  I & 0\\
 \xi^2  A_2  &  0 & I\\
   \cdots & & &\cdots\\
  \xi^{pm-1} A_{pm-1} & &  &  &I\\
  \xi^{pm} A_{pm}& 0& \cdots &\cdots&0
\end{array}
\right).
\end{align}
The fixed locus $S_{pm}^{G}$ consists of the elements with $A_j=0$ for all $j\not\equiv 0 \pmod p$.

Consider the standard $\mathbb{Z}_p$-action on $\mathbb{C}$ given by $\xi \cdot z=\xi z$, and the induced diagonal action on $\Sym^{2pm}(\mathbb{C})$.

\begin{lemma}
   Suppose that $\tau \in (\Sym^{2pm}(\mathbb{C}))^{G}$.  Then the $\mathbb{Z}_p$-action on $S_{pm}$ restricts to a $\mathbb{Z}_p$-action on $\cY_{\tau}$.
\end{lemma}

\begin{proof}
    The characteristic polynomial of $A \in S_{pm}$ is
    \[
    \det(t^{pm}-A_1 t^{pm-1}-A_2 t^{pm-2} -\dots-A_{pm}).
    \]
    Therefore, the characteristic polynomial of $\xi \cdot A$ is
    \begin{align*}
    &\det(t^{pm}-\xi A_1 t^{pm-1}-\xi^2 A_2 t^{pm-2} -\dots-\xi^{pm}A_{pm})\\
    =&\det((\xi^{-1} t)^{pm}- A_1 (\xi^{-1} t)^{pm-1} -\dots-A_{pm}).
    \end{align*}
    Since $\tau$ is chosen to be $\mathbb{Z}_p$-invariant, if $A \in \chi^{-1}(\tau)$, then $\xi \cdot A \in \chi^{-1}(\tau)$.
\end{proof}

We now verify the standing hypotheses of Subsection~\ref{subsec:floer-setup}.
The K\"ahler form is $\omega=-dd^c\Psi$ for the potential $\Psi$ built from the
$\psi_i=\varrho_i+\eta_i$ above; the radial part $\varrho_i$ is invariant
under $z\mapsto\xi^jz$, and we choose the compactly supported $\eta_i$
radial as well, so that $\Psi$, hence the primitive $\theta=-d^c\Psi$, is
$\Z_p$-invariant: the action is by $\theta$-preserving symplectomorphisms.
Convexity at infinity holds by the construction of the rescaled metric in
\cite[Section~5]{seidel2010localization}.  The matching Lagrangians are
compact, diffeomorphic to $(S^2)^{pm}$, so there is no intersection along the
ends; they are exact since $H^1((S^2)^{pm};\R)=0$.  Transversality of
$L_{\frm_0}$ and $L_{\frm_1}$ is achieved by an equivariant Hamiltonian
perturbation of $L_{\frm_1}$: a generic $\Z_p$-invariant Hamiltonian makes
the fixed pair transverse inside $\cY_\tau^G$, and at an invariant
intersection point transversality in the normal directions is a condition
on the invariant symmetric part of its Hessian, which is also generic; away
from the fixed locus the action is free and ordinary genericity applies.
By construction, the symplectic form $\omega_{S_{pm}}$ on $S_{pm}$ is $\mathbb{Z}_p$-equivariant.
Therefore, the rescaled parallel transport is $\mathbb{Z}_p$-equivariant.
Moreover, for any equivariant crossingless matching $\frm$ (in the sense that $\xi \cdot \frm=\frm$), the crossingless matching Lagrangian $L_{\frm}$ is $\mathbb{Z}_p$-equivariant.

Now, we analyze the fixed locus.
The $p$-th power map defines an isomorphism
\[
(\Sym^{2pm}(\mathbb{C}))^{G} \simeq \Sym^{2m}(\mathbb{C}),
\]
so for $\tau \in (\Sym^{2pm}(\mathbb{C}))^{G}$, we denote its image under this isomorphism by $\tau^p \in \Sym^{2m}(\mathbb{C})$.

\begin{lemma}\label{l:fixed set symp}
    The fixed locus $\cY_{\tau}^{G}$ is isomorphic to $\cY_{\tau^p}$.
\end{lemma}

\begin{proof}
    We define a map $f:\cY_{\tau^p} \to S_{pm}^{G}$ given by
    \[
    f(A)=f(A_1, \dots, A_m) =(0, \dots,0, A_1, 0, \dots,0,A_m),
    \]
    where there are $p-1$ zeros in front of each $A_j$.
    The characteristic polynomial of $f(A)$ is
    \begin{align}\label{eq:char_p}
    \det(t^{pm}-A_1 t^{p(m-1)}-\dots -A_m)=\det((t^p)^m-A_1 (t^p)^{(m-1)}-\dots -A_m).
    \end{align}
    Therefore, the eigenvalues of $f(A)$ are exactly the $p$-th roots of the eigenvalues of $A$.
    Since the $p$-th roots of the entries of $\tau^p$ are exactly $\tau$, we have $f(A) \in \cY_{\tau}^{G}$ for all $A \in \cY_{\tau^p}$.
    Conversely, if $A \in \cY_{\tau}^{G}$, then its eigenvalues are given by $\tau$ and the characteristic polynomial is of the form \eqref{eq:char_p}, so $A$ is in the image of $f$.
\end{proof}

\begin{lemma}\label{l:fixed matching}
Under the identification of Lemma~\ref{l:fixed set symp}, the fixed part
$L_{\frm}^{G}\subset\cY_\tau^{G}$ of the matching Lagrangian of an
equivariant matching $\frm$ is Hamiltonian isotopic to the matching
Lagrangian $L_{\frm^p}\subset\cY_{\tau^p}$ of the quotient matching
$\frm^p$, where $\frm^p$ consists of the $p$-th powers of the matching paths in $\frm$.
\end{lemma}

\begin{proof}
The arguments in \cite[Lemma~37]{seidel2010localization} can be applied here.
%The fixed part of the Lagrangian sphere $L$ over the configuration
%$\tau$ is the Lagrangian sphere of the quotient configuration, and the
%rescaled parallel transport $h_\gamma^{resc}$ along an equivariant path
%$\gamma$ restricts on the fixed locus to the rescaled parallel transport along
%the quotient path, which produces $L_{\frm^p}$ up to Hamiltonian isotopy.
\end{proof}

\subsection{Normal trivializations and Pin structures}\label{subsec:application-normal-data}

Fix a $\tau \in (\Sym^{2pm}(\mathbb{C}))^{G}$.
We have a short exact sequence of bundles
\begin{align}\label{eq:bdl_ses}
    0 \to T\cY_{\tau} \to TS_{pm}|_{\cY_{\tau}}=\ul{\mathbb{C}}^{4pm} \to T_{\tau}\Sym^{2pm}(\mathbb{C})=\ul{\mathbb{C}}^{2pm} \to 0
\end{align}
where $\ul{\mathbb{C}}^r$ refers to the trivial rank-$r$ bundle, the second arrow is induced by the inclusion $j:\cY_{\tau} \to S_{pm}$, and the third arrow is $d\chi$.
Therefore, $T\cY_{\tau}$ is stably trivial as a complex vector bundle.
Moreover, since $\chi$ is $\mathbb{Z}_p$-invariant, the exact sequence is equivariant with respect to the $\mathbb{Z}_p$-action.
In particular, it restricts to an exact sequence over $\cY_{\tau}^{G}$:
\[
0 \to T\cY_{\tau}|_{\cY_{\tau}^{G}} \to TS_{pm}|_{\cY_{\tau}^{G}}=\ul{\mathbb{C}}^{4pm} \to T_{\tau}\Sym^{2pm}(\mathbb{C})=\ul{\mathbb{C}}^{2pm} \to 0.
\]
For each irreducible real representation $\lambda$ of $\mathbb{Z}/p\mathbb{Z}$, we can take its isotypic component and obtain the following induced short exact sequence:
\begin{align}\label{eq:bdl_ses_perp}
    0 \to (T\cY_{\tau}|_{\cY_{\tau}^{G}})^{\lambda} \to (TS_{pm}^{\perp}|_{\cY_{\tau}^{G}})^{\lambda} \to (T_{\tau}\Sym^{2pm}(\mathbb{C}))^{\lambda}  \to 0.
\end{align}
The middle and right-hand bundles in \eqref{eq:bdl_ses_perp} are trivial:
they are the isotypic subbundles of the constant representations defining
the ambient tangent bundle and the tangent space of the base.

\begin{corollary}
    The isotypic components of the normal bundle of $\cY^{G}_{\tau}$ in $\cY_{\tau}$ are stably trivial as complex vector bundles.
\end{corollary}

\begin{proof}
Along $\cY_\tau^G$, there is a $G$-equivariant splitting
$$ T\cY_\tau|_{\cY_\tau^G} \cong T\cY_\tau^G\oplus \nu, $$
where $\nu=\nu_{\cY_\tau^G/\cY_\tau}$ is the normal bundle. Hence, for every nontrivial irreducible representation $\lambda$,
$$ \nu_\lambda=(T\cY_\tau|_{\cY_\tau^G})^\lambda. $$
We can choose an equivariant Hermitian metric to split the short exact sequence \eqref{eq:bdl_ses_perp} so we have 
\[\nu_\lambda \oplus (T_{\tau}\Sym^{2pm}(\mathbb{C}))^{\lambda}=(TS_{pm}^{\perp}|_{\cY_{\tau}^{G}})^{\lambda}\] and hence $\nu_\lambda$ is stably trivial.

%    The normal bundle of $\cY_{\tau}^{G}$ in $\cY_{\tau}$ is isomorphic to $(T\cY_{\tau}^{G})^{\perp}$, which by \eqref{eq:bdl_ses_perp} is stably trivial.
\end{proof}

We want to express $(T_{\tau}\Sym^{2pm}(\mathbb{C}))^{\lambda}$ and $(TS_{pm}^{\perp}|_{\cY_{\tau}^{G}})^{\lambda}$ as cotangent bundles of real $G$-representations.
To do this, we consider the irreducible representation $\rho_k \simeq \mathbb{R}^2$ such that $\sigma$ acts by
\[
 \left(
\begin{array}{cccccc}
   \cos(\frac{2\pi k}{p})  &  -\sin(\frac{2\pi k}{p}) \\
   \sin(\frac{2\pi k}{p})  &  \cos(\frac{2\pi k}{p}) 
\end{array}
\right).
\]
The induced action on $T^*\rho_k$ is given by $\sigma \cdot (q,p)=(\sigma q,\sigma^{-T} p)=(\sigma q,\sigma p)$.
Denote the canonical symplectic form by $\omega_{can}=dq_1 \wedge dp_1+dq_2 \wedge dp_2$.
On the other hand, let $\mathbb{C}_z \oplus \mathbb{C}_w$, with $z=x+iy$ and $w=u+iv$, be equipped with the standard symplectic form
$\omega_{std}:=dx \wedge dy+du \wedge dv$ and the $G$-action $\sigma \cdot (z,w)=(e^{\frac{2\pi i k}{p}}z, e^{\frac{-2\pi i k}{p}}w)$. One can check that the map $\mathbb{C}_z \oplus \mathbb{C}_w \to T^*\rho_k $ given by
\[
(z,w) \mapsto (q_1,q_2,p_1,p_2)=\left(\operatorname{Re}\left(\frac{z+\bar{w}}{\sqrt{2}}\right),\operatorname{Im}\left(\frac{z+\bar{w}}{\sqrt{2}}\right), \operatorname{Re}\left(-\frac{i(z-\bar{w})}{\sqrt{2}}\right), \operatorname{Im}\left(-\frac{i(z-\bar{w})}{\sqrt{2}}\right)\right)
\]
is a $G$-equivariant symplectomorphism.

By taking the direct sum of these isomorphisms, we obtain the following:

\begin{lemma}\label{l:CtoR}
    Let $W$ be a complex $G$-representation without invariant vectors such that for all $k$, the multiplicities of the eigenvalues $e^{\frac{2\pi i k}{p}}$ and $e^{\frac{-2\pi i k}{p}}$ of the $\sigma$-action are the same. Then the standard product symplectic form on $W$ with respect to the product decomposition into irreducible representations is $G$-equivariantly symplectomorphic to the cotangent bundle of a real $G$-representation.   
\end{lemma}

\begin{corollary}\label{c:symp_trivial}
    For each irreducible real representation $\lambda$ of $G$, the corresponding isotypic components 
$(T_{\tau}\Sym^{2pm}(\mathbb{C}))^{\lambda}$ and $(TS_{pm}^{\perp}|_{\cY_{\tau}^{G}})^{\lambda}$
    are $G$-equivariantly symplectomorphic to a cotangent bundle of a $G$-representation. Therefore, we have a $G$-equivariant isomorphism of symplectic vector bundles
    \[
    \nu_\lambda \oplus T^*\mathbb{R}^{d_1}=T^*\mathbb{R}^{d_2}
    \]
    for some $G$-representations $\mathbb{R}^{d_1}$ and $\mathbb{R}^{d_2}$.
\end{corollary}

\begin{proof}
    The action of $G$ on $S_{pm}$ is given by \eqref{eq:S_action}. The multiplicity of $\xi^k$ is $4m$ for all $k=0,\dots,p-1$.
    The isotypic component $(TS_{pm}^{\perp}|_{\cY_{\tau}^{G}})^{\lambda}$ picks up the eigenspaces with eigenvalues $\xi^k$ and $\xi^{p-k}$ for some $k \neq 0$, so Lemma~\ref{l:CtoR} applies.

    The coordinates of $\Sym^{2pm}(\mathbb{C})$ are elementary symmetric polynomials.
    For $d=1,\dots,2pm$, there is exactly one degree-$d$ elementary symmetric polynomial in $2pm$ variables and $\sigma$ acts on it by $\xi^d$.
    Therefore, the multiplicity of $\xi^k$ is $2m$ for all $k=0,\dots,p-1$ and we can apply Lemma~\ref{l:CtoR} to $(T_{\tau}\Sym^{2pm}(\mathbb{C}))^{\lambda}$.
\end{proof}

Now we need to discuss the Lagrangian submanifolds. Recall $D_{\epsilon}$ from Equation~\eqref{eq:disc}.
There is a contractible Lagrangian submanifold $\Delta \subset \chi^{-1}(D_\epsilon)$ such that $\Delta \subset \chi^{-1}(D_{\epsilon} \cap \Sym^{2m}(\mathbb{R}))$,
$L= \Delta \cap \chi^{-1}(\tau_\delta)$, and the following diagram commutes \cite[Lemma~31]{seidel2010localization}:
\begin{equation}\label{eq: rel stable trivialization}
	\begin{tikzcd}[row sep=large, column sep=large]
    0  \arrow[r] & TL \otimes_\mathbb{R}\mathbb{C} \arrow [d, "\simeq", sloped] \arrow[r, "dj"] & T\Delta|_{L} \otimes_\mathbb{R}\mathbb{C} \arrow [d, "\simeq", sloped] \arrow[r, "d\chi"] &  \ul{\mathbb{R}}^{m} \otimes_{\mathbb{R}} \mathbb{C}  \arrow [d, "\simeq", sloped]\arrow[r] & 0 \\
     0 \arrow[r] & T\cY_{\tau_\delta}|_{L}  \arrow[r, "dj"] & T\chi^{-1}(D_{\epsilon})|_{L}  \arrow[r, " d\chi"] &  \ul{\mathbb{C}}^{m}  \arrow[r] & 0
	\end{tikzcd}
\end{equation}
Up to a Moser argument, $\Delta$ is simply
\[
\Delta=\{a_i^2+b_i^2+c_i^2 \le \delta, a_i,b_i,c_i \in \mathbb{R}\} \subset (\mathbb{C}^3)^{m}
\]
under the local isomorphism \eqref{eq:local_iso}.

Similarly, we consider the following polydisk:
\[
D^{(p)}_{\epsilon}=\left\{ \{\xi^r (\mu_i \pm \sqrt{\epsilon_{i,r}}): r=1,\dots,p, \text{ and }i=1,\dots,m\} \in \Sym^{2pm}(\mathbb{C}) : |\epsilon_{i,r}|^2  <\epsilon \right\}.
\]
When $\epsilon_{i,r}=0$, it is understood that $\xi^r \mu_i$ appears twice.
The group $\mathbb{Z}_p$ acts on $D^{(p)}_{\epsilon}$ by complex multiplication.
Let $\tau_{\delta}^{(p)}$ be the point where all $\epsilon_{i,r}=\delta>0$.
Let $D^{(p)}_{\epsilon,\mathbb{R}}:=D^{(p)}_{\epsilon} \cap \{\epsilon_{i,r} \in \mathbb{R} \text{ for all }i,r\}$.
There is a contractible Lagrangian submanifold $\Delta^{(p)} \subset \chi^{-1}(D^{(p)}_\epsilon)$ such that $\Delta^{(p)} \subset \chi^{-1}(D^{(p)}_{\epsilon,\mathbb{R}})$,
$L^{(p)}:= \Delta^{(p)} \cap \chi^{-1}(\tau_\delta^{(p)})$, and the following diagram commutes \cite[Lemma~31]{seidel2010localization}:
\begin{equation}\label{eq: rel stable trivialization_zp}
	\begin{tikzcd}[row sep=large, column sep=large]
    0  \arrow[r] & TL^{(p)} \otimes_\mathbb{R}\mathbb{C} \arrow [d, "\simeq", sloped] \arrow[r, "dj"] & T\Delta^{(p)}|_{L^{(p)}} \otimes_\mathbb{R}\mathbb{C} \arrow [d, "\simeq", sloped] \arrow[r, "d\chi"] &TD^{(p)}_{\epsilon,\mathbb{R}} \otimes_{\mathbb{R}} \mathbb{C}=  L^{(p)} \times \mathbb{C}^{pm}  \arrow [d, "\simeq", sloped]\arrow[r] & 0 \\
     0 \arrow[r] & T\cY_{\tau_\delta^{(p)}}|_{L^{(p)}}  \arrow[r, "dj"] & T\chi^{-1}(D^{(p)}_{\epsilon})|_{L^{(p)}}  \arrow[r, " d\chi"] &  TD^{(p)}_{\epsilon}=L^{(p)} \times \mathbb{C}^{pm}   \arrow[r] & 0
	\end{tikzcd}
\end{equation}
where the maps are $\mathbb{Z}_p$-equivariant.
Up to an equivariant Moser argument, $\Delta^{(p)}$ is simply
\[
\Delta^{(p)}=\{(a_{i,r},b_{i,r},c_{i,r}):a_{i,r}^2+b_{i,r}^2+c_{i,r}^2 \le \delta, a_{i,r},b_{i,r},c_{i,r} \in \mathbb{R} \text{ for all }i,r\} \subset (\mathbb{C}^3)^{mp},
\]
and the $\mathbb{Z}_p$-action cyclically permutes the $(\mathbb{C}^3)^m$-factors.

\begin{corollary}[cf.\ {\cite[Lemma~31]{seidel2010localization}}]
\label{cor:application-normal-triviality}
For any $\Z_p$-equivariant crossingless matchings $\frm_0,\frm_1$, the
triple $(\cY_{\tau_\delta^{(p)}},L_{\frm_0},L_{\frm_1})$ satisfies
Assumption~\ref{ass:floer-normal-triviality}.
\end{corollary}

\begin{proof}
%We verify the two clauses of the assumption.
Denote $\tau_\delta^{(p)}$ by $\tau$ in this Corollary.

By Corollary~\ref{c:symp_trivial}, after adding $(T_{\tau}\Sym^{2pm}(\mathbb{C}))^{\lambda} \simeq T^*\mathbb{R}^{d_1}$, $\nu_{\lambda}$ 
can be trivialized to be $(TS_{pm}^{\perp}|_{\cY_{\tau}^{G}})^{\lambda}\simeq T^*\mathbb{R}^{d_2}$.

On the other hand, we can split the exact sequence \eqref{eq: rel stable trivialization_zp} and obtain
\begin{align}\label{eq:TLag}
TL^{(p)} \oplus TD^{(p)}_{\epsilon,\mathbb{R}} \simeq T\Delta^{(p)}|_{L^{(p)}}.
\end{align}
Note that both $T\Delta^{(p)}$ and $TD^{(p)}_{\epsilon,\mathbb{R}}$ are trivial bundles (for $T\Delta^{(p)}$, it is because $\Delta^{(p)}$ is contractible).
By restricting it to $(L^{(p)})^G$ and taking the non-trivial isotypic summands as before, it gives a trivialization of the isotypic components of the normal bundle of $(L^{(p)})^G$  that is compatible with 
\[
\nu_{\lambda}|_{(L^{(p)})^G} \oplus (TD^{(p)}_{\epsilon})^{\lambda} \simeq   (T\chi^{-1}(D^{(p)}_{\epsilon})|_{(L^{(p)})^G})^\lambda.
\]
Let $N$ be the normal bundle of $D^{(p)}_{\epsilon}$ in $\Sym^{2pm}(\mathbb{C})$.
Since $(TD^{(p)}_{\epsilon})^{\lambda}=(TD^{(p)}_{\epsilon,\mathbb{R}})^{\lambda} \otimes_\mathbb{R}\mathbb{C}$ and $(T\Sym^{2pm}(\mathbb{C}))^\lambda$ are of cotangent type,  so is $N \simeq T^*\R^{d_N}$.
Moreover, we have
\begin{align}
 (TD^{(p)}_{\epsilon})^{\lambda} \oplus N^{\lambda} &\simeq (T_{\tau}\Sym^{2pm}(\mathbb{C}))^{\lambda} \nonumber \\
\nu_{\lambda}|_{(L^{(p)})^G} \oplus (TD^{(p)}_{\epsilon})^{\lambda} \oplus N^{\lambda} &\simeq   (T\chi^{-1}(D^{(p)}_{\epsilon})|_{(L^{(p)})^G})^\lambda \oplus N^{\lambda} \simeq (TS_{pm}^{\perp}|_{(L^{(p)})^G})^{\lambda}. \label{eq:TStriv} 
\end{align}
Therefore, the stable trivialization \eqref{eq:TLag} further stabilized by $\R^{d_N}$ is compatible with the stable trivialization \eqref{eq:TStriv}.

%we have $T\Delta^{(p)}|_{L^{(p)}} \otimes_\mathbb{R}\mathbb{C} \simeq T\chi^{-1}(D^{(p)}_{\epsilon})|_{L^{(p)}}=TS_{pm}^{\perp}|_{L^{(p)}}$ and
%$ TD^{(p)}_{\epsilon,\mathbb{R}} \otimes_{\mathbb{R}} \mathbb{C} \simeq TD^{(p)}_{\epsilon} \simeq T_{\tau}\Sym^{2pm}(\mathbb{C})$.
%Note that both $T\Delta^{(p)}$ and $TD^{(p)}_{\epsilon,\mathbb{R}}$ are trivial bundle (for $T\Delta^{(p)}$, it is because $\Delta^{(p)}$ is contractible).
%Therefore, over $L^{(p)}$ and at the point $\tau$ respectively, we have
%\begin{align}\label{eq:anotherReal}
%T^*(T\Delta^{(p)}|_{L^{(p)}}) \simeq TS_{pm}^{\perp}|_{L^{(p)}} \quad  T^*(TD^{(p)}_{\epsilon,\mathbb{R}}) \simeq T_{\tau}\Sym^{2pm}%(\mathbb{C})
%\end{align}
%giving another way to write $TS_{pm}^{\perp}|_{L^{(p)}}$ and $T_{\tau}\Sym^{2pm}(\mathbb{C})$ as cotangent bundle of real $G$ representations.
%Moreover, that is compatible with \eqref{eq:anotherReal}.

A priori, it is possible that the isomorphism \eqref{eq:TLag} stabilized by $\R^{d_N}$ is not compatible with the global isomorphism $(TS_{pm}^{\perp}|_{\cY_{\tau}^{G}})^{\lambda}\simeq T^*\mathbb{R}^{d_2}$ in the sense that the isomorphism
\begin{align}\label{eq:Lagiso}
(T\Delta^{(p)} \otimes_{\R} \mathbb{C}|_{(L^{(p)})^G})^\lambda \oplus N^{\lambda} \simeq (TS_{pm}^{\perp}|_{(L^{(p)})^G})^\lambda \simeq T^*\mathbb{R}^{d_2}
\end{align}
sends $T\Delta^{(p)} \oplus \R^{d_N}$ to a non-trivial $(L^{(p)})^G$-family of $G$-invariant Lagrangian subspaces in $T^*\mathbb{R}^{d_2}$, and similarly for $(TD^{(p)}_{\epsilon,\mathbb{R}})^\lambda$.
However, $(L^{(p)})^G$ is contained in $(\Delta^{(p)})^G$ which is contractible.
Moreover, the isomorphism \eqref{eq:Lagiso} can be extended over $(\Delta^{(p)})^G$.
%\[
%(T^*(T\Delta^{(p)}|_{(\Delta^{(p)})^G}))^\lambda \simeq (TS_{pm}^{\perp}|_{(\Delta^{(p)})^G})^\lambda \simeq T^*\mathbb{R}^{d_2}
%\]
It guarantees a homotopy from the $(L^{(p)})^G$-family of $G$-invariant Lagrangian subspaces to a constant $(L^{(p)})^G$-family and hence to the constant family $\mathbb{R}^{d_2}$ or $\sqrt{-1}\mathbb{R}^{d_2}$ (see Lemma~\ref{l:eqLag} for the connectivity of the space of $G$-invariant Lagrangian subspaces).
It proves the existence of $\Lambda_i$ in Assumption~\ref{ass:floer-normal-triviality}.

%We want to use a homotopic argument to show that any $(L^{(p)})^G$-family of $G$-invariant Lagrangians in $T^*\mathbb{R}^{d_2}$ is stably trivial.
%By Lemma \ref{l:eqLag}, the space of $G$-invariant Lagrangian subspaces in $T^*\mathbb{R}^{d_2}$ is homeomorphic to $U(d_2,d_2)/GL(d_2,\mathbb{C})$ so it is homotopic equivalent to $U(d_2) \times U(d_2)/U(d_2) \simeq U(d_2)$.
%Therefore, the space of stable $G$-invariant Lagrangian subspaces with respect to $\lambda$ is homotopy equivalent to $\cup_d U(d) =U$.
%Since $(L^{(p)})^G$ is diffeomorphic to $(S^2)^m$ and $[(S^2)^m,U]=0$, it means that after possibly further stabilizing 
%$T^*\mathbb{R}^{d_2}$ the $(L^{(p)})^G$-family of $G$-invariant Lagrangian subspaces in $T^*\mathbb{R}^{d_2}$ coming from 
%$T\Delta^{(p)}$ is homotopic to a constant $(L^{(p)})^G$-family and hence to the constant family $\mathbb{R}^{d_2}$.
%It proves the existence of $\Lambda_i$ in Assumption~\ref{ass:floer-normal-triviality}.

Finally, $L_{\frm_i}$ is obtained from $L^{(p)}$ by equivariant rescaled
parallel transport along a path.
Since the stable trivialization is obtained using the fibration $\chi$, the parallel transport map respects the stable trivialization up to canonical homotopy.
Therefore, it verifies the assumption for $(\cY_{\tau_\delta^{(p)}},L_{\frm_0},L_{\frm_1})$.
\end{proof}

\paragraph{Pin structures.}

The remaining hypothesis of Theorem~\ref{thm:floer-smith} is
Assumption~\ref{ass:floer-pin}.  The Lagrangians $L_{\frm}$ are diffeomorphic
to $(S^2)^{pm}$, hence spin, and a spin structure is a Pin structure with
$E=0$; this is all the assumption asks.  Here $\Kh(K)=HF(L_{\frm_0},L_{\frm_1})$ and
$\rank_{\F_p}$ denotes the total rank over $\F_p$.

\paragraph{The orientation system of the fixed pair.}
Theorem~\ref{thm:floer-smith} gives the Floer homology of the fixed pair
$(L_{\frm_0}^G,L_{\frm_1}^G)$ with the orientation system of
Lemma~\ref{lem:floer-fixed-strips}: that of the relative Pin structures
$Q_i$ induced from the spin structures of $L_{\frm_i}$, with respect to the
background bundle $E'=\cY^G_\tau\times\R^m$
of \eqref{eq:floer-fixed-background}; whereas $\Kh(\bar K)$ is defined with the standard system,
background bundle zero and the unique Pin structures of the simply
connected fixed matching Lagrangians.  The two homologies have the same
rank, as follows.  The background bundle $E'$ is trivial; let $\mathfrak t$ be the
spin structure induced by its trivialization.  Then
$Q_i=Q^0_i\oplus\mathfrak t|_{L^G_{\frm_i}}$ for a unique Pin structure
$Q^0_i$ on $TL^G_{\frm_i}$, which is the unique Pin structure of
$L^G_{\frm_i}\cong(S^2)^m$; this conclusion does not depend on
$\mathfrak t$, since another spin structure on $E'$ differs from it by a
class in $H^1(\cY^G_\tau;\Z_2)$ whose restriction to $L^G_{\frm_i}$
vanishes.  Finally the orientation systems of $(Q_0,Q_1)$ with background
bundle $E'$ and of $(Q^0_0,Q^0_1)$ with background bundle zero coincide:
the boundary of a strip $u$ of the fixed pair bounds the disc
$u(D)\subset\cY^G_\tau$, over which $E'$ is trivialized, and the
restriction of $\mathfrak t$ to that boundary is the Pin structure of the
trivialization, so the Pin structure induced on the boundary loop by
$Q_0,Q_1$ after trivializing $E'$ over the disc is the one induced by
$Q^0_0,Q^0_1$, and the gluing isomorphisms are built from these loop
structures (Appendix~\ref{app:floer-orientations},
\cite{wehrheim2015orientations}).  Hence the rank of the fixed-pair Floer
homology over $\F_p$ is that of $\Kh(\bar K)$.  This matters because
ranks over $\F_p$ can depend on the orientation system.

\subsection{The rank inequality}\label{subsec:application-inequality}

The following result generalizes \cite[Corollary~3]{seidel2010localization}.

\begin{theorem}\label{thm:application-khovanov}
    For any $p$-periodic link $K$ with quotient $\bar K$, there is a rank inequality
    \[
    \rank_{\mathbb{F}_p}(\Kh(K)) \ge \rank_{\mathbb{F}_p}(\Kh(\bar K)).
    \]
\end{theorem}

\begin{proof}
Choose an equivariant presentation of $K$ by crossingless matchings
$\frm_0,\frm_1$; one exists: present the quotient link $\bar K$ as a plat
closure in the complement of the quotient axis, taken to be the vertical
axis through the origin, and lift through the $p$-fold cyclic branched
cover of $\R^3$ along the axis, under which cups, caps and the braid lift
to cups, caps and a rotation-invariant braid on $2pm$ strands.  By Lemmas~\ref{l:fixed set symp}
and~\ref{l:fixed matching}, the fixed pair of
$(\cY_\tau,L_{\frm_0},L_{\frm_1})$ is $(\cY_{\tau^p},L_{\frm_0^p},L_{\frm_1^p})$
up to Hamiltonian isotopy, which computes $\Kh(\bar K)$.  The hypotheses
of Theorem~\ref{thm:floer-smith} hold by Corollary~\ref{cor:application-normal-triviality} and by the
discussion of the Pin structures. Therefore, the inequality follows from the theorem.
\end{proof}

\begin{remark}
Throughout, $\Kh$ denotes symplectic Khovanov homology.  Its agreement with
combinatorial Khovanov homology is known over fields of characteristic zero
\cite{abouzaid2019khovanov}; whether it holds over $\F_p$ is not known to
us, and the inequality above is a statement about the symplectic invariant.
\end{remark}

\begin{remark}
    We do not know how to generalize the double branched cover spectral sequence \cite[Corollary~34]{seidel2010localization} to $p$-fold covers, because there is no obvious $\mathbb{Z}_p$-action on $\cY$ with the corresponding fixed locus.
\end{remark}

%{\color{red}
%\input{application_916}
%}
\appendix
\renewcommand{\thesection}{\Alph{section}}
\numberwithin{theorem}{section}

\section{Orientations of the fixed data}\label{app:floer-orientations}
This appendix proves Lemma~\ref{lem:floer-fixed-strips}.  The boundary signs of
Lemmas~\ref{lem:floer-boundary-formulas} and~\ref{lem:floer-sphere-boundary}
are computed in Subsection~\ref{subsec:IIF-orientations}, after the collars
have been constructed.

\subsection{The normal orientation line}\label{subsec:app-normal-line}

Recall the group complex structure $J_G$ of
Definition~\ref{def:group-complex-structure} and the complex orientation
it induces.

\begin{lemma}[Equivariant operators are complex linear]
\label{lem:app-equivariant-operators}
Let $D$ be a $G$-equivariant Fredholm operator between spaces of sections
of $G$-equivariant vector bundles whose fibers have no invariant vectors,
over a base on which $G$ acts trivially, and let $J_G$ act on sections
fiberwise.  Then $D$ is $J_G$-complex linear, $\ker D$ and
$\operatorname{coker}D$ are complex vector spaces for $J_G$, and $\det(D)$
carries the resulting \emph{complex orientation}.  This orientation is
continuous along families of such operators and is preserved by
$G$-equivariant isomorphisms between them.
\end{lemma}
\begin{proof}
The spaces of sections are $G$-representations without invariant vectors,
since an invariant section takes values in the invariant part of the
fibers, so Lemma~\ref{lem:group-complex-structure} applies to $D$ and to
equivariant isomorphisms.  The kernel is an invariant subspace, hence
$J_G$-complex; the cokernel is an invariant quotient, and it has no
invariant vectors, since a vector that is invariant modulo the image
differs from its average, which is invariant and hence zero, by an
element of the image; so the cokernel is $J_G$-complex as well.  The
determinant line of a complex linear Fredholm operator is canonically
oriented, continuously in families \cite[Appendix~A.2]{mcduff2012j}, and
an equivariant isomorphism preserves the complex orientation by
Lemma~\ref{lem:group-complex-structure}.
\end{proof}

\emph{The normal cap operator.}  At every invariant intersection point
$\cP{p}$ the normal factor of the adapted data is the same: the space $W$
with $J_W$, the boundary pair $\R^m,\sqrt{-1}\,\R^m$, and the normal part
$\lambda^\perp(t)=e^{J_W\pi t/2}\,\R^m$ of the capping path of
Definition~\ref{conv:floer-orientation-data}\ref{it:F-paths}.  Let
$D^\perp$ be the Cauchy--Riemann operator on the disc with one boundary
puncture, with values in $W$ and boundary condition $\lambda^\perp$.  The
cap operator at $\cP{p}$ splits as $D_{\cP{p}}=D^G_{\cP{p}}\oplus D^\perp$,
where $D^G_{\cP{p}}$ is the cap operator of the fixed data for the
tangential path $\lambda^G_{\cP{p}}(t)=e^{J_0\pi t/2}\,T_{\cP{p}}\widehat L^G_0$,
and accordingly
\begin{equation}\label{eq:floer-orientation-line-splitting}
    o_{\cP{p}}=o^G_{\cP{p}}\otimes o^\perp,
    \qquad o^G_{\cP{p}}:=\det(D^G_{\cP{p}}),
    \quad o^\perp:=\det(D^\perp).
\end{equation}
The \emph{normal orientation line} $o^\perp$ is one line, the same at
every invariant point.  The operator $D^\perp$ is $G$-equivariant, since
$J_W$ and $\lambda^\perp$ are $G$-invariant, so by
Lemma~\ref{lem:app-equivariant-operators} its kernel and cokernel are
complex, hence even-dimensional, and $o^\perp$ carries a complex
orientation.  Consequently no sign enters in
\eqref{eq:floer-orientation-line-splitting} or in any reordering of normal
and tangential factors below.  We orient $o^\perp$ by its complex
orientation; the orientations \ref{it:F-fixed} of $o_{\cP{p}}$ then become
orientations of the tangential lines $o^G_{\cP{p}}$, which are the
orientation lines of the fixed data.  Any other orientation of $o^\perp$
would reverse all $o^G_{\cP{p}}$ at once and change nothing below.

\emph{The induced capping data.}  Recall from
Subsection~\ref{subsec:floer-orientations} that $Q_i$ is the transport of
$P_i|_{L^G_i}$ through $1\oplus\rho_i\oplus1$, where $\rho_0=\mathrm{id}$
and $\rho_1=-J_W$ identify $\nu_{\widehat L_i}=L^G_i\times\Lambda_i(1)$
with $L^G_i\times\R^m$.  Likewise
$\rho_t:=e^{-J_W\pi t/2}\colon\lambda^\perp(t)\to\R^m$, which is $\rho_0$
at $t=0$ and $\rho_1$ at $t=1$, transports the Pin extension of
\ref{it:F-paths} along $\lambda_{\cP{p}}=\lambda^G_{\cP{p}}\oplus\lambda^\perp$
to a Pin extension along $\lambda^G_{\cP{p}}$ for the fixed data with
background bundle $E'$, agreeing with $Q_0$ and $Q_1$ at the ends.  These
are the capping data of the fixed data.

\subsection{Proof of Lemma~\ref{lem:floer-fixed-strips}}
\label{subsec:app-fixed-strips}

\emph{The gluing isomorphism in the formalism.}  We use the following
description of \eqref{eq:floer-gluing-isomorphism} and the following
properties of the orientations entering it
\cite{wehrheim2015orientations,seidel2008fukaya,fooo2009lagrangian}.  Let
$u$ be a regular strip from $\cP{x}$ to $\cP{y}$ and trivialize
$u^*T\widehat M$ over the disc.  Gluing the cap at $\cP{y}$ to $u$ gives
an operator $D_u\sharp D_{\cP{y}}$ on the disc with one puncture, at
$\cP{x}$; its boundary condition is a path from $T_{\cP{x}}\widehat L_0$ to
$T_{\cP{x}}\widehat L_1$ which, followed by the capping path at $\cP{x}$
reversed, is a loop $F_u$ of Lagrangian subspaces, and the glued operator
is deformed, through operators on the disc with one puncture, to the
operator $D_{\cP{x}}\sharp D_{F_u}$ obtained by attaching to the cap at
$\cP{x}$ a closed disc with boundary condition $F_u$.  The isomorphism
\eqref{eq:floer-gluing-isomorphism} is the composite
\begin{equation}\label{eq:app-gluing-composite}
 \det(D_u)\otimes o_{\cP{y}}\cong\det(D_u\sharp D_{\cP{y}})
 \cong\det(D_{\cP{x}}\sharp D_{F_u})
 \cong o_{\cP{x}}\otimes\det(D_{F_u})\cong o_{\cP{x}}
\end{equation}
of two linear gluing isomorphisms, the identification along the
deformation, and the trivialization of $\det(D_{F_u})$ by the orientation
$o_{\mathfrak s}$ determined by the Pin structure $\mathfrak s$ on
$F_u\oplus E|_{F_u}$ assembled from $P_0$ along $u(\cdot,0)$, the
extension at $\cP{y}$, $P_1$ along $u(\cdot,1)$ and the extension at
$\cP{x}$.  The isomorphism \eqref{eq:floer-gluing-isomorphism}
and the orientation $o_{\mathfrak s}$ of the determinant line of the disc
operator with boundary loop $F$ and Pin structure $\mathfrak s$ have the
following properties.
\begin{enumerate}[label=\textup{(P\arabic*)},start=0]
    \item\label{it:P-constant} For the constant strip at an
    intersection point $\cP{x}$, whose linearized operator is invertible so
    that its determinant line is canonically $\R$, the isomorphism
    \eqref{eq:floer-gluing-isomorphism} is the identity of $o_{\cP{x}}$;
    this is forced by compatibility with breaking, since gluing the
    constant strip to any strip returns that strip.
    \item\label{it:P-homotopy} $o_{\mathfrak s}$ is natural under deformations of
    $(F,\mathfrak s)$ and under isomorphisms of Pin structures covering
    the identity of $F$.
    \item\label{it:P-sum} It is compatible with direct sums: if
    $F=F'\oplus F''$ and $\mathfrak s=\mathfrak s'\oplus\mathfrak s''$, then
    $o_{\mathfrak s}=o_{\mathfrak s'}\otimes o_{\mathfrak s''}$ under
    $\det(D_F)=\det(D_{F'})\otimes\det(D_{F''})$, up to the Koszul sign
    $(-1)^{\ind D_{F'}\,\ind D_{F''}}$.
    \item\label{it:P-stab} A relative Pin structure with respect to a
    background bundle enters through the Pin structure it induces on $F$
    after trivializing the background bundle over the disc, for the
    summand $M^G\times\R^m$ by its constant trivialization, and this
    destabilization is compatible with direct sums.
\end{enumerate}

\begin{proof}[Proof of Lemma~\ref{lem:floer-fixed-strips}]
\emph{The data split.}  Let $u$ be a fixed strip from $\cP{p}$
to $\cP{q}$.  Everything in \eqref{eq:app-gluing-composite} splits into a
tangential and a normal part.  The trivialization of $u^*T\widehat M$ may
be taken to be the direct sum of a trivialization of $u^*TM^G$ and the
constant $W$; the operators split, $D_u=D^G_u\oplus D^\perp_u$
(Proposition~\ref{thm:floer-model-regularity}, where $D^G_u$ is written $D^T$) and
$D_{\cP{p}}=D^G_{\cP{p}}\oplus D^\perp$, $D_{\cP{q}}=D^G_{\cP{q}}\oplus D^\perp$;
the loop splits, $F_u=F^G_u\oplus F^\perp$, where $F^\perp$ is the loop in
$W$ consisting of $\R^m$ along $u(\cdot,0)$, $\lambda^\perp$ at $\cP{q}$,
$\sqrt{-1}\,\R^m$ along $u(\cdot,1)$ and $\lambda^\perp$ reversed at
$\cP{p}$; the linear gluings of split operators are split; and the
deformation can be chosen split, with normal part running from
$D^\perp_u\sharp D^\perp$ to $D^\perp\sharp D_{F^\perp}$ through
$G$-equivariant operators, since all normal data are $G$-invariant.  Every
Koszul sign produced by reordering tangential and normal factors is $+1$,
the normal operators having complex, hence even-dimensional, kernels and
cokernels.

\emph{The Pin structure splits.}  The bundle isomorphism
$\rho\colon F^\perp\to\partial D\times\R^m$ given by $\rho_0$, $\rho_t$,
$\rho_1$ and $\rho_t$ on the four pieces of the loop is continuous.  Let
$\mathfrak s_0$ be the Pin structure on $\partial D\times\R^m$ induced by
the constant trivialization of $M^G\times\R^m$ over the disc, as in
\ref{it:P-stab}, and
$\mathfrak s'':=\rho^*\mathfrak s_0$.  Since
the Pin structures on $F^G_u\oplus E|_{F_u}\oplus F^\perp$ and on
$F^G_u\oplus E|_{F_u}$ are torsors over $H^1(\partial D;\Z_2)$ and
$\mathfrak s'\mapsto\mathfrak s'\oplus\mathfrak s''$ is equivariant, there is
a unique Pin structure $\mathfrak s'$ on $F^G_u\oplus E|_{F_u}$ with
$\mathfrak s=\mathfrak s'\oplus\mathfrak s''$.  Transporting through $\rho$,
the Pin structure $(1\oplus\rho)_*\mathfrak s=\mathfrak s'\oplus\mathfrak s_0$
on $F^G_u\oplus E'|_{F_u}$ is the one assembled from $Q_0$, $Q_1$ and the
induced extensions, and by \ref{it:P-stab} the Pin structure it induces on
$F^G_u$ after trivializing $E'=E\oplus(M^G\times\R^m)$ over the disc is the one induced
by $\mathfrak s'$ after trivializing $E$.  Hence the tangential part of
\eqref{eq:app-gluing-composite}, with $\det(D_{F^G_u})$ trivialized by
$o_{\mathfrak s'}$, is exactly the gluing isomorphism of the fixed data
with the relative Pin structures $Q_0,Q_1$.

\emph{The normal part.}  By \ref{it:P-sum},
$o_{\mathfrak s}=o_{\mathfrak s'}\otimes o_{\mathfrak s''}$ on
$\det(D_{F_u})=\det(D_{F^G_u})\otimes\det(D_{F^\perp})$, the Koszul sign
being $+1$ because $\ind D_{F^\perp}$ is even.  The normal part of
\eqref{eq:app-gluing-composite} is therefore the composite of the linear
gluings and of the identification along the normal part of the
deformation, which involve only $G$-equivariant operators and hence
preserve complex orientations
(Lemma~\ref{lem:app-equivariant-operators}), and of the trivialization of
$\det(D_{F^\perp})$ by $o_{\mathfrak s''}$.  Neither the loop $F^\perp$ nor
its Pin structure $\mathfrak s''$ depends on $u$, so $o_{\mathfrak s''}$
differs from the complex orientation of $\det(D_{F^\perp})$ by a sign
$c_\perp\in\{\pm1\}$ independent of $u$.  Since $D^\perp_u$ is invertible,
$\det(D^\perp_u)=\R$ with its canonical orientation, which is also its
complex orientation, so the normal part is an isomorphism
$o^\perp\cong\R\otimes o^\perp$ equal to $c_\perp$ times the identity, for
every fixed strip $u$.

\emph{Conclusion.}  Under \eqref{eq:floer-orientation-line-splitting} at
both ends, \eqref{eq:floer-gluing-isomorphism} for $u$ is therefore
$c_\perp$ times the tensor product of the gluing isomorphism of the fixed
data with the identity of $o^\perp$.  The constant strip at
$\cP{p}$ is a fixed strip to which all of the above applies: for it,
\eqref{eq:floer-gluing-isomorphism} is the identity of $o_{\cP{p}}$ by
\ref{it:P-constant}, and its tangential part is the identity of
$o^G_{\cP{p}}$ by \ref{it:P-constant} for the fixed data, so its normal
part is the identity and $c_\perp=+1$.  Hence the orientation of
$\det(D_u)=\det(D^G_u)$ defined by
Convention~\ref{conv:floer-strip-orientation} from the orientations
\ref{it:F-fixed} is the one defined by the same convention for the fixed
data from the induced orientations of the $o^G_{\cP{p}}$; the translation
vector $\partial_su$ is tangential, so the orientations of
$\M_{M^G}(\cP{p},\cP{q})$ agree.  The remaining
assertions follow: the two complexes have the same generators and, by
Proposition~\ref{thm:floer-model-regularity}, the same rigid strips, now
with the same signs.
\end{proof}

\section{Floer compactification}\label{app:floer-compactification}
We prove Proposition~\ref{thm:floer-compactification-cited} and classify the
boundary of the zero- and one-dimensional moduli spaces.  The gluing
collars and their orientations are constructed in
Appendix~\ref{app:floer-gluing}.

\subsection{Broken convergence}
A sequence of incidence configurations $(\vartheta_n,u_n)$ converges in the
\emph{broken topology} when the strip components converge, after separate
translations of their levels, to a broken Floer strip and the sphere
half-trajectories converge to a broken half-trajectory, with compatible
endpoint evaluations.  The number of levels is finite.

For fixed intersection endpoints the energy of every strip is the action
difference \eqref{eq:floer-action}.  Exactness excludes sphere and disc
bubbles; the confinement of Subsection~\ref{subsec:lagrangian-preparation} keeps the images in a compact
set.  Gromov--Floer compactness therefore gives convergence on compact
subsets of the translated components, with the neck images approaching
the intermediate intersection points.  No nonconstant zero-energy level
occurs.  Since there are finitely many intersections and each nonconstant
level strictly lowers action, the number of levels is uniformly bounded.
The sphere factors have the ordinary broken-trajectory compactification
for the fixed Morse--Smale pair on the compact sphere
\cite{seidel2008fukaya,wehrheim2012smooth}.

\begin{lemma}[No degeneration to an entirely fixed building]
\label{lem:IIF-no-degeneration}
A sequence of non-fixed strips cannot converge to a broken configuration
all of whose strip components are fixed.
\end{lemma}
\begin{proof}
Suppose otherwise.  All fixed components lie in $Q^G$, and the adapted
tube contains a neighborhood of this compact set.  Convergence on the
components, together with convergence of the necks to the break points,
places the whole strip in the adapted tube for sufficiently large $n$.
Its normal part then solves
\[
 (\partial_s+J_W\partial_t)w_n=0,
 \qquad w_n(s,i)\in\Lambda_i(1),
 \qquad w_n(s,\cdot)\longrightarrow0\quad(s\to\pm\infty).
\]
Each nonzero eigenmode of the invertible asymptotic operator fails to
decay at one end.  Thus $w_n=0$, contrary to the assumption that the strips
are non-fixed.
\end{proof}

\subsection{Persistence of evaluation}
\begin{theorem}\label{thm:IIF-evaluation}
For a broken limit of an incidence sequence, let $u_*$ be the first
non-fixed strip component.  Its negative end is invariant, its selected
coefficient vector is nonzero, and
\[
 \Ev^-_{\cP{p},N}(u_n)\longrightarrow\Ev^-_{\cP{q},N}(u_*),
\]
where $\cP{q}$ is the negative end of $u_*$ and the spheres at $\cP{p}$
and $\cP{q}$ are identified by the constant normal model.  Thus evaluation
extends continuously to the broken compactification.
\end{theorem}
\begin{proof}
The component $u_*$ exists by Lemma~\ref{lem:IIF-no-degeneration}.
Every preceding component is fixed, so its negative end $\cP{q}$ is
invariant.  The cutoff corollary applies to $u_*$, giving $c_N(u_*)\ne0$.

If the first component is non-fixed, use its translation coordinates and
choose a fixed slice $s_0$ in its adapted negative end.  For large $n$
the same slice lies in the adapted end of $u_n$.  The evaluation formula
in Lemma~\ref{lem:floer-coefficient-readoff} gives convergence of each
selected coefficient; weighted normalization is continuous at the nonzero
limit.

If fixed components precede $u_*$, translate to the coordinates of $u_*$
and choose such a slice $s_0$ there.  The entire part of $u_n$ before this
slice lies in the adapted tube for large $n$: it shadows fixed components
in $Q^G$ and necks at invariant intersections.  Its normal equation is
exactly constant-coefficient.  Hence its source evaluation equals
\[
 q_N\bigl(\Pi_{E_N^-}w_{u_n}(s_0,\cdot)\bigr).
\]
 where $\Pi_{E_N^-}$ projects the coefficients of $w_{u_n}$ to $E_N^-$.
This is invariant under the elapsed-time weighted scaling.  Slice
convergence identifies the limit with the same expression for $u_*$,
which is defined because its selected projection is nonzero.  The sphere
half-trajectories have this limiting endpoint and converge or break
accordingly.
\end{proof}

\begin{proof}[Proof of Proposition~\ref{thm:floer-compactification-cited}]
The strip compactness described above and compactness of the sphere
half-trajectories provide the subsequences.  Theorem~\ref{thm:IIF-evaluation}
identifies the limiting incidence.  The non-fixed assertion is
Lemma~\ref{lem:IIF-no-degeneration}; closedness of $M^G$ implies that a
limit of fixed strips has only fixed components.
\end{proof}

\subsection{Low-dimensional boundary strata}
\begin{lemma}[No higher-codimension limits]\label{lem:IIF-codim-two}
A limit of a space of one of the four kinds in
Theorem~\ref{thm:floer-gluing-package}, of dimension $d\le1$, has at most
$d$ breaks in total.  If $d=1$ and a sphere half-trajectory breaks from
$\cP{v}$ to $\cP{w}$, then $j(\cP{w})=j(\cP{v})-1$.
\end{lemma}
\begin{proof}
For a strip building with $k$ nonconstant components, additivity of the
Fredholm index and quotienting each component by its translation reduce
the total dimension by $k-1$.  A sphere break reduces dimension by one as
well.  The incidence codimension is $N-1-j(\cP{v})$ on the first non-fixed
factor, after transporting across fixed factors.  Thus a stratum with $b$
breaks has expected dimension $d-b$.  All its strip and incidence factors
are transverse by the ambient and incidence regularity theorems, and the
sphere pair is Morse--Smale.  Negative-dimensional strata are empty.
If a single sphere break decreases $j$ by at least two, its trajectory
factor has positive dimension and the remaining incidence factor in a
zero-dimensional boundary stratum would have negative dimension.  It is
therefore empty.
\end{proof}

\begin{proposition}[Exhaustion and finiteness]\label{prop:IIF-strata}
Every zero-dimensional space of the four kinds is finite.  The boundary
limits of a one-dimensional space are exactly the configurations listed
in \textup{(S1)}--\textup{(S4)} of
Theorem~\ref{thm:floer-gluing-package}, each a finite product stratum.
\end{proposition}
\begin{proof}
Sequential compactness and Lemma~\ref{lem:IIF-codim-two} show that a
zero-dimensional space has no broken limits, so it is compact and
discrete.  For dimension one there is at most one break.  An outside
intermediate intersection gives \textup{(S1)}, and a sphere break gives
\textup{(S2)}.  Fixed strips remain fixed and give \textup{(S3)}.
At an invariant intermediate intersection in a non-fixed configuration,
there is at least one non-fixed component; the three possibilities are
fixed outgoing, fixed incoming, and both non-fixed, as in
\textup{(S4a)}--\textup{(S4c)}.  Persistence of evaluation places the
constraint on the stated factor.  If the outgoing strip is an ordinary
rigid non-fixed strip, Lemma~\ref{lem:floer-top-cells} partitions it by
top cells without an extra constraint.  The boundary factors are
zero-dimensional and hence finite.  Conversely the gluing constructions
of Appendix~\ref{app:floer-gluing} realize every listed configuration as
a boundary limit.
\end{proof}

\section{Floer gluing}\label{app:floer-gluing}
There are four analytic phenomena to handle: breaking at an outside
intersection point; breaking at an invariant intersection point, where the
evaluation constraint is carried through the break along a fixed component;
breaking of the auxiliary sphere trajectory; and the recovery of every
nearby unbroken strip, which makes the collars a description of the whole
boundary.  Subsection~\ref{subsec:IIF-outside-gluing} treats the outside
break in full, as the analytic model; the remaining subsections adapt the
construction to the other strata, each beginning with what changes, and the
last assembles the collars into the compactification and gluing theorem and
computes the boundary orientations.

Four parameters play distinct roles throughout: $\ell$ is the fixed cutoff
width, $T$ the neck length, $\kappa$ the kernel parameter, and $\xi$ the
correction in the complement of the kernel.  The analytic order is: choose
$\ell$ once and for all; for each $(T,\kappa)$ solve for $\xi(T,\kappa)$ by
a contraction; then solve the incidence equation for $\kappa(T)$.  In
particular $\kappa$ is not an independent parameter of the final collar,
which is parametrized by $T$ alone: it is eliminated by incidence
transversality.

\subsection{Gluing at an outside break}\label{subsec:IIF-outside-gluing}
Let $\cP{p}$ be invariant, $\cP{y}$ outside, $\cP{b}$ an arbitrary
intersection, and $\cP{v}\in\crit(h_N)$.  The broken stratum is
\[
 Z=\mathcal N^0_{\cP{p},N}(\cP{v};\cP{y})
       \times\M^0(\cP{y},\cP{b}).
\]
Its first factor is a rigid \emph{incidence configuration}; the ambient
strip need not be rigid.  Its second factor is an ordinary rigid strip.
Both strips are non-fixed.  The sphere half-trajectory in the first
factor is determined by the evaluation.

\begin{proposition}[Outside-break gluing]\label{thm:IIF-outside-gluing}
There are $T_1>0$ and a map
\[
 \mathcal G:
 \mathcal N^0_{\cP{p},N}(\cP{v};\cP{y})
 \times\M^0(\cP{y},\cP{b})\times[T_1,\infty)
 \longrightarrow\mathcal N^1_{\cP{p},N}(\cP{v};\cP{b})
\]
which is a homeomorphism onto its image.  It extends continuously by
$\mathcal G(u_-,u_+,\infty)=(u_-,u_+)$ to a collar of the whole finite
stratum $Z$.  Every unbroken configuration sufficiently close to $Z$ is
in its image.  The maps are equivariant on the union of translates of
$Z$, and the boundary sign is the product of the signs of the two factors.
\end{proposition}

We first fix $(u_-,u_+)\in Z$ and construct its local collar.  The common
constants and assembly over $Z$ are discussed in
Subsection~\ref{subsec:IIF-proof}.

\subsubsection{Linear theory}\label{subsec:IIF-linear}
Fix a Sobolev exponent $q>2$.  The linearizations of
$\mathcal F(u)=\partial_su+J_t(u)\partial_tu$ at $u_\pm$ act as
\[
 D_\pm:W^{1,q}(u_\pm^*T\widehat M;T\widehat L_0,T\widehat L_1)
         \longrightarrow L^q(u_\pm^*T\widehat M).
\]
Boundary values are in $T\widehat L_i$ along $t=i$, and sections decay at
the ends in the Sobolev sense.  Nondegenerate intersections give Fredholm
operators; the kernel elements are smooth and exponentially decaying.

Because $W^{1,q}$ is not a Hilbert space, specify the complement by
\[
 (\ker D_\pm)^\perp=\{\xi\in W^{1,q}:
       \langle\xi,\kappa\rangle_{L^2}=0\text{ for all }\kappa\in\ker D_\pm\}.
\]
The pairing is bounded since the kernel elements belong to $L^{q'}$.
For an $L^2$-orthonormal basis $(\kappa_a)$ of the kernel,
$\Pi_\pm\xi=\sum_a\langle\xi,\kappa_a\rangle_{L^2}\kappa_a$ is a
bounded projection in $W^{1,q}$.  Thus the complement is closed.  If
$D_\pm$ is surjective, its restriction to this complement has a bounded
inverse $P_\pm$, with
\[
 D_\pm P_\pm=\operatorname{id},\qquad
 P_\pm D_\pm=\operatorname{id}-\Pi_\pm.
\]
The two independent inputs to gluing are:
\begin{enumerate}[label=\textup{(T\arabic*)}]
\item surjectivity of $D_-$ and $D_+$, provided by ambient regularity;
\item transversality of $\Ev^-_{\cP{p},N}$ to
$\descending_{h_N}(\cP{v})$ at $[u_-]$.
\end{enumerate}
The first solves the ambient strip equation; the second solves the
finite-dimensional incidence equation.

\begin{lemma}[Kernel dimensions]\label{lem:IIF-linear-T1}
At the fixed broken pair,
\[
 \dim\ker D_-=N-j(\cP{v}),\qquad\ker D_+=\R\partial_su_+.
\]
\end{lemma}
\begin{proof}
Rigidity of the first incidence configuration and
\eqref{eq:floer-incidence-dimension} give $\ind D_-=N-j(\cP{v})$.
The second strip has index one.  Surjectivity identifies these indices
with kernel dimensions; translation supplies the nonzero kernel vector
of $D_+$.
\end{proof}

Let
\[
 \mathcal K=\ker D_-\cap(\partial_su_-)^\perp.
\]
On the adapted negative end a kernel section has a full expansion in
negative normal modes.  Define $\varsigma$ to retain only the first $r$
coefficients:
\[
 \varsigma:\ker D_-\to E_N^-,\qquad
 \varsigma(\kappa)=\bigl(c_{\lambda_1}(\kappa),\ldots,
                              c_{\lambda_r}(\kappa)\bigr).
\]
This does not discard the other modes from the strip equation.  By
differentiating translation,
\begin{equation}\label{eq:IIF-translation-coefficient}
 \varsigma(\partial_su_-)=
   (-\lambda_1c_{\lambda_1}(u_-),\ldots,-\lambda_rc_{\lambda_r}(u_-)).
\end{equation}
It is the weighted Euler vector and is killed by $dq_N$.

\begin{lemma}[The incidence linearization]\label{lem:IIF-linear-T2}
Set $z=\Ev^-_{\cP{p},N}(u_-)$.  Then
\[
 \mathcal L:\mathcal K\xrightarrow{\varsigma}E_N^-
 \xrightarrow{dq_N|_{c_N(u_-)}}T_zS_N
 \longrightarrow T_zS_N/T_z\descending_{h_N}(\cP{v}),
\]
the last map the quotient by the tangent space of the descending cell,
is an isomorphism.  Both its domain and target have dimension
$N-1-j(\cP{v})$, and $\varsigma$ is injective on $\ker D_-$.
\end{lemma}
\begin{proof}
The kernel slice is the tangent space to the unparametrized ambient
moduli space.  Read-off of the exact negative-end expansion identifies
its evaluation derivative with $dq_N\circ\varsigma$.
A transverse zero-dimensional incidence has an invertible normal
derivative, proving the first claim.  If $\varsigma(\kappa)=0$, its
slice component lies in $\ker\mathcal L=0$.  The remaining translation
component also vanishes, because
\eqref{eq:IIF-translation-coefficient} is nonzero.  No surjectivity of
$\varsigma$ onto the whole $E_N^-$ is asserted or needed.
\end{proof}

\subsubsection{Adapted charts and separated cutoffs}
\label{subsec:IIF-contraction}
Choose representatives so that $u_-(s,\cdot)$ is in the adapted chart
$\Phi_{\cP{y}}$ for $s\ge0$, and $u_+(s,\cdot)$ is in it for $s\le0$.
Set $V_\pm=\Phi_{\cP{y}}^{-1}\circ u_\pm$ on these half-strips.
The two Lagrangians are linear and $J=J_{\cP{y}}$ is constant on the
whole chart, by Convention~\ref{conv:lagrangian-preparation} and the
admissibility conditions.  Thus
\[
 (\partial_s+A_{\cP{y}})V_\pm=0,\qquad
 A_{\cP{y}}=J_{\cP{y}}\partial_t,
\]
with constant transverse Lagrangian boundary conditions.
Choose $2\delta>0$ below the smallest absolute eigenvalue of
$A_{\cP{y}}$.  The mode expansion and elliptic estimates give
\begin{equation}\label{eq:IIF-rates}
 \|V_-(s,\cdot)\|_{C^1}+\|\partial_sV_-(s,\cdot)\|_{C^0}
       \le Ce^{-2\delta s}\quad(s\ge0),
\end{equation}
and the analogous estimate $Ce^{2\delta s}$ for $V_+$ when $s\le0$.
Kernel sections satisfy the corresponding estimates with their norms
as an additional factor.

Let $\upsilon:\R\to[0,1]$ be smooth and nondecreasing, zero on
$(-\infty,-1]$ and one on $[1,\infty)$.  For a cutoff width $\ell\ge1$
and a gluing length $T>4\ell$, set
\begin{equation}\label{eq:IIF-separated-cutoffs}
 \beta_+^{T,\ell}(s)=\upsilon\left(\frac{s+T/2}{\ell}\right),\qquad
 \beta_-^{T,\ell}(s)=1-\upsilon\left(\frac{s-T/2}{\ell}\right).
\end{equation}
Their derivative supports lie in disjoint intervals centered at $-T/2$
and $T/2$, respectively.  In particular,
\begin{equation}\label{eq:IIF-cutoff-supports}
 \beta_-^{T,\ell}=1\text{ on }\operatorname{supp}(\beta_+^{T,\ell})',
 \qquad
 \beta_+^{T,\ell}=1\text{ on }\operatorname{supp}(\beta_-^{T,\ell})',
\end{equation}
\begin{equation}\label{eq:IIF-cutoff-derivative}
 \| (\beta_\pm^{T,\ell})'\|_{L^\infty}\le C/\ell.
\end{equation}
Both cutoffs equal one on $[-T/2+\ell,T/2-\ell]$; they do not form a
partition of unity.  We choose $\ell$ sufficiently large and then hold it
fixed while $T$ increases, writing $\beta_\pm^T$ for short.  This is the
separated-transition arrangement of \cite[Section 5]{hutchings2009gluing};
Figure~\ref{fig:cutoffs} shows it.

\begin{figure}[ht]
\centering
\begin{tikzpicture}[x=0.9cm,y=1.7cm]
  % axis
  \draw[->] (-6.8,0) -- (7.4,0) node[right] {$s$};
  % ticks at -T/2 and T/2
  \foreach \x/\lab in {-4/-\tfrac T2,4/\tfrac T2}
    {\draw (\x,0.06) -- (\x,-0.06) node[below] {$\lab$};}
  % transition intervals
  \foreach \x/\lab in {-5/-\tfrac T2-\ell,-3/-\tfrac T2+\ell,3/\tfrac T2-\ell,5/\tfrac T2+\ell}
    {\draw[dotted] (\x,0) -- (\x,1.25) node[above,font=\scriptsize] {$\lab$};}
  % beta_+ (blue): 0 on the left, rises over [-T/2-l,-T/2+l], 1 on the right
  \draw[thick,blue] (-6.4,0) -- (-5,0)
    .. controls (-4.3,0) and (-3.7,1) .. (-3,1) -- (6.4,1)
    node[above left,font=\small] {$\beta_+^T$};
  % beta_- (red, drawn slightly higher for visibility): 1 on the left, falls over [T/2-l,T/2+l], 0 on the right
  \draw[thick,red] (-6.4,1.08) -- (3,1.08)
    .. controls (3.7,1.08) and (4.3,0.08) .. (5,0.08) -- (6.4,0.08)
    node[below right,font=\small,yshift=-2pt] {$\beta_-^T$};
  \node[red,font=\small] at (-5.9,1.25) {$\beta_-^T$};
  \node[blue,font=\small] at (-5.9,-0.18) {$\beta_+^T$};
  % annotations
  \draw[<->,blue] (-5,0.5) -- (-3,0.5) node[midway,fill=white,inner sep=1pt,font=\scriptsize] {$\operatorname{supp}(\beta_+^T)'$};
  \draw[<->,red] (3,0.5) -- (5,0.5) node[midway,fill=white,inner sep=1pt,font=\scriptsize] {$\operatorname{supp}(\beta_-^T)'$};
  \draw[<->] (-3,-0.42) -- (3,-0.42) node[midway,below,font=\small] {$\beta_-^T=\beta_+^T=1$: both components present};
  \node[font=\small,below] at (-5.9,-0.55) {$a_-(s+T,\cdot)$};
  \node[font=\small,below] at (5.9,-0.55) {$a_+(s-T,\cdot)$};
\end{tikzpicture}
\caption{The separated cutoffs of \eqref{eq:IIF-separated-cutoffs}, $\beta_+^T$
in blue and $\beta_-^T$ in red, the latter drawn slightly above its true
height so that the two graphs are visible where they coincide.  Each
rises or falls over an interval of fixed width $2\ell$ centered at
$\mp T/2$; on the middle segment both equal one and the preglued map
\eqref{eq:IIF-preglue} is the sum of the two translated components in the
chart at $\cP{y}$.  Where one cutoff varies the other equals one, which is
\eqref{eq:IIF-cutoff-supports}.  The width $\ell$ is fixed first; only $T$
grows.}
\label{fig:cutoffs}
\end{figure}
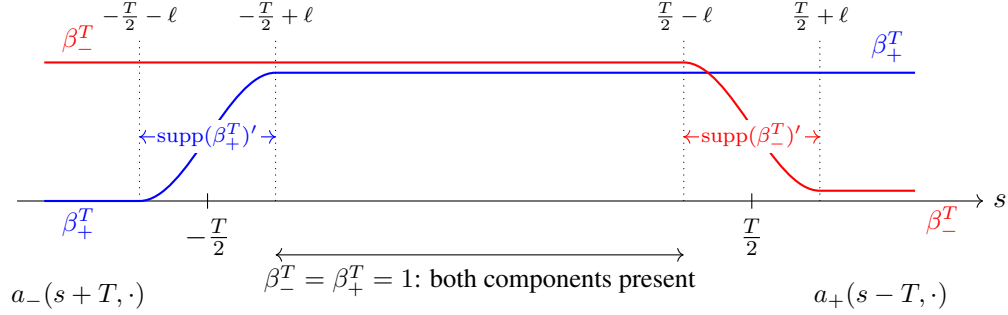

\paragraph{Component maps and pregluing.}
Choose boundary-compatible tubular charts along $u_\pm$, affine in the
adapted intersection charts and in the normal coordinate of the
adapted tube.  They can be obtained from a metric making the Lagrangian
boundary conditions totally geodesic along the boundary, and agreeing
with the adapted affine charts near the ends.  We write these charts as
$\exp_{u_\pm}$ and use their induced bundle identifications for the equation.
For $\kappa_-\in\mathcal K$ and complementary corrections $\xi_\pm$, set
\[
 a_- =\exp_{u_-}(\kappa_-+\xi_-),\qquad
 a_+ =\exp_{u_+}(\xi_+).
\]
For small $W^{1,q}$ norms these maps obey the Lagrangian boundary
conditions.  Define the preglued map in the common chart by
\begin{equation}\label{eq:IIF-preglue}
\begin{split}
 u_T^{\mathrm{par}}(s,t)=\Phi_{\cP{y}}\bigl(&
   \beta_-^T(s)(V_-+\kappa_-+\xi_-)(s+T,t)\\
  &+\beta_+^T(s)(V_++\xi_+)(s-T,t)\bigr).
\end{split}
\end{equation}
It equals $a_-(s+T,t)$ on the left, where $\beta_+^T=0$, and
$a_+(s-T,t)$ on the right, where $\beta_-^T=0$.  The overlap lies in the
adapted ball for large $T$ and small corrections, by the embedding
$W^{1,q}\hookrightarrow C^0$.  The sum in
\eqref{eq:IIF-preglue} is a sum of coordinate vectors, not of manifold
points.  It respects the linear boundary conditions.

In the component charts write the nonlinear equation as
\[
 \mathcal F_\pm(\eta)=D_\pm\eta+\mathcal Q_\pm(\eta).
\]
Following \cite[Section~5.4]{hutchings2009gluing}, call a map
$\mathcal Q$ from sections of $u^*T\widehat M$ to sections of
$u^*T\widehat M\otimes\Lambda^{0,1}$ \emph{type 1 quadratic} if it can be
written as
\[
 \mathcal Q(\eta)=P(\eta)+Q(\eta)\cdot\nabla\eta,
\]
where $P$ and $Q$ are fiberwise smooth bundle maps with uniformly bounded
fiber derivatives to every order, satisfying $|P(\eta)|\le c|\eta|^2$ and
$|Q(\eta)|\le c|\eta|$.

\begin{lemma}[Type 1 quadratic estimate]\label{lem:IIF-type1}
The component nonlinearities $\mathcal Q_\pm$ are type 1 quadratic, and
every type 1 quadratic map satisfies, on a ball in $W^{1,q}$,
\begin{equation}\label{eq:IIF-quadratic}
 \|\mathcal Q(\eta)-\mathcal Q(\widetilde\eta)\|_{L^q}
 \le C\bigl(\|\eta\|_{W^{1,q}}+\|\widetilde\eta\|_{W^{1,q}}\bigr)
                   \|\eta-\widetilde\eta\|_{W^{1,q}},
\end{equation}
with $C$ depending on $c$, on the radius of the ball, and on $q$, but not
on $T$; in particular $\mathcal Q(0)=0$ and $\|\mathcal Q(\eta)\|_{L^q}\le C\|\eta\|_{W^{1,q}}^2$.
\end{lemma}
\begin{proof}
In the component chart, $\mathcal Q_\pm(\eta)$ is the difference between the
Floer operator at $\exp_{u_\pm}\eta$ and its linearization, which has the
stated form with $P,Q$ built from $J_t$, the chart maps and their
derivatives; the pointwise bounds hold because $\mathcal Q_\pm$ vanishes to
second order at $\eta=0$, and the bundle maps vanish identically on the
adapted half-strips, where the equation is linear.  For the estimate,
$P(\eta)-P(\widetilde\eta)$ and $(Q(\eta)-Q(\widetilde\eta))\cdot\nabla\eta$
are bounded pointwise by $c(|\eta|+|\widetilde\eta|)|\eta-\widetilde\eta|$
and $c|\eta-\widetilde\eta||\nabla\eta|$, and
$Q(\widetilde\eta)\cdot\nabla(\eta-\widetilde\eta)$ by
$c|\widetilde\eta||\nabla(\eta-\widetilde\eta)|$; the embedding
$W^{1,q}\hookrightarrow C^0$ for $q>2$ and the product estimate
$L^\infty\cdot L^q\subset L^q$ give \eqref{eq:IIF-quadratic}.
\end{proof}

\paragraph{The exact coupled equations.}
On the overlap the product rule for $\partial_s+A_{\cP{y}}$ produces the
two cutoff derivatives.  By \eqref{eq:IIF-cutoff-supports}, the term
$(\beta_+^T)'(V_++\xi_+)(s-T,t)$ belongs to the incoming equation, and
the term $(\beta_-^T)'(V_-+\kappa_-+\xi_-)(s+T,t)$ belongs to the outgoing
equation.  In the individual component coordinate $\tilde s$, define
\begin{equation}\label{eq:IIF-cross}
\begin{aligned}
 X_-^T(\xi_+)(\tilde s,t)
  &=(\beta_+^T)'(\tilde s-T)(V_++\xi_+)(\tilde s-2T,t),\\
 X_+^T(\kappa_-,\xi_-)(\tilde s,t)
  &=(\beta_-^T)'(\tilde s+T)(V_-+\kappa_-+\xi_-)(\tilde s+2T,t).
\end{aligned}
\end{equation}
Each term is extended by zero outside its cutoff support.  The right-hand
side is evaluated in the common adapted chart.  Set
\begin{align}
 \mathcal E_-&=D_-\xi_-+\mathcal Q_-(\kappa_-+\xi_-)+X_-^T(\xi_+),
       \label{eq:IIF-theta-minus}\\
 \mathcal E_+&=D_+\xi_++\mathcal Q_+(\xi_+)+X_+^T(\kappa_-,\xi_-).
       \label{eq:IIF-theta-plus}
\end{align}
Because $D_-\kappa_-=0$, the full error, in the indicated chart
identifications, is exactly
\begin{equation}\label{eq:IIF-theta}
 \mathcal F(u_T^{\mathrm{par}})(s,t)
 =\beta_-^T(s)\mathcal E_-(s+T,t)+\beta_+^T(s)\mathcal E_+(s-T,t).
\end{equation}
There is no extra nonlinear overlap remainder: $J$ is constant throughout
the overlap.  The component nonlinearities away from that region remain
in $\mathcal Q_\pm$.  Thus $\mathcal E_-=\mathcal E_+=0$ is sufficient for a
solution of the strip equation.

\begin{lemma}[Cross-term estimates]\label{lem:IIF-cross}
For fixed $\ell$, the cross terms satisfy
\begin{gather*}
 \|X_-^T(0)\|_{L^q}+\|X_+^T(\kappa_-,0)\|_{L^q}
   \le C_\ell e^{-\delta T}(1+|\kappa_-|),\\
 \|X_-^T(\xi_+)-X_-^T(\widetilde\xi_+)\|_{L^q}
   \le \frac{C}{\ell}\|\xi_+-\widetilde\xi_+\|_{L^q},\\
 \|X_+^T(\kappa_-,\xi_-)-X_+^T(\kappa_-,\widetilde\xi_-)\|_{L^q}
   \le \frac{C}{\ell}\|\xi_--\widetilde\xi_-\|_{L^q},\\
 \|\partial_{\kappa_-}X_+^T\|\le C_\ell e^{-\delta T}.
\end{gather*}
The linear constants $C/\ell$ are independent of $T$.
\end{lemma}
\begin{proof}
On the support of $(\beta_+^T)'$, the outgoing argument $s-T$ is at most
$-3T/2+\ell$; on the support of $(\beta_-^T)'$, the incoming argument
$s+T$ is at least $3T/2-\ell$.  The tail and kernel estimates
\eqref{eq:IIF-rates}, integrated over intervals of fixed length $2\ell$,
give the first and last bounds.  For arbitrary trial corrections no
pointwise tail estimate is used: multiplication by the cutoff derivative
has norm at most $C/\ell$, and translation preserves the $L^q$ norm.
This proves the two Lipschitz estimates.
\end{proof}

\subsubsection{Solving the component equations}\label{subsec:IIF-fixed-point}
Use the right inverses on the chosen complements and define
\begin{equation}\label{eq:IIF-contraction}
 \mathcal T_T(\xi_-,\xi_+)=
 \begin{pmatrix}
  -P_-\bigl[\mathcal Q_-(\kappa_-+\xi_-)+X_-^T(\xi_+)\bigr]\\
  -P_+\bigl[\mathcal Q_+(\xi_+)+X_+^T(\kappa_-,\xi_-)\bigr]
 \end{pmatrix}.
\end{equation}
A fixed point solves both equations and hence the strip equation by
\eqref{eq:IIF-theta}.

\begin{proposition}[Uniform contraction]\label{prop:IIF-contraction}
Choose $\ell$ sufficiently large.  There are $\rho_0>0$, $T_1>4\ell$,
and $C>0$ such that, for $T\ge T_1$ and $|\kappa_-|\le\rho_0$,
$\mathcal T_T$ has a unique fixed point in a fixed sufficiently small
neighborhood of zero in the product of complements.  It satisfies
\begin{equation}\label{eq:IIF-correction-estimates}
 \|\xi_\pm\|_{W^{1,q}}\le C(e^{-\delta T}+|\kappa_-|^2),\qquad
 \|d_{\kappa_-}\xi_\pm\|\le C(e^{-\delta T}+|\kappa_-|).
\end{equation}
The fixed point is continuous in $T$ and $C^1$ in $\kappa_-$.
Write $u_T(\kappa_-)$ for the resulting strip modulo translation.
\end{proposition}
\begin{proof}
The component operators and right inverses are fixed, independent of $T$.
On a correction ball of radius $\rho$,
the type 1 quadratic estimate \eqref{eq:IIF-quadratic} and
Lemma~\ref{lem:IIF-cross} give
\[
 \operatorname{Lip}(\mathcal T_T)
 \le C_P C\bigl(\ell^{-1}+|\kappa_-|+\rho\bigr),\qquad
 \|\mathcal T_T(0,0)\|\le C_0(e^{-\delta T}+|\kappa_-|^2).
\]
First choose $\ell$ so that the first summand in the Lipschitz bound is
at most $1/4$, then a fixed small ball and $\rho_0$ so that the other
summands make the bound at most $1/2$.  Finally enlarge $T_1$ so that the
value at zero lies in half the fixed ball.  The same map preserves the
smaller ball of radius $2C_0(e^{-\delta T}+|\kappa_-|^2)$.
The contraction theorem gives the unique fixed point and the first bound
in \eqref{eq:IIF-correction-estimates}; uniqueness holds in the original
fixed neighborhood as well.

Differentiate the fixed-point equation.  The inverse of
$1-D_\xi\mathcal T_T$ has norm at most two.  The derivative in
$\kappa_-$ comes from the component nonlinearity, bounded by
$C(|\kappa_-|+\|\xi_-\|)$, and from $X_+^T$, bounded by
$C_\ell e^{-\delta T}$.  This proves the derivative bound.  Continuity
in $T$ follows from translation continuity in the Sobolev spaces and
the uniform contraction.
\end{proof}

The two parameters have different roles: a large \emph{fixed} $\ell$
makes the linear coupling small, and a large $T$ makes the original tails
small on its supports.  Smallness of a correction ball alone would not
make a fixed linear coupling small.  The construction retains all strip
modes; the finite spectral selection enters only in the evaluation step.

\subsubsection{Evaluation of the glued strip}\label{subsec:IIF-evaluation}
All coefficient estimates in this subsection use the incoming
parametrization
\[
 (s',t)\longmapsto u_T^{\mathrm{par}}(s'-T,t).
\]
It agrees with $\exp_{u_-}(\kappa_-+\xi_-)$ for $s'\le T/2-\ell$.
The coefficients in the original glued coordinate differ by the weighted
factors $e^{-\lambda_aT}$, but the evaluation does not depend on this
translation.

\begin{lemma}[Evaluation of a nearby strip]\label{lem:IIF-readoff}
Fix a slice $s_0$ in the adapted negative end of $u_-$.  If a strip $u$
in the same incoming parametrization agrees on that end with
$\exp_{u_-}\eta$, where $\eta$ has small $W^{1,q}$ norm, then
\[
 c_{\lambda_a}(u)=e^{\lambda_as_0}\Pi_{E_a}w_u(s_0,\cdot),
       \qquad 1\le a\le r.
\]
For the decomposition $\eta=\eta^{\ker}+\eta^{\mathrm{comp}}$ along
$\ker D_-\oplus(\ker D_-)^\perp$,
\[
 \bigl|c_N(u)-c_N(u_-)-\varsigma(\eta^{\ker})\bigr|
       \le C\|\eta^{\mathrm{comp}}\|_{W^{1,q}}.
\]
In particular equality with zero remainder holds for a kernel section.
\end{lemma}
\begin{proof}
On the negative half-strip the normal equation is linear and the chart
is affine in the normal direction.  The coefficient is therefore exactly
the stated eigenprojection of the slice.  Normal parts add in that chart;
for a kernel section their selected coefficients are $\varsigma$ by
definition.  For the complementary section the bounded trace at the fixed
slice, followed by a finite-dimensional eigenprojection, bounds the
remaining coefficient by its $W^{1,q}$ norm.  Here the superscript
$\mathrm{comp}$ denotes the kernel complement, not a normal-bundle
projection.
\end{proof}

\begin{proposition}[Persistence of the evaluation]
\label{prop:IIF-persistence}
For large $T$ and small $\kappa_-$, with the incoming parametrization,
\begin{equation}\label{eq:IIF-persistence}
\begin{aligned}
 c_N(u_T(\kappa_-))&=c_N(u_-)+\varsigma(\kappa_-)+R_T(\kappa_-),\\
 |R_T(\kappa_-)|&\le C(e^{-\delta T}+|\kappa_-|^2),\\
 |d_{\kappa_-}R_T(\kappa_-)|&\le C(e^{-\delta T}+|\kappa_-|).
\end{aligned}
\end{equation}
The coefficient vector remains nonzero.  The evaluation is $C^1$ in
$\kappa_-$ and converges in $C^1$ on a smaller parameter ball to the
evaluation of the local unbroken family $u_-(\kappa_-)$ as $T\to\infty$.
\end{proposition}
\begin{proof}
For large $T$, the fixed slice $s_0$ is before the outgoing cutoff.
Apply Lemma~\ref{lem:IIF-readoff} with $\eta=\kappa_-+\xi_-$ and the
bounds of Proposition~\ref{prop:IIF-contraction}.  The resulting bounded
linear map $\Pi_{E_N^-}$ gives both estimates in \eqref{eq:IIF-persistence}.
Since $c_N(u_-)\ne0$, shrinking the parameter ball and increasing $T_1$
keep the vector away from zero.

For the last claim, solve the uncoupled equation
\[
 \xi_-^\infty=-P_-\mathcal Q_-(\kappa_-+\xi_-^\infty)
\]
on the same complement.  Its solution defines
$u_-(\kappa_-)=\exp_{u_-}(\kappa_-+\xi_-^\infty)$, a local family of
unbroken strips.  Its adapted positive tails, and their parameter
derivatives, decay uniformly.  Substitution of $(\xi_-^\infty,0)$ in
\eqref{eq:IIF-contraction} leaves a residual of size $O(e^{-\delta T})$,
including its first $\kappa_-$ derivative, by the cross-term estimates
on these actual tails.  Uniform contraction and differentiation of the
fixed-point equation imply
\[
 \|\xi_-(T,\cdot)-\xi_-^\infty\|_{C^1(W^{1,q})}
       +\|\xi_+(T,\cdot)\|_{C^1(W^{1,q})}
       \le Ce^{-\delta T}.
\]
The trace formula and smoothness of $q_N$ on nonzero vectors
now prove the asserted convergence.
\end{proof}

\subsubsection{The incidence equation}\label{subsec:IIF-constraint}
The remaining condition on the glued strip is the incidence
$\Ev^-_{\cP{p},N}(u_T(\kappa_-))\in\descending_{h_N}(\cP{v})$, with the
kernel parameter $\kappa_-$ as the unknown.  Recall from
Lemma~\ref{lem:IIF-linear-T2} that $z=\Ev^-_{\cP{p},N}(u_-)$ is the
evaluation of the incoming component of the \emph{broken} configuration
and lies on the cell, and that the composite $\mathcal L$ of the
coefficient map with the quotient by $T_z\descending_{h_N}(\cP{v})$ is an
isomorphism from $\mathcal K$ onto the normal space
$T_zS_N/T_z\descending_{h_N}(\cP{v})$ of the cell at $z$, both of
dimension $c(\cP{v})$.  For finite
$T$ the glued strip at $\kappa_-=0$ is not the broken configuration: its
incoming piece is $\exp_{u_-}(\xi_-(T,0))$ with $\xi_-(T,0)=O(e^{-\delta T})$
nonzero, so its evaluation is $O(e^{-\delta T})$-close to $z$ but not on
the cell, and $\kappa_-$ must be adjusted.

The adjustment is found by the following quantitative form of the inverse
function theorem.

\begin{lemma}[Quantitative inverse function theorem]\label{lem:IIF-transverse-ift}
Let $F\colon B_R\to\R^{c}$ be a $C^1$ map on the closed ball of radius $R$
about $0$ in $\R^c$, and let $L\colon\R^c\to\R^c$ be an invertible linear
map.  Suppose that for some $\eta,\theta\ge0$
\[
    |F(0)|\le\eta,
    \qquad
    \|dF(x)-L\|\le\theta\ \text{ for all }x\in B_R,
    \qquad
    \|L^{-1}\|\,\theta\le\tfrac12,
    \qquad
    2\|L^{-1}\|\,\eta\le R .
\]
Then there is exactly one $x\in B_R$ with $F(x)=0$, and it satisfies
$|x|\le2\|L^{-1}\|\,\eta$.  If $F$ depends continuously on a parameter and
the hypotheses hold uniformly, the solution depends continuously on the
parameter.
\end{lemma}

The lemma is the contraction principle applied to $x\mapsto x-L^{-1}F(x)$,
a $\tfrac12$-contraction of $B_R$ into itself under the stated
inequalities.

\begin{proposition}[Unique solution of the incidence]\label{prop:IIF-ift}
For $T\ge T_1$ sufficiently large there is a unique $\kappa_-(T)$ in a
fixed small ball of $\mathcal K$ with
$\Ev^-_{\cP{p},N}(u_T(\kappa_-(T)))\in\descending_{h_N}(\cP{v})$.  It is
continuous in $T$ and satisfies $|\kappa_-(T)|\le Ce^{-\delta T}$.
\end{proposition}
\begin{proof}
\emph{The degenerate case.}  If $c=0$, that is, if $\cP{v}$ is a maximum
of $h_N$ and its cell is open in $S_N$, then $\mathcal K=0$, the incidence
holds for large $T$ because $\Ev^-_{\cP{p},N}(u_T(0))$ is within
$O(e^{-\delta T})$ of the interior point $z$ of the cell, and
$\kappa_-(T)=0$.  Assume $c>0$ from now on.

\emph{Coordinates.}  Take coordinates $(y_1,\dots,y_{N-1})$ on a
neighborhood of $z$ in $S_N$ in which the descending cell is the coordinate
subspace $\{y_{j+1}=\dots=y_{N-1}=0\}$, $j=j(\cP{v})$, with $z$ at the
origin.  Such coordinates exist for any submanifold, and the last $c$ of
them identify the normal space $T_zS_N/T_z\descending_{h_N}(\cP{v})$ with
$\R^c$.  Let $F(\kappa_-)\in\R^c$ be the last $c$ coordinates of
$\Ev^-_{\cP{p},N}(u_T(\kappa_-))$, for $\kappa_-$ in a ball
$B_R\subset\mathcal K$.  The glued strip evaluates into the cell exactly
when $F(\kappa_-)=0$; this condition does not depend on the coordinates.

\emph{Estimates.}  By Proposition~\ref{prop:IIF-persistence} the evaluation
at $\kappa_-=0$ is $z+O(e^{-\delta T})$ and its derivative is
$\mathcal L$ up to $O(e^{-\delta T}+|\kappa_-|)$, that is,
\[
    |F(0)|\le C_1e^{-\delta T},
    \qquad
    \|dF(\kappa_-)-\mathcal L\|\le C_2\bigl(e^{-\delta T}+|\kappa_-|\bigr)
    \quad\text{on }B_R,
\]
and $\mathcal L$ is invertible by Lemma~\ref{lem:IIF-linear-T2}.

\emph{Choice of constants.}  Fix $R$ with $\|\mathcal L^{-1}\|C_2\,R\le\frac14$,
and then $T_1$ with $\|\mathcal L^{-1}\|C_2\,e^{-\delta T_1}\le\frac14$ and
$2\|\mathcal L^{-1}\|C_1\,e^{-\delta T_1}\le r$.

\emph{Conclusion.}  For $T\ge T_1$ the hypotheses of
Lemma~\ref{lem:IIF-transverse-ift} hold with $L=\mathcal L$,
$\eta=C_1e^{-\delta T}$ and $\theta=C_2(e^{-\delta T}+R)$.  The lemma gives
the unique zero $\kappa_-(T)\in B_R$, the bound
$|\kappa_-(T)|\le2\|\mathcal L^{-1}\|C_1e^{-\delta T}$, and continuity in $T$.
\end{proof}

Let $\vartheta^T$ be the sphere half-trajectory ending at
$\Ev^-_{\cP{p},N}(u_T(\kappa_-(T)))$ in the descending cell of $\cP{v}$.
The local gluing map is
\[
 \mathcal G(u_-,u_+,T)=\bigl(\vartheta^T,u_T(\kappa_-(T))\bigr).
\]
Only the displacement transverse to the descending cell is constrained by
the incidence; the displacement along the cell determines the new
half-trajectory.

\subsubsection{Recovery, uniqueness, and the collar}\label{subsec:IIF-proof}
To prove that the map is a collar, we recover the component maps of every
nearby unbroken strip.  The exact linear neck and the separated cutoffs
make this reconstruction explicit.

\begin{lemma}[Solutions on a linear neck]
\label{lem:IIF-neck-estimates}
Let $A=A_{\cP{y}}=J_{\cP{y}}\,d/dt$ on paths $[0,1]\to T_{\cP{y}}\widehat M$
with boundary values in $T_{\cP{y}}\widehat L_0$ and $T_{\cP{y}}\widehat L_1$,
and let $U$ be a smooth solution of $(\partial_s+A)U=0$ on
$[a,b]\times[0,1]$.  Then $U$ is the finite-or-convergent sum
\[
    U(s,t)=\sum_{\lambda\in-\pi/2+\pi\Z}c_\lambda\,e^{-\lambda s}e^{-\sqrt{-1}\lambda t},
    \qquad c_\lambda\in T_{\cP{y}}\widehat L_0 ,
\]
and its stable and unstable parts
$\Pi^{\mathrm s}U=\sum_{\lambda>0}$, $\Pi^{\mathrm u}U=\sum_{\lambda<0}$
are again solutions, with
$\Pi^{\mathrm s}U(s)=e^{-(s-a)A}\Pi^{\mathrm s}U(a)$ and
$\Pi^{\mathrm u}U(s)=e^{-(s-b)A}\Pi^{\mathrm u}U(b)$.  For every unit
sub-strip $[s_0,s_0+1]\times[0,1]$ with $s_0-a\ge1$, respectively
$b-s_0-1\ge1$,
\[
    \|\Pi^{\mathrm s}U\|_{W^{1,q}([s_0,s_0+1]\times[0,1])}
    \le C\,e^{-\frac{\pi}{2}(s_0-a)}\,\|U(a,\cdot)\|_{L^2[0,1]},
\] 
\[
    \|\Pi^{\mathrm u}U\|_{W^{1,q}([s_0,s_0+1]\times[0,1])}
    \le C\,e^{-\frac{\pi}{2}(b-s_0-1)}\,\|U(b,\cdot)\|_{L^2[0,1]},
\]
with $C$ depending only on $q$.  In particular $\delta$ of
\eqref{eq:IIF-rates} may be taken to be any number below $\pi/4$.
\end{lemma}
\begin{proof}
In the chart the equation is $\partial_sU+J_{\cP{y}}\partial_tU=0$ with
$J_{\cP{y}}$ constant and $T_{\cP{y}}\widehat L_1=J_{\cP{y}}T_{\cP{y}}\widehat L_0$,
so each slice $U(s,\cdot)$ expands in the eigenfunctions
$e^{-\sqrt{-1}\lambda t}c$, $c\in T_{\cP{y}}\widehat L_0$, of $A$, with
$\lambda\in\pi/2+\pi\Z$ because the boundary condition at $t=1$ forces
$\cos\lambda=0$, the coefficients solving
$\dot c_\lambda(s)=-\lambda c_\lambda(s)$; this is the expansion, and the
formulas for the two parts follow.  For the estimate, on the sub-strip
$|\partial_s^i\partial_t^k(c_\lambda e^{-\lambda s}e^{-\sqrt{-1}\lambda t})|
\le|\lambda|^{i+k}|c_\lambda|e^{-\lambda(s_0-a)}e^{-\lambda a}$, and
$c_\lambda e^{-\lambda a}$ are the Fourier coefficients of $U(a,\cdot)$;
Cauchy--Schwarz against $\sum_{\lambda>0}(1+\lambda)^2e^{-2\lambda(s_0-a)}$,
bounded for $s_0-a\ge1$ by $Ce^{-\pi(s_0-a)}$ since the smallest positive
eigenvalue is $\pi/2$, gives the $C^1$ bound on the sub-strip, hence the
$W^{1,q}$ bound.  The unstable part is the same with $b$ in place of $a$.
\end{proof}

Only these two facts are used below: the spectral parts of a neck
solution are solutions, and a spectral part is exponentially small, in
$W^{1,q}$ and uniformly in the neck length, at unit distance from the slice
where it is read.  The reconstructed pieces are never measured closer than
that to their own slice, because there the cutoff multiplying them
vanishes.

\paragraph{Component reconstruction.}
Let $u$ be an actual strip sufficiently close in the broken topology to
$(u_-,u_+)$.  A long middle neck lies in the adapted chart and
$U=\Phi_{\cP{y}}^{-1}u$ satisfies
$(\partial_s+A_{\cP{y}})U=0$, so Lemma~\ref{lem:IIF-neck-estimates}
applies.  Once the two translations have been chosen,
use the cutoffs to define on this neck
\begin{equation}\label{eq:IIF-reconstruction}
\begin{aligned}
 \Phi_{\cP{y}}^{-1}a_-(s+T,t)
   &=\Pi^{\mathrm s}U(s,t)+(1-\beta_+^T(s))\Pi^{\mathrm u}U(s,t),\\
 \Phi_{\cP{y}}^{-1}a_+(s-T,t)
   &=\Pi^{\mathrm u}U(s,t)+(1-\beta_-^T(s))\Pi^{\mathrm s}U(s,t).
\end{aligned}
\end{equation}
At least one cutoff equals one at every point.  Multiplying these formulas
by $\beta_-^T$ and $\beta_+^T$ and adding recovers $U$ exactly.
Since the spectral projections commute with the linear neck equation,
differentiation gives \eqref{eq:IIF-theta-minus} and
\eqref{eq:IIF-theta-plus} with zero right-hand sides.  For example the
incoming derivative is $-(\beta_+^T)'\Pi^{\mathrm u}U$, and on that support
the outgoing coordinate is exactly $\Pi^{\mathrm u}U$, so the cross term
cancels it.

The incoming map agrees with $u$ on the left and has only stable modes
on the right; extend the latter by its decaying positive-end solution.
The outgoing map extends by unstable modes to its negative end and agrees
with $u$ on the right.  Together with the unchanged outer pieces these
are whole-strip component maps.  The estimates of
Lemma~\ref{lem:IIF-neck-estimates} show that they are small $W^{1,q}$
deformations of $u_\pm$ as the broken limit is approached.  The decomposition is unique among component maps satisfying
the coupled equations, the reconstruction identity, and the prescribed
end decay: in the middle, their stable and unstable parts are forced to
be $\Pi^{\mathrm s}U$ and $\Pi^{\mathrm u}U$, respectively, and the
coupled linear equations determine the transition regions.

\paragraph{Translations and surjectivity.}
Choose the two translations by requiring
\[
 \langle\eta_-,\partial_su_-\rangle_{L^2}=0,
 \qquad\langle\eta_+,\partial_su_+\rangle_{L^2}=0,
 \qquad a_\pm=\exp_{u_\pm}\eta_\pm.
\]
At the broken limit the derivative of the phase equations is diagonal
with nonzero entries $\pm\|\partial_su_\pm\|_{L^2}^2$; the cross terms
and cutoff terms tend to zero by the tail estimates.  The implicit
function theorem fixes the translations uniquely near the chosen representatives.
Their relative shift is $2T$ and their common shift is the irrelevant
translation of the unparametrized strip.  Set
$\kappa_- =\operatorname{proj}_{\mathcal K}\eta_-$,
$\xi_- =\eta_- -\kappa_-$, and $\xi_+=\eta_+$.
The reconstructed corrections lie in the chosen complements and satisfy
the coupled equations, so Proposition~\ref{prop:IIF-contraction} identifies
them with its unique fixed point.  If $u$ obeys the incidence, its kernel
parameter is the unique solution of Proposition~\ref{prop:IIF-ift}.  Therefore every nearby incidence configuration
is in the image of $\mathcal G$.

\paragraph{Injectivity and the boundary.}
The same phase normalization recovers $T$ and the two component maps from
the glued strip, so different gluing data in the local neighborhood do
not have the same image.  The correction and incidence estimates show
that the glued family converges to $(u_-,u_+)$ as $T\to\infty$.
The reconstruction depends continuously on the strip in the broken
topology.  Thus adjoining $T=\infty$ gives a homeomorphism onto a
neighborhood of this boundary point.  The parameter $e^{-\delta T}$
is a boundary coordinate.  Allowing small component parameters in the
same construction gives the usual linear gluing identification of the
kernels and determinant lines; its orientation is calculated in
Subsection~\ref{subsec:IIF-orientations}.

\paragraph{Assembly over the stratum.}
The stratum $Z$ is finite by Proposition~\ref{prop:IIF-strata}.  Choose
the cutoff width large enough for all its pairs, and choose $T_1$ and the
small neighborhoods uniformly by taking finitely many maxima and minima
of the local constants.  Distinct broken points have disjoint sufficiently
small neighborhoods in the broken topology; increasing $T_1$ keeps the
collars disjoint.  Choose representatives, charts and complements at one
representative of each $G$-orbit and transport them by $G$.  Uniqueness of
the two contractions then makes the resulting maps equivariant.
This assembles the maps on $Z\times[T_1,\infty)$ and proves
Proposition~\ref{thm:IIF-outside-gluing}, including the collar assertion.

\subsection{Ordinary strip breaking and invariant intermediate points}
\label{subsec:IIF-invariant-gluing}
The preceding ambient construction does not depend on the outside nature
of the intermediate point, except to specify its adapted chart and the
source of regularity.  We state its applications to the other strip
boundary types, keeping the incidence on the appropriate component.

\subsubsection{Ordinary strip spaces and fixed outgoing strips}
\label{subsec:IIF-fixed-continuation}
\begin{proposition}[Ordinary strip gluing]\label{prop:IIF-ordinary-gluing}
For a pair of ordinary rigid strips and any intermediate intersection
$\cP{y}$ there is a collar map
\[
 \M^0(\cP{x},\cP{y})\times\M^0(\cP{y},\cP{b})\times[T_1,\infty)
          \longrightarrow\M^1(\cP{x},\cP{b}),
\]
with the boundary inclusion as its extension at infinity.  The same
statement holds when a factor is restricted to fixed strips.
\end{proposition}
\begin{proof}
Both kernels consist only of their translation directions.  In the
chosen slices there is no kernel parameter and no incidence equation.
The pregluing, the coupled contraction with $\kappa_-=0$, and the exact
reconstruction above therefore give the collar.  At an invariant
intersection use the chart of Corollary~\ref{cor:lagrangian-preparation-hat}; at an outside one
use the chart of Convention~\ref{conv:lagrangian-preparation} at an outside point.  A fixed factor is
regular because its tangential operator is regular and its normal
operator is invertible.  If both factors are fixed, the construction
remains in $M^G$, with zero normal correction.  Assembly and orientations
are as in the outside-break construction.
\end{proof}

In particular a non-fixed incoming strip followed by a fixed strip gives
\begin{equation}\label{eq:IIF-fixed-outgoing-map}
 \M^0(\cP{x},\cP{q})\times\M_{M^G}^0(\cP{q},\cP{b})
       \times[T_1,\infty)\longrightarrow\M^1(\cP{x},\cP{b}).
\end{equation}
This is the contribution $\Dinv\D{o}{s}$: the incoming strip lands on
$\Theta_{\cP{q}}$, and the fixed strip transports that cycle with sign
$+1$.  A non-fixed rigid outgoing continuation is instead partitioned by
its top cells (Lemma~\ref{lem:floer-top-cells}) and gives $\Dnon\D{o}{s}$.

\begin{corollary}[Fixed continuation with a source incidence]
\label{prop:IIF-fixed-continuation}\label{cor:IIF-fixed-continuation-incidence}
For invariant $\cP{p},\cP{q},\cP{b}$ there is a collar map
\[
 \mathcal N^0_{\cP{p},N}(\cP{v};\cP{q})
 \times\M_{M^G}^0(\cP{q},\cP{b})\times[T_1,\infty)
       \longrightarrow\mathcal N^1_{\cP{p},N}(\cP{v};\cP{b}).
\]
Its extension at infinity is the inclusion of the whole broken stratum.
\end{corollary}
\begin{proof}
The incoming kernel and its incidence linearization are unchanged from
Subsection~\ref{subsec:IIF-linear}.  The outgoing fixed strip has a
surjective linearization and only the translation kernel.  The neck is
linear, the evaluation is read near the incoming source, and every step
of the contraction, evaluation, incidence and recovery proof applies to
these operators.  This is type \textup{(S4a)} with an incidence source,
contributing $\Dinv\Dnon$.
\end{proof}

\subsubsection{Fixed incoming strip}\label{subsec:IIF-fixed-incoming}
Let $\cP{p},\cP{q}$ be invariant and $\cP{b}$ arbitrary.  Consider the
stratum
\[
 Z_{\mathrm{in}}=\M_{M^G}^0(\cP{p},\cP{q})
                    \times\mathcal N^0_{\cP{q},N}(\cP{v};\cP{b}).
\]
The normal coefficients at the original source tend to zero when the
incoming level is fixed.  The incidence is nevertheless controlled by
reading the weighted ray on the outgoing side of the fixed stretch.

\begin{proposition}[Gluing with a fixed incoming strip]
\label{thm:IIF-fixed-incoming}
There is a collar map
\[
 \mathcal G_{
 \mathrm{in}}:Z_{\mathrm{in}}\times[T_1,\infty)
       \longrightarrow\mathcal N^1_{\cP{p},N}(\cP{v};\cP{b}),
\]
with extension at infinity the inclusion of $Z_{\mathrm{in}}$, and with
the same injectivity, exhaustion and equivariance properties as
Proposition~\ref{thm:IIF-outside-gluing}.  Its boundary sign is the product
sign with $(-1)^{N-1-j(\cP{v})}$ on the fixed incoming factor.
\end{proposition}
\begin{proof}
Fix $(u_-,u_+)\in Z_{\mathrm{in}}$.  The fixed operator splits as
$D_-=D_-^T\oplus D_-^\perp$, with $D_-^\perp$ invertible and
$\ker D_-^T=\R\partial_su_-$.  There is therefore no kernel parameter
on that component.  On the outgoing component set
\[
 \mathcal K_+=\ker D_+\cap(\partial_su_+)^\perp,
 \qquad\dim\mathcal K_+=N-1-j(\cP{v}).
\]
Let $\varsigma_+$ be the selected normal coefficient map at its negative
end.  With $z=\Ev^-_{\cP{q},N}(u_+)$, incidence transversality makes
\[
 \mathcal L_+:\mathcal K_+\xrightarrow{\varsigma_+}E_N^-
       \xrightarrow{dq_N}T_zS_N\longrightarrow T_zS_N/T_z\descending_{h_N}(\cP{v})
\]
an isomorphism, exactly as in Lemma~\ref{lem:IIF-linear-T2}.

\emph{Coupled equations.}
Use the same separated cutoffs and replace the component maps by
$\exp_{u_-}\xi_-$ and $\exp_{u_+}(\kappa_++\xi_+)$.
The two equations now read, in their own coordinates,
\[
\begin{aligned}
 \mathcal E_-&=D_-\xi_-+\mathcal Q_-(\xi_-)+
    (\beta_+^T)'(\,\cdot-T)(V_++\kappa_++\xi_+)(\,\cdot-2T),\\
 \mathcal E_+&=D_+\xi_++\mathcal Q_+(\kappa_++\xi_+)+
    (\beta_-^T)'(\,\cdot+T)(V_-+\xi_-)(\,\cdot+2T).
\end{aligned}
\]
The $t$ coordinate is unchanged in these expressions.
The cross-term bounds use the same $C/\ell$ linear estimates and the
outgoing kernel's exponentially decaying negative tail.  Since
$D_+\kappa_+=0$, the contraction gives
\[
 \|\xi_\pm\|\le C(e^{-\delta T}+|\kappa_+|^2),\qquad
 \|d_{\kappa_+}\xi_\pm\|\le C(e^{-\delta T}+|\kappa_+|).
\]

\emph{Transport and evaluation.}
Choose a fixed slice $s_1$ on the adapted negative end of $u_+$.
For large $T$, the corresponding glued slice is beyond the incoming
cutoff and the strip before it lies in the adapted tube, shadowing the
fixed incoming level and the neck.  The exact normal equation implies
\[
 \Ev^-_{\cP{p},N}(u_T)
    =q_N\bigl(\Pi_{E_N^-}w_{u_T}(s_1,\cdot)\bigr),
\]
where the slice is written in outgoing coordinates.  After multiplication
by the fixed factors $e^{\lambda_as_1}$, these selected slice
values give a vector
\[
 \widetilde c_T(\kappa_+)
   =c_N(u_+)+\varsigma_+(\kappa_+)+R_T(\kappa_+),
\]
with the two estimates of \eqref{eq:IIF-persistence}.  Weighted
normalization is unchanged by that common time adjustment, so the original
source evaluation is $q_N(\widetilde c_T)$.

\emph{Incidence and collar.}
As in the proof of Proposition~\ref{prop:IIF-ift}, take coordinates near
$z=\Ev^-_{\cP{q},N}(u_+)$ in which the descending cell of $\cP{v}$ is a
coordinate subspace, and let $F(\kappa_+)$ be the transverse coordinates of
$q_N(\widetilde c_T(\kappa_+))$.  Then $|F(0)|=O(e^{-\delta T})$ and
$dF=\mathcal L_++O(e^{-\delta T}+|\kappa_+|)$, so
Lemma~\ref{lem:IIF-transverse-ift} gives a unique
$\kappa_+(T)=O(e^{-\delta T})$ with $q_N(\widetilde c_T(\kappa_+(T)))$ on
the cell; if the cell is a top cell the incidence is automatic and
$\kappa_+(T)=0$.
The reconstructed neck pieces, phase conditions and uniqueness of the
component equations are those of
\eqref{eq:IIF-reconstruction}, now with the kernel slice on the outgoing
factor.  They give the collar and exhaustion.  Assemble over the finite
stratum.  The orientation factor is computed in
Subsection~\ref{subsec:IIF-orientations}.
\end{proof}

Exact transport, not an additive comparison of very small coefficients,
is essential here.  Different normal modes at the original source decay
at different rates through the fixed stretch.  The same weighted ray can
instead be read on the outgoing slice, where its selected vector has a
nonzero limit.  This uses the exact normal equation along the whole
adapted tube, not just its linearization at an intersection.

\subsubsection{Both strips non-fixed}\label{subsec:IIF-both-nonfixed}
For invariant $\cP{p},\cP{q}$ and arbitrary $\cP{b}$ let
\[
 Z_{\mathrm{nf}}=\mathcal N^0_{\cP{p},N}(\cP{v};\cP{q})
                    \times\M^{\circ,0}(\cP{q},\cP{b}).
\]
The top-cell partition identifies the second factor with the disjoint
union of its incidences at $M_g$.  This is a partition of
rigid outgoing strips, not an extra positive-end condition on the
incoming strip.

\begin{proposition}[Invariant break with two non-fixed components]\label{thm:IIF-both-nonfixed}
There is an equivariant collar map
\[
 Z_{\mathrm{nf}}\times[T_1,\infty)
       \longrightarrow\mathcal N^1_{\cP{p},N}(\cP{v};\cP{b})
\]
with boundary extension the inclusion of the stratum and product boundary
orientations.  It has the exhaustion and injectivity properties of
Proposition~\ref{thm:IIF-outside-gluing}.
\end{proposition}
\begin{proof}
Both operators are regular.  The incoming incidence gives the kernel
slice $\mathcal K$ and its invertible normal derivative, while the
outgoing ordinary rigid strip has only the translation kernel.
The adapted chart at $\cP{q}$ has constant $J$ and linear boundary
conditions.  Its full positive spectrum describes the incoming tail and
its full negative spectrum the outgoing tail; the latter is not truncated
to $E_N^-$.  The separated-cutoff equations, the contraction, the evaluation at the
original source $\cP{p}$, and the incidence step of
Proposition~\ref{prop:IIF-ift}, solved by
Lemma~\ref{lem:IIF-transverse-ift}, therefore apply verbatim.  Exact reconstruction gives every nearby strip and
uniqueness, and the finite stratum supplies a uniform $T_1$.
The outgoing top-cell label occurs only at the broken boundary; the
glued strip has no evaluation at the intermediate point.  Thus the
partition introduces neither an extra gluing parameter nor an additional
constraint.  These are the terms $\D{s}{o}\Dnon$ and $\Dnon\Dnon$
according to the final target.
\end{proof}

\subsection{Sphere breaking}\label{subsec:IIF-sphere-breaking}
Here the ambient strip does not break.  Let $j(\cP{w})=j(\cP{v})-1$,
and fix a boundary configuration
\[
 (\vartheta_-,u_+)\in
 \M_{h_N}^0(\cP{v},\cP{w})
       \times\mathcal N^0_{\cP{p},N}(\cP{w};\cP{b}).
\]
The ambient component $u_+$ belongs to a rigid incidence configuration;
it is not asserted to be an ordinary rigid strip.  Put
$z=\Ev^-_{\cP{p},N}(u_+)\in\descending(\cP{w})$.

\begin{lemma}[The descending-cell sheets]\label{lem:IIF-sheets}
Near $z$, each rigid sphere trajectory
$\vartheta\in\M_{h_N}^0(\cP{v},\cP{w})$ determines a local
sheet $\Sigma_\vartheta$ of the compactified descending cell of $\cP{v}$.
It is an embedded manifold with boundary $\descending(\cP{w})$ near $z$,
and
\[
 T_z\Sigma_\vartheta=T_z\descending(\cP{w})\oplus\R a_\vartheta,
 \qquad a_\vartheta\notin T_z\descending(\cP{w}).
\]
Its interior consists of endpoints of half-trajectories close to the
broken sphere trajectory, and is contained in $\descending(\cP{v})$.
The interiors of the sheets belonging to different rigid trajectories
are disjoint.
\end{lemma}
\begin{proof}
Apply the ordinary gluing of a rigid sphere trajectory to a half-trajectory
from $\cP{w}$.  The local Morse coordinates at $\cP{w}$ provide one
inward parameter for the incoming stable direction, in addition to the
coordinates on its descending cell.  Flowing to the endpoint near $z$
gives the stated sheet and tangent splitting.  A point of the descending
cell has a unique backwards half-trajectory, so the sheet interiors are
disjoint after restricting the local neighborhoods.  This is the
codimension-one part of the smooth compactification of the auxiliary
Morse trajectory spaces \cite{wehrheim2012smooth}.
\end{proof}

\begin{lemma}[The evaluation is transverse to the cell]\label{lem:IIF-slice}
The evaluation restricts to an embedding of a neighborhood of $u_+$ in
its ambient component of $\M^\circ(\cP{p},\cP{b})$ onto a submanifold
$\Sigma\subset S_N$ satisfying
\[
 T_z\Sigma\oplus T_z\descending(\cP{w})=T_zS_N.
\]
\end{lemma}
\begin{proof}
Incidence transversality at $\cP{w}$ gives a surjective normal evaluation
derivative.  Since the incidence is zero-dimensional, the ambient
dimension equals the rank of that normal space.  The derivative is
therefore injective as a map into $T_zS_N$ as well.  The
constant-rank theorem gives the local embedding and the tangent splitting.
\end{proof}

\begin{proposition}[Sphere-breaking gluing]\label{prop:IIF-sphere-breaking}
For each pair $\cP{v},\cP{w}$ with $j(\cP{w})=j(\cP{v})-1$, there is a
collar map
\[
 \M_{h_N}^0(\cP{v},\cP{w})
 \times\mathcal N^0_{\cP{p},N}(\cP{w};\cP{b})\times[T_1,\infty)
 \longrightarrow\mathcal N^1_{\cP{p},N}(\cP{v};\cP{b}).
\]
It extends by inclusion of the broken stratum at infinity, is injective
onto a neighborhood of that stratum, and has the product boundary sign.
The maps are equivariant and $T_1$ can be chosen uniformly over the
finite stratum.
\end{proposition}
\begin{proof}
The transverse slice $\Sigma$ meets $\Sigma_\vartheta$ and its boundary
transversally.  Their intersection has dimension
\[
 j(\cP{v})+(N-1-j(\cP{w}))-(N-1)=1,
\]
and has boundary $\Sigma\cap\descending(\cP{w})=\{z\}$ locally.
Pulling the punctured arc back through the local evaluation embedding
gives an arc in the desired incidence space.  Its sphere half-trajectory
approaches $\vartheta$ followed by the half-trajectory from $\cP{w}$.
Conversely, every nearby incidence configuration has endpoint on the appropriate sheet
and its ambient strip in the evaluation neighborhood, so it lies on
this arc.  Distinct sheets have disjoint interiors and give distinct
boundary collars.  Parametrize the arcs by the sphere gluing length,
then assemble over the finite set of boundary configurations.  Equivariant
choices and uniqueness give equivariance.  The product sign is verified
below from the sphere-cell boundary orientation.
\end{proof}

Ordinary one-dimensional sphere trajectory spaces compactify by two rigid
sphere trajectories and have the standard collars and product signs of
Lemma~\ref{lem:floer-sphere-boundary}.  There is no strip contraction in
this boundary type; its new incidence is obtained by the transverse
intersection just described.

\subsection{Boundary orientations and completion}\label{subsec:IIF-orientations}
\label{subsec:floer-gluing-conclusion}
The component gluing is compatible with the determinant-line isomorphisms
\eqref{eq:floer-gluing-isomorphism}.  Linearizing the parametrized ambient
gluing identifies the glued kernel with the kernels of its two components,
incoming factor first.  The common translation is the diagonal direction;
increasing the neck length is the relative translation
$\partial_su_- -\partial_su_+$, which is outward at the broken boundary.
We first derive the strip and sphere boundary signs from these facts and
then check the incidence signs.

\begin{proof}[Proof of Lemma~\ref{lem:floer-boundary-formulas}]
Let $[m_1]$, $[m_2]$ be oriented bases of $T\M(\cP{x},\cP{y})$ and
$T\M(\cP{y},\cP{z})$ at the broken strip $(u_1,u_2)$, of dimensions
$d_1,d_2$, and let $\partial^{(1)},\partial^{(2)}$ denote the images in the
glued kernel of the translation vectors $\partial_su_1,\partial_su_2$.  By
Convention~\ref{conv:floer-strip-orientation}, $\det D_{u_i}=[m_i,\partial^{(i)}]$,
and by the compatibility of \eqref{eq:floer-gluing-isomorphism} with
breaking, incoming factor first, the kernel of the linearized operator at
a glued strip with long neck is $[m_1,\partial^{(1)},m_2,\partial^{(2)}]$ as
an oriented space.  The preglued strip is $u_1(s+T)$ on the left and
$u_2(s-T)$ on the right, so its translation is
$\partial_su=\partial^{(1)}+\partial^{(2)}$ and increasing the neck length
is $n_{\mathrm{out}}=\partial^{(1)}-\partial^{(2)}$, up to terms vanishing
as $T\to\infty$; since the collar reaches the stratum at $T=\infty$,
$n_{\mathrm{out}}$ is outward.  In the basis $(\partial^{(1)},\partial^{(2)})$
the pair $(n_{\mathrm{out}},\partial_su)$ has matrix
$\bigl(\begin{smallmatrix}1&1\\-1&1\end{smallmatrix}\bigr)$ of determinant
$2>0$, so $[\partial^{(1)},\partial^{(2)}]=[n_{\mathrm{out}},\partial_su]$.
Moving $\partial^{(1)}$ past the $d_2$ vectors of $m_2$,
\[
 \ker D_u=(-1)^{d_2}[m_1,m_2,\partial^{(1)},\partial^{(2)}]
 =(-1)^{d_2}[m_1,m_2,n_{\mathrm{out}},\partial_su],
\]
so by Convention~\ref{conv:floer-strip-orientation}, moduli first and
translation second,
\[
 T\M(\cP{x},\cP{z})=(-1)^{d_2}[m_1,m_2,n_{\mathrm{out}}]
 =(-1)^{d_2}(-1)^{d_1+d_2}[n_{\mathrm{out}},m_1,m_2]
 =(-1)^{d_1}[n_{\mathrm{out}}]\oplus[m_1,m_2].
\]
With the outward normal first, the boundary orientation is $(-1)^{d_1}$
times the product orientation.  For $d_1=d_2=0$ a boundary point of a
one-dimensional strip space has sign $\epsilon(u_1)\epsilon(u_2)$; the
signed count of the boundary of a compact oriented one-manifold vanishes,
and summing over $\cP{y}$ gives $d_{\mathrm F}^2=0$.
\end{proof}

\begin{proof}[Proof of Lemma~\ref{lem:floer-sphere-boundary}]
Both statements are local near $\cP{w}$.  In the standard quadratic
coordinates $(x,y)\in\R^{j(\cP{w})}\times\R^{n-j(\cP{w})}$ at $\cP{w}$,
$n=N-1$, with $h_N=h_N(\cP{w})-|x|^2/2+|y|^2/2$ and the constant metric,
the negative gradient flow is $(x,y)\mapsto(e^tx,e^{-t}y)$, so
$\descending(\cP{w})=\{y=0\}$ and $\ascending(\cP{w})=\{x=0\}$.  Write
$d=\dim\M_{h_N}(\cP{v},\cP{w})$.  Near a point of
$\descending(\cP{v})\cap\ascending(\cP{w})$ the cell $\descending(\cP{v})$
is flow invariant and transverse to $\ascending(\cP{w})$, so in the chart
it is $\R^{j(\cP{w})}_x\times C$ for a flow-invariant cone
$C\subset\R^{n-j(\cP{w})}_y$ of dimension $d+1$ whose rays are the
trajectories from $\cP{v}$ to $\cP{w}$; write $y=rm$ with $r=|y|>0$ and
$m$ in the slice $C\cap S^{n-j(\cP{w})-1}\cong\M_{h_N}(\cP{v},\cP{w})$.
The compactified cell adds $r=0$, the stratum
$\M_{h_N}(\cP{v},\cP{w})\times\descending(\cP{w})$, so $-\partial_r$ is an
outward normal; the flow direction on $C$ is $\dot y=-y=-r\partial_r$, so
the outward normal is the flow direction $\dot\vartheta$ towards $\cP{w}$.
Let $[\mu]$ be an oriented basis of $T\M_{h_N}(\cP{v},\cP{w})$, realized by
the slice directions, and $[\xi]$ an oriented basis of
$T\descending(\cP{w})=\R^{j(\cP{w})}_x$, which represents the quotient
$TS_N/T\ascending(\cP{w})$ transported from $\cP{w}$.  The two
exact-sequence conventions give
$o(\descending(\cP{v}))=[\mu,\dot\vartheta,\xi]=(-1)^d[\dot\vartheta,\mu,\xi]$
on $\R^{j(\cP{w})}_x\times C$, and with the outward normal $\dot\vartheta$
first the boundary orientation is $(-1)^d[\mu,\xi]$, that is,
$(-1)^{\dim\M_{h_N}(\cP{v},\cP{w})}$ times the product orientation.

For breaking in a trajectory space $\M_{h_N}(\cP{v},\cP{w}')$ at $\cP{w}$,
let $d'=\dim\M_{h_N}(\cP{w},\cP{w}')$ and let $[\mu']$ be an oriented basis
of its tangent space.  Near $\cP{w}$ the stable manifold
$\ascending(\cP{w}')$ is $C'\times\R^{n-j(\cP{w})}_y$ for a flow-invariant
cone $C'\subset\R^{j(\cP{w})}_x$ of dimension $d'+1$ with slice
$\M_{h_N}(\cP{w},\cP{w}')$, so
$\descending(\cP{v})\cap\ascending(\cP{w}')=C'\times C$ near $\cP{w}$, with
coordinates $x=r'm'$, $y=rm$.  The function $\rho=rr'=|x|\,|y|$ is flow
invariant and vanishes exactly on the broken stratum, so it is a collar
parameter and decreasing $\rho$ is outward.  Applying the first
convention to $\ascending(\cP{w}')$ and to $\ascending(\cP{w})$, with $[q]$
an oriented lift of $TS_N/T\ascending(\cP{w}')$, gives
$[\xi]=[\mu',\dot\vartheta',q]$ and
$o(C'\times C)\oplus[q]=o(\descending(\cP{v}))
=[\mu,\dot\vartheta_C,\mu',\dot\vartheta',q]$, where
$\dot\vartheta_C=-r\partial_r$ and $\dot\vartheta'=r'\partial_{r'}$ are the
flow directions of the incoming and outgoing trajectories; hence
$o(C'\times C)=[\mu,\dot\vartheta_C,\mu',\dot\vartheta']$.  The flow
direction of the glued trajectory is
$\dot\vartheta=\dot\vartheta_C+\dot\vartheta'$, and
$n_{\mathrm{out}}=\dot\vartheta_C-\dot\vartheta'$ satisfies
$d\rho(n_{\mathrm{out}})=-2\rho<0$, so it is outward; the pair
$(n_{\mathrm{out}},\dot\vartheta)$ has determinant $2>0$ in the basis
$(\dot\vartheta_C,\dot\vartheta')$.  Therefore
\[
 o(C'\times C)=(-1)^{d'}[\mu,\mu',\dot\vartheta_C,\dot\vartheta']
 =(-1)^{d'}[\mu,\mu',n_{\mathrm{out}},\dot\vartheta],
\]
and the second convention, moduli first, gives
$o(\M_{h_N}(\cP{v},\cP{w}'))=(-1)^{d'}[\mu,\mu',n_{\mathrm{out}}]
=(-1)^{d}[n_{\mathrm{out}},\mu,\mu']$: the boundary orientation is
$(-1)^{\dim\M_{h_N}(\cP{v},\cP{w})}$ times the product orientation.  For
$d=0$ both signs are $+1$, which is the product sign of rigid breaking.
\end{proof}

We now check the incidence signs explicitly.

\begin{proof}[Proof of Lemma~\ref{lem:floer-boundary-signs}]
Write $c=N-1-j(\cP{v})$ for the incidence codimension.  The ambient strip
space near a one-dimensional incidence has dimension $c+1$.
Let $n_{\mathrm{out}}$ be an outward neck vector, and write $s\in\{\pm1\}$
for the boundary sign, so that the oriented tangent of the incidence
curve is $s[n_{\mathrm{out}}]$.  The parameter estimates and exact normal
transport identify the limiting evaluation derivatives along the factors.

\emph{Incidence on the incoming factor.}
This includes outside breaking and the types \textup{(S4a)} and
\textup{(S4c)}.  The ambient incoming factor has dimension $c$, and the
outgoing factor is zero-dimensional, of sign $\epsilon_2$.
Let $[m_1]$ be an oriented incoming basis and let $\tilde s_1$ be the sign
of its rigid incidence.  The ambient boundary formula gives
\[
 o(\M)=(-1)^c\epsilon_2[n_{\mathrm{out}},m_1]
                =\epsilon_2[m_1,n_{\mathrm{out}}].
\]
The pulled-back normal orientation is $\tilde s_1[m_1]$.  The normal-first
convention therefore gives
$o(\M)=\tilde s_1s[m_1,n_{\mathrm{out}}]$, whence
$s=\tilde s_1\epsilon_2$.  A top-cell label on the outgoing strip has no
normal directions and leaves its ordinary sign unchanged.

\emph{Fixed incoming factor.}
For \textup{(S4b)}, the fixed incoming factor is zero-dimensional with
ordinary sign $\epsilon_1$, and the ambient outgoing factor has dimension
$c$.  If $[m_2]$ is its oriented basis and $\tilde s_2$ its incidence sign,
then
\[
 o(\M)=\epsilon_1[n_{\mathrm{out}},m_2],\qquad
 \Ev^*o(\nu(\cP{v}))=\tilde s_2[m_2].
\]
The evaluation is the outgoing evaluation, transported without error.
Normal-first ordering yields
$o(\M)=\tilde s_2s[m_2,n_{\mathrm{out}}]$, so
\[
 s=(-1)^c\epsilon_1\tilde s_2.
\]
This is the claimed product with the transport sign on the fixed factor.

\emph{Sphere breaking.}
Let $\vartheta$ be the incoming rigid sphere trajectory from $\cP{v}$
to $\cP{w}$, and let $\tilde s$ be the sign of the rigid incidence at
$\cP{w}$.  By Lemma~\ref{lem:floer-sphere-boundary}, for an outward normal
$n$ to the corresponding descending-cell sheet,
\[
 o(\descending(\cP{v}))=
     \epsilon(\vartheta)[n]\oplus o(\descending(\cP{w})).
\]
Compare the sphere orientations
$o(S_N)=o(\descending(\cP{v}))\oplus o(\nu(\cP{v}))
=o(\descending(\cP{w}))\oplus o(\nu(\cP{w}))$.  Moving $n$ past the
$j(\cP{w})$ descending directions gives
\[
 o(\nu(\cP{w}))=
 \epsilon(\vartheta)(-1)^{j(\cP{w})}[n]\oplus o(\nu(\cP{v})).
\]
Lift $n$ through the transverse evaluation slice to the outward tangent
of the incidence arc.  Its ambient orientation from the lower rigid
incidence is then
$\tilde s\epsilon(\vartheta)(-1)^{j(\cP{w})}
[n_{\mathrm{out}}]\oplus\Ev^*o(\nu(\cP{v}))$.
Normal-first ordering for the one-dimensional incidence gives instead
$s(-1)^c[n_{\mathrm{out}}]\oplus\Ev^*o(\nu(\cP{v}))$.
Since $j(\cP{w})+c=N-2$ is even,
$s=\tilde s\epsilon(\vartheta)$.  This is the product sign.
These are all incidence boundary types.
\end{proof}

\begin{proof}[Proof of Theorem~\ref{thm:floer-gluing-package}]
Proposition~\ref{thm:floer-compactification-cited} and
Proposition~\ref{prop:IIF-strata} give compactness, finiteness, and the
exhaustive boundary list.  Ordinary strip and fixed-strip breaks have
collars by Proposition~\ref{prop:IIF-ordinary-gluing}.  The incidence
outside breaks, fixed outgoing, fixed incoming, and both-non-fixed breaks
are covered by Proposition~\ref{thm:IIF-outside-gluing},
Corollary~\ref{cor:IIF-fixed-continuation-incidence},
Proposition~\ref{thm:IIF-fixed-incoming}, and
Proposition~\ref{thm:IIF-both-nonfixed}, respectively.  Sphere breaking is
Proposition~\ref{prop:IIF-sphere-breaking}, together with the ordinary
sphere collars.  Their injectivity and exhaustion identify a neighborhood
of every broken point with a half-interval.  The compactification is
therefore a compact one-manifold with precisely the stated boundary.
The finite strata give uniform parameters, and the orbitwise choices
make all maps equivariant.  The ordinary strip and sphere signs and the
incidence calculation above give the stated orientations.
\end{proof}

\newpage
\section*{Notation guide}\label{sec:notation-guide}
\noindent The sphere Morse complex is auxiliary to the Floer construction.
The prime is $p$, the Sobolev exponent is $q>2$, and $N=rm$ is the selected
real spectral dimension.  Strip indices refer to Fredholm indices of
components, not to an absolute grading of intersections.
\medskip
\begin{center}
\small
\begin{tabular}{@{}>{\raggedright\arraybackslash}p{0.23\textwidth}>{\raggedright\arraybackslash}p{0.48\textwidth}>{\raggedright\arraybackslash}p{0.23\textwidth}@{}}
\hline
Symbol & Meaning & Definition\\
\hline
$\sigma$, $\zeta$ & Generator of $G=\Z_p$ and $\zeta=e^{2\pi i/p}$; in Section~\ref{sec:applications}, $\tau$ is a point of $\Sym^{2m}(\C)$ as in \cite{seidel2010localization} & \S\ref{subsec:floer-setup}\\
invariant / outside & Intersection point in $M^G$ (stabilizer $G$) / in a free orbit (trivial stabilizer) & \S\ref{subsec:floer-setup}\\
fixed / non-fixed & Strip or configuration contained in $M^G$ / not contained in $M^G$; the \emph{fixed pair} is $(L_0^G,L_1^G)$ & \S\ref{subsec:floer-setup}\\
$J_G$ & Group complex structure of a representation without invariant vectors, and its complex orientation & Definition~\ref{def:group-complex-structure}\\
$M,L_i$; $\widehat M,\widehat L_i$ & Original pair; suspension and then adapted pair & \S\ref{subsec:floer-suspension}\\
$\nu$, $\nu_{L_i}$; $\R^m$, $W=T^*\R^m$ & Normal bundles of $M^G$ in $M$ and of $L_i^G$ in $L_i$; model normal representation and its cotangent lift, with standard Lagrangians $\R^m$ and $\sqrt{-1}\,\R^m$ & Assumption~\ref{ass:floer-normal-triviality}\\
$A^\perp$, $E_j$, $E_N^-$ & Normal asymptotic operator, ordered negative eigenspaces, and their selected sum & \eqref{eq:floer-finite-spectral-space}\\
$S_N,L_N$ & Spectral sphere and its quotient & Lemma~\ref{lem:floer-sphere-data}\\
$q_N,\Ev^-_{\cP p,N}$ & Weighted normalization and negative evaluation & Definition~\ref{def:floer-finite-spectral-evaluation}\\
$C_N,\Theta_N,\epsilon_N$ & Auxiliary sphere complex, fundamental cycle, and augmentation & Lemma~\ref{lem:floer-sphere-data}\\
$\cP v,j(\cP v),\vartheta$ & Sphere critical point, its Morse index, and a sphere trajectory & \S\ref{subsec:floer-transversality}\\
$o_{\cP x}$, $\epsilon(u)$ & Orientation line of an intersection point; sign of a rigid strip & \S\ref{subsec:floer-orientations}\\
$c(\cP v)=N-1-j(\cP v)$ & Incidence codimension of a sphere critical point; the transport sign is $(-1)^{c(\cP v)}$ & Convention~\ref{conv:floer-incidence-orientation}\\
$Q_i$, $E'$ & Relative Pin structures induced on the fixed pair and their background bundle & Lemma~\ref{lem:floer-fixed-strips}\\
$\M^d,\M_{M^G}^d$ & Unparametrized strip spaces of dimension $d$; fixed-strip spaces & \S\ref{subsec:floer-setup}\\
$\mathcal N^d_{\cP p,N}(\cP v;\cP b)$ & One-sided incidence spaces of dimension $d$ & \eqref{eq:floer-incidence-dimension}\\
$CF,HF,CF_{\mathrm{st}}$ & Ordinary Floer module, its homology, and stabilization & \S\ref{subsec:floer-stabilized}\\
$C^o,C^s,C^i$ & Outside, sphere, and retained invariant summands & \S\ref{subsec:floer-stabilized}\\
$\D\alpha\beta$ & Block $C^\alpha\to C^\beta$; columns index sources & \eqref{eq:floer-block-matrix}\\
$\ell,T$ & Cutoff width, fixed first; gluing length & \eqref{eq:IIF-separated-cutoffs}\\
$u_\pm,\kappa_\pm,\xi_\pm$ & Component strips, kernel parameters, and complementary corrections & \S\ref{subsec:IIF-linear}\\
$\cY_\tau$, $S_m$, $\chi$ & Fiber of the adjoint quotient $\chi$ over $\tau$ in the slice $S_m$ & \S\ref{subsec:matching-lagrangians}\\
$L_\frm$, $\Kh(K)$ & Matching Lagrangian of a crossingless matching $\frm$; symplectic Khovanov homology of a link $K$ & \S\ref{subsec:matching-lagrangians}\\
$h^{resc}_\gamma$ & Rescaled parallel transport along a path $\gamma$ of configurations & \S\ref{subsec:matching-lagrangians}\\
\end{tabular}
\end{center}

\bibliographystyle{alpha}
\bibliography{mybib}

\end{document}